\documentclass[11pt,a4paper]{article}
\usepackage[T1]{fontenc}
\usepackage{lmodern}
\usepackage{textcomp}
\usepackage{amsmath,amssymb,amsthm,mathtools}
\usepackage[textwidth=440pt,top=27mm,bottom=27mm]{geometry}
\usepackage{microtype}
\usepackage{xcolor}
\usepackage{tikz}
\usepackage{flafter,needspace}
\usepackage[colorlinks=true,linkcolor=blue!45!black,citecolor=blue!45!black,urlcolor=blue!45!black]{hyperref}
\hypersetup{pdftitle={A proof of the Karpelevic theorem via invariant polygons},pdfauthor={Brecht Verbeken and Vincent Ginis}}
\newcommand{\R}{\mathbb R}
\newcommand{\C}{\mathbb C}

\DeclareMathOperator{\conv}{conv}
\DeclareMathOperator{\aff}{aff}
\DeclareMathOperator{\relint}{relint}
\DeclareMathOperator{\ext}{ext}
\DeclareMathOperator{\Arg}{Arg}
\newtheorem{theorem}{Theorem}[section]
\newtheorem{lemma}[theorem]{Lemma}
\newtheorem{proposition}[theorem]{Proposition}
\newtheorem{corollary}[theorem]{Corollary}
\theoremstyle{definition}
\newtheorem{definition}[theorem]{Definition}

\theoremstyle{remark}

\numberwithin{equation}{section}
\begin{document}
\title{A proof of the Karpelevi\v{c} theorem via invariant polygons}
\author{%
  Brecht Verbeken\textsuperscript{1,2} and
  Vincent Ginis\textsuperscript{1,2,3}\\[0.8em]
  \parbox{0.96\textwidth}{\centering\small
    \textsuperscript{1}Department of Business Technology and Operations,\\
    Data Analytics Laboratory, Vrije Universiteit Brussel (VUB),\\
    Pleinlaan 2, 1050 Brussels, Belgium.\\[0.45em]
    \textsuperscript{2}imec-SMIT, Vrije Universiteit Brussel,\\
    Pleinlaan 9, 1050 Brussels, Belgium.\\[0.45em]
    \textsuperscript{3}School of Engineering and Applied Sciences,\\
    Harvard University, Cambridge, Massachusetts 02138, USA.\\[0.6em]
    \href{mailto:brecht.verbeken@vub.be}{brecht.verbeken@vub.be};
    \href{mailto:vincent.ginis@vub.be}{vincent.ginis@vub.be}
  }%
}
\date{}
\maketitle
\begin{abstract}
The Karpelevi\v{c} theorem describes the complex numbers that occur as
eigenvalues of stochastic matrices of a fixed order. We give a
self-contained proof based on invariant polygons. Their geometry yields
a product equation for an eigenvalue of maximal modulus at a fixed
argument. Convexity determines this modulus, explicit stochastic
matrices attain it, and a comparison between Farey intervals completes
the boundary description. As the matrix order increases, the resulting
equation gives a uniform relative asymptotic for the gap $1-|z|$ between
a boundary point $z$ and the unit circle, even arbitrarily close to arc
endpoints. We also determine the optimal contraction factor for a planar
rotation and dilation when vector size is measured by scaling a polygon
with at most $N$ vertices.
\end{abstract}


\section{The boundary theorem}
Let $\Theta_n$ be the set of individual eigenvalues of real row-stochastic matrices of order $n$. Let $F_n$ list the reduced fractions in $[0,1]$ with denominator at most $n$, in increasing order, and put $F_n^+=F_n\cap[0,1/2]$. We prove the following form of the Karpelevi\v{c} theorem \cite{Karp,Ito}.

\begin{theorem}[Karpelevi\v{c}, in Ito form]\label{thm:karpelevic}
Let $n\ge4$. Each pair of consecutive fractions $f<g$ in $F_n^+$
determines an upper boundary arc of $\Theta_n$, joining $e^{2\pi if}$
to $e^{2\pi ig}$.

Suppose first that $f$ has the smaller denominator, and write
$f=p/q$ and $g=r/s$. Then $q<s$ and $rq-ps=1$.
Set $m=\lfloor n/q\rfloor$. The arc is the branch of
\begin{equation}\label{eq:main-ito}
 z^s(z^q-\beta)^m=(1-\beta)^m z^{mq},
 \qquad 0\le\beta\le1,
\end{equation}
joining $e^{2\pi ip/q}$ to $e^{2\pi ir/s}$.

Fix an angle $\vartheta$ with $2\pi p/q<\vartheta<2\pi r/s$, and put
\[
 A=q\vartheta-2\pi p,\qquad
 B=\frac{2\pi r-s\vartheta}{m}.
\]
The boundary point at argument $\vartheta$ is $z=\rho e^{i\vartheta}$,
where $\rho$ is the unique solution in $(0,1)$ of
\[
 \rho^{s/m}\sin A+\rho^q\sin B=\sin(A+B).
\]

If $g$ has the smaller denominator, apply this description to the
arc reflected in the real axis: put $(p/q,r/s)=(1-g,1-f)$ and
replace $z$ by $\bar z$. The angle $\vartheta$ is then the argument
of $\bar z$, and $\rho=|\bar z|=|z|$.

These arcs, ordered by increasing argument, form the upper boundary
of $\Theta_n$; reflection in the real axis gives the lower boundary.
The region consists of all line segments joining zero to a boundary point.

The points of $\Theta_n$ on the unit circle are exactly the roots of
unity of orders at most $n$. For the small orders,
\[
 \Theta_1=\{1\},\qquad \Theta_2=[-1,1],\qquad
 \Theta_3=\conv\{1,e^{2\pi i/3},e^{-2\pi i/3}\}\cup[-1,-1/2].
\]
For $n=3$, the upper boundary ends with the vertical segment from
$e^{2\pi i/3}$ to $-1/2$, followed by the real segment $[-1,-1/2]$.
\end{theorem}

\subsection*{The intermediate product}

The geometric construction leads to the product relation in
Theorem~\ref{thm:product}. Let $z$ be a nonreal eigenvalue with
$|z|<1$, and let $N\ge4$ be the smallest order of a stochastic matrix
having $z$ as an eigenvalue. Suppose $z$ has maximal modulus at its
argument in $\Theta_N$. Write $z=\rho e^{i\vartheta}$ with $0<\vartheta<2\pi$.
If necessary, apply the following description to $\bar z$ instead,
using $\vartheta=\arg\bar z$.

A suitably chosen invariant polygon yields consecutive fractions
$p/q<\vartheta/(2\pi)<r/s$ in $F_N$, with $q<s$, and coefficients
$0\le\beta_j<1$. Put $m=\lfloor N/q\rfloor$ and $e=s-mq$. Then
\[
 z^e\prod_{j=1}^{m}(z^q-\beta_j)
       =\prod_{j=1}^{m}(1-\beta_j),
 \qquad
 e\vartheta+\sum_{j=1}^{m}u_j=2\pi(r-mp).
\]
Here $u_j$ is the argument of $z^q-\beta_j$ in $[A,M)$, where
$A=q\vartheta-2\pi p$ and $M=\Arg(z^q-1)\in(A,\pi)$.
If $e<0$, the first identity is understood for $z\ne0$, or after
clearing the negative power.

The product equation alone determines the total argument only modulo
$2\pi$. The geometric construction tracks the factor arguments and
the integer $r-mp$, establishing the second identity as an equality
of real numbers. For fixed $\vartheta$, this determines the sum of
the $u_j$ and hence their average. Taking logarithms of absolute
values in the first identity gives a sum of a convex function of
these arguments. A convexity comparison then bounds $\rho$, with
equality exactly when $\beta_1=\cdots=\beta_m$.
The geometric proof only needs one polygon with these identities.
A classification of all invariant polygons is not necessary.

\subsection*{Relation to earlier work}

Dmitriev and Dynkin obtained early restrictions on stochastic eigenvalues and determined the small orders \cite{DD}; Karpelevi\v{c} described the region in every order \cite{Karp}. \DJ okovi\'c later simplified the statement of Karpelevi\v{c}'s theorem using cyclic polygons generated by powers of a complex number \cite{Djokovic}. Ito's polynomial formulation \cite{Ito} makes the boundary description especially concise. His argument establishes its equivalence with Karpelevi\v{c}'s theorem. A direct proof must instead explain why these polynomials describe the boundary, starting from the stochastic eigenvalue problem itself.

Realizing polynomial roots as stochastic eigenvalues, selecting continuous arcs, and proving maximal modulus at each argument are separate tasks. Johnson and Paparella construct stochastic matrices realizing the roots of the relevant Ito polynomials \cite[Theorem~3.2]{JP}. Munger, Nickerson, and Paparella identify these three tasks for a direct proof, with the arcs joining the appropriate roots of unity over Farey intervals \cite[Proposition~1.2]{MNP}. Kirkland, Laffey, and \v{S}migoc give a trigonometric characterization of the unique boundary point at a fixed argument, using Karpelevi\v{c}'s boundary theorem \cite[Theorems~1.1 and~1.2]{KirklandLaffeySmigoc}. In our proof, the product relation is derived from an arbitrary eigenvalue of maximal modulus at a fixed argument. Convexity and an independent comparison between Farey intervals then determine its modulus. The same trigonometric equation is obtained within the proof and later used for the asymptotic estimates.

The structure of realizing matrices has also been studied beyond the existence question. Kirkland and \v{S}migoc classify realizations of the Type~0 and Type~I reduced Ito polynomials and the sparsest realizations of Types~II and~III \cite{KirklandSmigoc}. The later preprints \cite{VGTypeII,VGTypeIII} treat the non-degenerate full-degree Type~II case and the Type~III realization conjecture over its stated nonzero parameter range. These results concern matrices realizing fixed boundary polynomials. Theorem~\ref{thm:independent-realization} uses the Johnson--Paparella graph and allows the weights $\beta_j$ on edges closing the local cycles to vary independently, with complementary weights $1-\beta_j$ on the connecting edges. A cycle expansion gives the full characteristic polynomial. Appendix~\ref{app:matrix-family} records the graph's structure and uses the convexity comparison to show that, among roots satisfying the argument identity above, the boundary modulus is attained precisely when all $\beta_j$ are equal. The polynomial family also has roots that do not satisfy this identity, so the condition is essential.

The Karpelevi\v{c} region provides a reference for eigenvalue regions of stochastic matrices with additional constraints. Kirkland determined the region of non-Perron eigenvalues of stochastic Leslie matrices in every order \cite{KirklandLeslie}. For matrices with nonzero entries only on a directed cycle and the diagonal, Ran and Teng formulated a conjecture in order four \cite[Conjecture~20]{RanTeng}; the four-cycle region and its extension to every order are treated in \cite{VagenendeFourCycle,VGNCycle}. Work on monotone stochastic matrices determines the region through order three, constructs realizing matrices, and gives a reduction valid in every order \cite{VagenendeMonotone}. Star-convexity for stochastic eigenvalue regions and several subclasses is studied in \cite{VagenendeStar}. All these studies concern the eigenvalues that can occur individually within a matrix family.

Further extensions to constrained stochastic families would require new arguments. The doubly stochastic eigenvalue problem belongs to the Perfect--Mirsky programme \cite{PerfectMirsky}. The proposed region is correct through order four \cite{LevickPereiraKribs}, but fails in order five \cite{HarlevJohnsonLim}. Doubly stochastic matrices, monotone matrices of higher order, and matrices with sparse support are natural settings for studying product equations together with a specified total argument. We do not classify their boundaries here.

\Needspace{9\baselineskip}
\subsection*{Proof strategy and consequences}

The proof is self-contained apart from elementary linear algebra, the Perron--Frobenius theorem, and basic convexity and calculus. In Section~6, the segment comparison applies the one-variable fundamental theorem of calculus along a complex line segment to the polynomial $z^m$. Background references are Berman and Plemmons \cite{BermanPlemmons}, Schneider \cite{Schneider}, Coxeter \cite{Coxeter}, and Hardy and Wright \cite{HardyWright}. We prove the invariant-polygon correspondence, the deflation and projective deformation arguments, the arithmetic of return paths, and the stochastic realization. No part of the Karpelevi\v{c} boundary theorem is assumed.

The invariant-polygon correspondence and boundary-contact principle of Dmitriev and Dynkin \cite{DD}, also developed in Swift's account \cite{Swift}, reduce the eigenvalue problem to a rotation followed by a contraction $T$. For an eigenvalue of maximal modulus at a fixed argument, take a realizing invariant polygon with the fewest vertices. Every vertex image under $T$ must lie on its boundary. An interior vertex image would allow a stochastic realization with a positive row. Subtracting a small multiple of the stationary row from that row lowers the Perron root without changing the other eigenvalues. Restoring stochasticity then increases the chosen eigenvalue's modulus without changing its argument, contradicting maximality. We give the stationary-vector proof of this classical rank-one deflation principle \cite{Brauer,Goddard}.

The goal is to arrange the polygon's contacts so that they yield the cyclic product. We follow Karpelevi\v{c}'s strategy of vertex replacement, cyclic index shifts, and projective deformation \cite{Karp}. After fixing the polygon's scale, we first minimize the number of vertex images lying strictly between the endpoints of a side, and then minimize area among polygons attaining that count. The permitted vertex replacements preserve the number of vertices and cannot increase this count.

One replacement uses a factorization $A=HS$ of a realizing stochastic matrix, with $H,S$ stochastic and $S$ invertible. Reversing their order gives $SH=SAS^{-1}$, so the new matrix is stochastic and has the same spectrum. We choose $S$ so that applying it to the eigenvector replaces one polygon vertex by a point between that vertex and its predecessor on the boundary. These replacements and minimality show that the sides containing interior contacts form one uninterrupted run around the boundary, with each side sharing an endpoint with the next.

Choose a direction of traversal around the polygon. For each side in this run, call the endpoint reached second its base vertex. Starting from a base $w$, follow the successive images $w,Tw,T^2w,\ldots$---its orbit---through polygon vertices until the first image lies inside one of these sides. This is a return path: it returns to the chosen run of sides, not necessarily to $w$. The cyclic index shift determines its length and target side. The remaining issue is whether a first return can skip an intervening base.

To rule this out, suppose such skipping occurs. A point inside a side cannot later become a vertex under $T$. Together with the contact order, this gives supporting lines that touch the polygon at just one vertex. We move one base slightly and determine successive moved vertices by intersecting their supporting lines with lines through the preceding moved vertex and a fixed contact point. This keeps the corresponding side contacts in place; the remaining vertex motions are carried along their $T$-images to preserve the vertex-to-vertex relations.

The closing comparison is proved in projective coordinates whose denominators are positive on the polygon, so segments are preserved. If $r_i$ describes the position of the projected point on its line, its reciprocal $x_i=1/r_i$ satisfies
\[
 x_i=a_i x_{i-1}+b_i,\qquad a_i,b_i>0.
\]
Starting from $x_1>0$, this keeps every $x_i$ positive and hence every $r_i$ finite in this chart. The final comparison selects a motion that puts one vertex image strictly inside the polygon while preserving every other vertex and side contact. The deformed polygon still has the same number of vertices and remains invariant under the unchanged map $T$. But the preceding boundary-contact result applies to every such polygon and requires every vertex image to lie on its boundary. The skipped return would therefore produce an impossible polygon.

Following each return path gives an equation linking its base to the endpoints of the side reached. Multiplying these equations around the cycle eliminates the vertex coordinates and gives the product. Karpelevi\v{c} already obtained a cyclic recurrence with coefficients depending on the side \cite[Theorem~VI and equation~(13)]{Karp}, but made the coefficients equal geometrically before eliminating the vertices. We retain the unequal coefficients and also track the sum of the factor arguments as a real number. The product alone determines this sum only modulo $2\pi$; the additional argument identity selects the boundary branch.

Dividing each factor by its positive coefficient $1-\beta_j$ places all normalized factors on one straight line. Along this line, the logarithm of the modulus is a strictly convex function of the argument, as in the restricted $n$-cycle family \cite[Lemma~2.2 and Proposition~3.1]{VGNCycle}. The argument identity fixes the average of these arguments. Jensen's inequality then bounds the eigenvalue's modulus, with equality precisely when all coefficients agree. Expanding the characteristic polynomial of the Johnson--Paparella matrix in terms of its directed cycles gives an eigenvalue attaining this bound.

The geometric argument uses the eigenvalue's least realizing order, so we must also compare the resulting bound with the candidate for the required matrix order. As the order increases, a Farey interval either remains unchanged or is split by inserting $(a+c)/(b+d)$ between its endpoints $a/b<c/d$. The inserted fraction is formed by adding the two numerators and the two denominators separately. On an unchanged interval, convexity handles any increase in the number of product factors. On a split interval, one integral identity for $z^m$ along a complex line segment proves that the candidate modulus increases on both new open intervals. Together, these comparisons complete the passage to the required order.

The boundary theorem also determines how well a polygon can measure contraction under a rotation and dilation. Let $R_\theta$ denote rotation through $\theta$ and let $\rho>0$. Among polygons $P$ with at most $N\ge4$ vertices and $0\in\operatorname{int}P$, the smallest possible factor $\gamma$ satisfying
\[
 \rho R_\theta P\subseteq\gamma P
\]
is $\rho/K_N(\theta)$, where $K_N(\theta)$ is the maximal modulus of an order-$N$ stochastic eigenvalue at argument $\theta$. Measure a vector's size by the least scaling of $P$ that contains it. The inclusion then says that its image has size at most $\gamma$ times its original size. Section~\ref{sec:gauges} proves that the minimum is attained and that its largest value over all angles is $\rho\sec(\pi/N)$.

Section~\ref{sec:asymptotic} uses the resulting equation to approximate the gap $1-K_N(\theta)$ as $N$ increases. The relative error bound remains valid arbitrarily close to either endpoint of a Farey arc. These estimates determine how many vertices are needed for a given contraction accuracy, both at a fixed badly approximable angle and uniformly over all angles. Here badly approximable means that, for $x=\theta/(2\pi)$, some $c>0$ satisfies $|x-a/b|\ge c/b^2$ for every rational $a/b$.

The appendices develop the preceding use of Jensen's inequality. For the normalized product factors, the logarithm of the modulus is a strictly convex function of the argument. The product determines the sum of these logarithms, while the argument identity determines the average argument. Jensen's inequality then gives the bound on the eigenvalue's modulus, with equality precisely when all factor arguments, and hence all coefficients $\beta_j$, agree. Appendix~\ref{app:comparisons} quantifies this equality condition: for fixed interior Farey data, the coefficients approach equality as the eigenvalue's modulus approaches the boundary value. It also treats factors with different exponents by optimizing their arguments within the allowed ranges.

Appendix~\ref{app:matrix-family} studies the realizing matrix family with independently varying parameters $\beta_j$. It describes connectivity of the transition graph, positivity of sufficiently high matrix powers, the number of nonzero entries, and the endpoint matrices. The equality condition above identifies the equal-parameter member attaining the boundary modulus among roots satisfying the argument identity. These appendix results are not needed for the boundary proof.


\Needspace{8\baselineskip}
\section{Saturation and cyclic contacts}
For $n\ge1$, let $\Theta_n$ be the set of eigenvalues of real row-stochastic matrices of order $n$. Every such matrix has operator norm one on $\ell^\infty$, so $\Theta_n$ lies in the closed unit disk. It is compact: from any convergent sequence of these eigenvalues, compactness of the matrix set gives a convergent subsequence of corresponding matrices, and continuity of the determinant makes the limit an eigenvalue of the limiting matrix. Since the matrices are real, $\Theta_n$ is invariant under complex conjugation. Adjoining a $1\times1$ identity block preserves all eigenvalues, giving $\Theta_n\subseteq\Theta_{n+1}$.

For $n\ge2$, the invariant-polygon correspondence of Dmitriev and Dynkin \cite[\S\S1--2]{DD}, presented in English by Swift \cite{Swift}, states that
\begin{equation}\label{eq:polygon-criterion}
 \lambda\in\Theta_n\quad\Longleftrightarrow\quad
 \begin{gathered}
 \lambda P\subseteq P\text{ for some non-singleton polytope }P\subset\C\\
 \text{with at most }n\text{ vertices}.
 \end{gathered}
\end{equation}
Here a polytope is the convex hull of finitely many points; it may be a segment.

To prove the forward implication, choose $Av=\lambda v$ with $v\ne0$ and set $P=\conv\{v_1,\ldots,v_n\}$. Each row of $A$ expresses $\lambda v_i$ as a convex combination of the coordinates of $v$, so $\lambda P\subseteq P$. If $P$ is a singleton, all coordinates are the same nonzero number, forcing $\lambda=1$; any nontrivial segment then gives the required $P$. Conversely, write each image of a vertex of $P$ as a convex combination of its vertices. These coefficients form the rows of a stochastic matrix with eigenvalue $\lambda$. If $P$ has fewer than $n$ vertices, adjoining an identity block gives order $n$.

For nonreal $\lambda$, $P$ cannot be a segment, because $\lambda P$ would be a nonparallel segment and could not lie in $P$. Thus $P$ is a polygon. Moreover, if $0<|\lambda|\le1$, then $0\in\operatorname{int}P$. When $|\lambda|<1$, the iterates $\lambda^k w$ of any $w\in P$ converge to zero. When $|\lambda|=1$, their averages converge to zero. In both cases, invariance, convexity, and closedness give $0\in P$.

If zero were on the boundary, a supporting line through it would put all iterates of a nonzero $w\in P$ in one closed half-plane. Their directions are obtained by successive rotations through $\arg\lambda$. For a rotation of finite order, the rotated unit vectors sum to zero and are not collinear; for a rotation of infinite order, their directions are dense in the circle. Neither set of directions can lie in that half-plane, so zero is interior.

The sets $\Theta_n$, $n\ge2$, are star-shaped about zero: if $\lambda\in\Theta_n$, then $t\lambda\in\Theta_n$ for every $0\le t\le1$. To see this for $\lambda\ne1$, choose $Av=\lambda v$ and a stationary probability row vector $\pi^T$. Since $\pi^TA=\pi^T$, we have $(\lambda-1)\pi^Tv=0$. Thus $\pi^Tv=0$, and the stochastic matrix
\[
 A_t=tA+(1-t)\mathbf1\pi^T
\]
satisfies $A_tv=t\lambda v$; here $\mathbf1$ is the all-ones column vector. For $\lambda=1$, the matrix with rows $(1,0)$ and $(1-t,t)$ has eigenvalue $t$ and can be padded to order $n$. Write $R_n(\theta)$ for the maximum modulus at argument $\theta$ in $\Theta_n$.

For the geometric reduction, fix a nonreal number $\zeta=\rho e^{i\theta_\zeta}$, $0<\rho<1$, satisfying
\begin{equation}\label{eq:extremal}
 \begin{gathered}
 \min\{|\ext P|:\zeta P\subseteq P,\ \operatorname{int}P\ne\varnothing\}=N,\\
 \text{for every }t>1,\quad t\zeta\text{ has no invariant polygon}\\
 \text{with at most }N\text{ vertices}.
 \end{gathered}
\end{equation}
Here $P$ ranges over polygons and $\ext P$ denotes its vertices. The first condition makes $N$ the smallest number of vertices needed for invariance under multiplication by $\zeta$. The second says that increasing its modulus while keeping its argument fixed requires more than $N$ vertices. Write $Tz=\zeta z$. These conditions hold for a nonreal eigenvalue of maximal modulus at a fixed argument, with $N$ its least realizing matrix order. Any real planar map with nonreal eigenvalues becomes multiplication by either eigenvalue after a suitable invertible real change of coordinates.

\paragraph{Notation and orientation.}
The boundary point $z$ lies in the upper half-plane, with $\theta=\arg z\in(0,\pi)$. The geometric proof uses $\zeta\in\{z,\bar z\}$, with argument $\theta_\zeta\in(0,2\pi)$. Its choice allows one image vertex per side when each side includes its ending vertex and excludes its starting vertex. The product theorem may conjugate again: it uses $\omega\in\{\zeta,\overline\zeta\}$, with argument $\vartheta$, chosen so that the left Farey endpoint has the smaller denominator. Thus both $\zeta$ and $\omega$ denote either $z$ or $\bar z$, and all three have modulus $\rho$. Conjugation sends an argument $\alpha$ to $2\pi-\alpha$, reverses the cyclic vertex order, and reflects an ordered Farey pair $(f,g)$ to $(1-g,1-f)$. Section~5 returns to the original upper-half-plane point $z$ and specifies which of $z$ and $\bar z$ is used in the product.

The integers used at the successive stages are summarized below. Some symbols are reused; their values are specified at each stage and in each return case.

The vertex indices are cyclic: $v_{i+N}=v_i$. A cyclic shift adds the same integer to every index, with indices taken modulo $N$. The contact shift $\kappa$ records which side contains each vertex image:
\[
 Tv_i\in(v_{i+\kappa-1},v_{i+\kappa}].
\]
Thus the image of $v_i$ lies on the side ending at $v_{i+\kappa}$, with the same $\kappa$ for every vertex. The image may be inside the side or at its ending vertex. For example, if $N=5$ and $\kappa=2$, then $Tv_4\in(v_5,v_1]$: the index $4+2=6$ wraps around to $1$.
\begin{center}
\begin{tabular}{@{}p{0.23\linewidth}p{0.71\linewidth}@{}}
$N$ & Least realizing matrix order, also the number of polygon vertices.\\[3pt]
$\kappa$ & Fixed shift from a source vertex index to the index of the side containing its image, where a side is indexed by its ending vertex.\\[3pt]
$\delta$ & Greatest common divisor $\gcd(N,\kappa)$.\\[3pt]
$\varphi$ & Number of vertex images strictly inside sides, after these sides form one uninterrupted boundary run.\\[3pt]
$\nu,q,h,\Delta$ & Number of bases $\nu$, return times $q$ and $q+h$, and shift $\Delta$ of the base index on return, with base indices taken modulo $\nu$. For the polygon's return paths, $\nu=\varphi$; the product construction may use more bases.\\[3pt]
$\ell$ & Parameter of the convex chain in the projective lemma, whose vertices are $X_0,\ldots,X_{\ell+1}$.\\[3pt]
$q,s,m,e$ & Farey endpoint denominators $q<s$, factor count $m=\lfloor N/q\rfloor$, and exponent $e=s-mq$ in the product.
\end{tabular}
\end{center}

\subsection{Perron deflation forces boundary contact}
The boundary-contact argument rests on a rank-one eigenvalue shift due to Brauer \cite{Brauer} and developed further by Goddard \cite{Goddard}. Let $A$ be an irreducible stochastic matrix of order $n$ with stationary probability row vector $\pi^T$, and suppose every entry of row $i$ is positive. Subtract $\varepsilon\pi^T$ from row $i$, leaving the other rows unchanged. The resulting matrix is
\[
 B=A-\varepsilon e_i\pi^T,
\]
where $e_i$ is the $i$th standard basis column vector and $\varepsilon>0$ is chosen small enough that row $i$ remains positive. This lowers the Perron root from $1$ to $r=1-\varepsilon\pi_i\in(0,1)$ and preserves all other eigenvalues. Choose a positive eigenvector $h=(h_1,\ldots,h_n)^T$ of $B$, with $h_j>0$ and $Bh=rh$, and set
\[
 D=\operatorname{diag}(h_1,\ldots,h_n),\qquad
 C=r^{-1}D^{-1}BD.
\]
The coordinate change $x=Dy$ rescales coordinate $j$ by $h_j$ and replaces $B$ by $D^{-1}BD$, preserving its eigenvalues. Dividing by $r$ gives the stochastic matrix $C$, whose non-Perron eigenvalues have larger modulus whenever they are nonzero.

The boundary-contact conclusions go back to Dmitriev and Dynkin \cite[\S\S1--2]{DD}. We prove them for every polygon with $N$ vertices invariant under multiplication by the fixed number $\zeta$. The extremality of $\zeta$ means that $N$ is the smallest vertex count for which such a polygon exists, and no $t\zeta$ with $t>1$ has an invariant polygon with at most $N$ vertices. The conclusion remains applicable after replacing vertices while preserving invariance and their number.

\begin{lemma}\label{lem:perron-deflation}
Let $A$ be an irreducible row-stochastic matrix with stationary probability row vector $\pi^T$. Suppose $u\ge0$,
\[
 B=A-u\pi^T\quad\text{is nonnegative and irreducible},
 \qquad r=1-\pi^Tu\in(0,1).
\]
Then $r$ is the Perron root of $B$. There is a diagonal matrix $D=\operatorname{diag}(h_1,\ldots,h_n)$ with $h_i>0$ such that
\[
 C=r^{-1}D^{-1}BD
\]
is row-stochastic. The change of coordinates $x=Dy$ replaces $B$ by $D^{-1}BD$ without changing its eigenvalues. Thus $C$ has Perron eigenvalue $1$, and all remaining eigenvalues are those of $A$ divided by $r$, with algebraic multiplicities preserved. If $A$ has a row with every entry positive, $u$ can be chosen to satisfy these hypotheses.
\end{lemma}
\begin{proof}
Since $A$ is irreducible, its stationary probability vector satisfies $\pi>0$. We have
\[
 \pi^TB=\pi^TA-(\pi^Tu)\pi^T=r\pi^T,
\]
so the Perron--Frobenius theorem identifies $r$ as the Perron root of $B$. Choose $Bh=rh$ with $h>0$ and put $D=\operatorname{diag}(h_1,\ldots,h_n)$. Then $C=r^{-1}D^{-1}BD$ is nonnegative and
\[
 C\mathbf1=r^{-1}D^{-1}Bh=D^{-1}h=\mathbf1,
\]
so $C$ is row-stochastic.

The subspace $H=\{v:\pi^Tv=0\}$ is invariant under both $A$ and $B$, and $Bv=Av$ for $v\in H$. Since $\pi^T\mathbf1=1$, a basis of $H$ together with $\mathbf1$ gives a basis of the full space. In this basis, $A$ and $B$ are block triangular with the same $(n-1)\times(n-1)$ block on $H$ and final diagonal entries $1$ and $r$, respectively. Thus $B$ has the non-Perron eigenvalues of $A$, including algebraic multiplicities. Similarity by $D$ preserves these eigenvalues, and division by $r$ proves the assertion for $C$. A non-Perron eigenvector $v$ of $A$ becomes $D^{-1}v$ for $C$.

If row $i$ of $A=(a_{ij})$ is strictly positive, choose
\[
 u=\varepsilon e_i,\qquad
 0<\varepsilon<\min\left\{1,\min_j\frac{a_{ij}}{\pi_j}\right\},
\]
where $e_i$ is the $i$th standard basis column vector. The modified row has entries $a_{ij}-\varepsilon\pi_j>0$, and the other rows are unchanged. Hence $B$ has the same pattern of positive entries as $A$ and remains irreducible. Also $r=1-\varepsilon\pi_i\in(0,1)$.
\end{proof}

\begin{theorem}\label{thm:saturation}
Assume \eqref{eq:extremal}. If $P$ is a polygon with at most $N$ vertices and $\zeta P\subseteq P$, then $P$ has exactly $N$ vertices. For every vertex $v$ of $P$, $\zeta v\in\partial P$, and every side of $P$ intersects $\zeta P$.
\end{theorem}
\begin{proof}
Minimality gives exactly $N$ vertices, which we denote by $v_1,\ldots,v_N$. Since $\zeta$ is nonreal, $0\in\operatorname{int}P$, so every vertex is nonzero. Expressing each $\zeta v_i$ as a convex combination of the vertices gives a stochastic matrix $A=(a_{ij})$ satisfying $Av=\zeta v$, where $v=(v_1,\ldots,v_N)^T$.

Every such $A$ is irreducible. Otherwise, some nonempty proper index set $I$ would satisfy $a_{ij}=0$ for $i\in I$, $j\notin I$. The submatrix $A_{II}$ would then be stochastic and satisfy
\[
 A_{II}(v_i)_{i\in I}=\zeta(v_i)_{i\in I}.
\]
This eigenvector is nonzero, giving a realization of $\zeta$ of order $|I|<N$, contrary to minimality.

Suppose $\zeta v_i$ lies in the interior of $P$. Choose $0<\varepsilon<1/N$ small enough that
\[
 w=\frac{\zeta v_i-\varepsilon\sum_{j=1}^N v_j}{1-N\varepsilon}
 \in P.
\]
Such a choice is possible because $w\to\zeta v_i$ as $\varepsilon\to0$. Express $w$ as a convex combination of the vertices. The resulting expression for $\zeta v_i$ gives every vertex weight at least $\varepsilon$, so row $i$ of $A$ can be chosen strictly positive. Lemma~\ref{lem:perron-deflation} then gives a stochastic matrix of order $N$ with eigenvalue $\zeta/r$, where $0<r<1$. By \eqref{eq:polygon-criterion}, there is a polygon with at most $N$ vertices invariant under multiplication by $\zeta/r$. Since $|\zeta/r|>|\zeta|$, this contradicts extremality. Thus every vertex image lies on the boundary.

To prove contact with every side, form the polar polygon
\[
 P^\circ=\{a\in\R^2:\langle a,x\rangle\le1
                    \text{ for every }x\in P\},
\]
where $\langle\cdot,\cdot\rangle$ is the Euclidean inner product. Since $0\in\operatorname{int}P$, its vertices correspond to the sides of $P$: a polar vertex $a$ corresponds to the side defined by $\langle a,x\rangle=1$.

The adjoint $T^*$ is defined by $\langle T^*a,x\rangle=\langle a,Tx\rangle$ and is multiplication by $\bar\zeta$. For $a\in P^\circ$ and $x\in P$, invariance gives $Tx\in P$ and hence $\langle T^*a,x\rangle\le1$. Thus $T^*P^\circ\subseteq P^\circ$. Complex conjugation preserves the extremality conditions, so the vertex-contact conclusion just proved applies to $P^\circ$. For each polar vertex $a$, it gives $T^*a\in\partial P^\circ$. Therefore some $x\in P$ satisfies
\[
 \langle a,Tx\rangle=\langle T^*a,x\rangle=1.
\]
The point $Tx\in TP$ lies on the side corresponding to $a$. Every side of $P$ therefore meets $TP$.
\end{proof}

We will use an elementary fact about faces. Suppose a convex combination with strictly positive weights lies on the boundary of a polytope, taken within its affine hull. A supporting linear functional attains its maximum at this combination. Since each point used has functional value at most that maximum, equality forces every point to attain it. Thus all points used lie in one common proper face. In particular, a boundary point of a polygon can use at most two vertices with positive weights in a convex combination, and if two are used, they are adjacent.

\subsection{Counting vertex images}
The least realizing order fixes the number of vertices. To control vertex replacement, we also count how many vertex images are not vertices. For an injective linear map $T$ and a polytope $P$ with $TP\subseteq P$ and vertex set $V$, define
\[
 d_T(P)=|\{v\in V:Tv\in V\}|,\qquad
 b_T(P)=|V|-d_T(P).
\]
Thus $d_T(P)$ counts vertices mapped to vertices, and $b_T(P)$ counts the remaining vertices. In the polygon setting just proved, every image lies on the boundary, so $b_T(P)$ counts the images strictly inside sides. The following theorem gives conditions under which replacing vertices cannot increase this count.

\begin{theorem}\label{thm:branching}
In any dimension, let $T$ be an injective linear map and let $P,Q$ be polytopes satisfying
\[
 TP\subseteq Q\subseteq P,\qquad
 |\ext Q|=|\ext P|,\qquad
 \ext Q\subseteq\ext P\cup T\ext P.
\]
Then $TQ\subseteq Q$ and $b_T(Q)\le b_T(P)$.
\end{theorem}
\begin{proof}
Since $Q\subseteq P$, we have $TQ\subseteq TP\subseteq Q$, so $Q$ is invariant. Write $V=\ext P$ and $W=\ext Q$. Let
\[
 R=V\setminus W,\qquad U=W\setminus V
\]
be the removed and new vertices. Equal vertex counts give $|R|=|U|$. Put
\[
 E=\{v\in V:Tv\in V\},
\]
so $|E|=d_T(P)$.

An old vertex lying in $Q$ remains a vertex of $Q$: a nontrivial convex decomposition within $Q$ would also be one within $P$. Since $TV\subseteq Q$, no removed vertex belongs to $TV$. Thus $R\cap TV=\varnothing$, and each vertex in $E\setminus R$ still maps to a vertex of $Q$.

Every new vertex belongs to $TV$. By injectivity, its preimage is a unique vertex of $P$. Therefore
\[
 S=T^{-1}U\subseteq V,\qquad |S|=|U|=|R|.
\]
Also $S\cap E=\varnothing$, since vertices in $S$ map to new vertices rather than old ones. Each retained vertex in $S\setminus R$ maps to a vertex in $U$ and is therefore also counted by $d_T(Q)$. Consequently,
\[
 d_T(Q)\ge |E\setminus R|+|S\setminus R|
 =|E|+|R|-|E\cap R|-|S\cap R|
 \ge |E|.
\]
The last inequality follows because $E\cap R$ and $S\cap R$ are disjoint subsets of $R$. Since $P$ and $Q$ have equally many vertices, this gives $b_T(Q)\le b_T(P)$.
\end{proof}

We now return to $Tz=\zeta z$ and the invariant $N$-gons in \eqref{eq:extremal}. Fix their scale by requiring
\[
 \max_{x\in P}|x|=1.
\]
Since $T$ is linear, scaling about zero preserves invariance and the vertex count. The normalized class is compact. Indeed, from any sequence of these polygons, choose a subsequence along which all $N$ vertex coordinates converge in the closed unit disk. Their limiting convex hull remains invariant, contains zero, and has maximum modulus one. The last two properties exclude a single point. Nonreality of $\zeta$ excludes a segment, and minimality of $N$ excludes a polygon with fewer vertices. Thus all $N$ limiting points are distinct vertices, and the limit belongs to the same normalized class.

Throughout this section, choose a normalized polygon $P$ by first minimizing $b_T(P)$ and then minimizing area among the polygons attaining that count. Such a choice exists. Any equality $Tv_i=v_j$ that holds throughout a convergent sequence of vertex coordinates also holds at the limit; additional equalities may appear. Consequently, a limit of polygons minimizing $b_T$ still minimizes $b_T$. These minimizers form a closed subset of the compact normalized class, and continuity of area gives an area minimizer among them. Scaling changes neither $d_T$ nor $b_T$, so the chosen $P$ also minimizes $b_T$ among all invariant $N$-gons, without fixing their scale.

The two minimization conditions have separate uses. Suppose a polytope $Q$ satisfies
\[
 TP\subseteq Q\subseteq P,\qquad
 |\ext Q|=N,\qquad
 \ext Q\subseteq\ext P\cup T\ext P.
\]
Theorem~\ref{thm:branching} gives $b_T(Q)\le b_T(P)$. Since $P$ already minimizes this count, equality follows. If $Q$ also contains a vertex of $P$ with modulus one, it is normalized. A proper inclusion $Q\subsetneq P$ would then give smaller area with the same minimum count, contradicting the choice of $P$. Thus such a $Q$ must equal $P$. The later vertex replacements need only to preserve the minimum of $b_T$.

\subsection{Assigning vertex images to sides}
Write $[a,b]_{\partial P}$ for the boundary arc from $a$ to $b$ in the counterclockwise (positive) direction; parentheses indicate excluded endpoints. The half-open sides $(v_{i-1},v_i]$ partition the boundary: each includes its ending vertex and excludes its starting vertex. Our goal is to arrange that each contains exactly one vertex of $TP$. A contact at a vertex belongs to two closed sides, and this convention assigns it to just one. If necessary, we will reflect $P$ in the real axis and replace $\zeta$ by $\bar\zeta$ to obtain the required assignment.

List the vertices $y_1,\ldots,y_N$ of $TP$ counterclockwise, with indices taken modulo $N$. The boundary-contact theorem puts them on $\partial P$ and ensures that every side of $P$ meets $TP$. For each image edge $[y_j,y_{j+1}]$, let $Q_j$ be the intersection of $P$ with the closed half-plane bounded by the line through its endpoints and containing $TP$. Define
\[
 k_j=|\ext P\cap[y_j,y_{j+1}]_{\partial P}|.
\]
Thus $k_j$ counts the vertices of $P$ on this boundary arc, including any at its endpoints.

If $Q_j\ne P$, the cutting line crosses the interior of $P$. Its two boundary intersections are $y_j,y_{j+1}$, and it discards the counterclockwise open arc between them. Counting the old vertices on the closed arc and adding back the two endpoints gives
\[
 TP\subseteq Q_j\subsetneq P,\qquad
 |\ext Q_j|\le N+2-k_j.
\]
Since $TQ_j\subseteq TP\subseteq Q_j$, the cut polygon is invariant and must have at least $N$ vertices. Hence $k_j\le2$. If $k_j=2$, it has exactly $N$ vertices, and every new vertex is one of the endpoints $y_j,y_{j+1}\in T\ext P$. Theorem~\ref{thm:branching} therefore gives $b_T(Q_j)=b_T(P)$, since $P$ minimizes this count.

Choose a vertex $v_*$ of $P$ with $|v_*|=1$. Since $|y_j|\le\rho<1$, this vertex is not one of the $y_j$ and lies in exactly one open boundary arc $(y_j,y_{j+1})_{\partial P}$. Any proper cut with $k_j=2$ must remove $v_*$. Otherwise, $Q_j$ would still be normalized and minimize $b_T$, but have smaller area than $P$, contradicting its choice. The arcs discarded by different cuts are disjoint, so at most one proper cut has $k_j=2$. If $Q_j=P$, the image-edge line supports $P$ and the arc lies along one side. Its half-open part $(y_j,y_{j+1}]_{\partial P}$ therefore contains at most one vertex of $P$.

To count vertices with the ending endpoint included and the starting endpoint excluded, define
\[
 r_j=|\ext P\cap(y_j,y_{j+1}]_{\partial P}|,
 \qquad
 \varepsilon_j=
 \begin{cases}
  1,&y_j\in\ext P,\\
  0,&y_j\notin\ext P.
 \end{cases}
\]
The half-open arcs partition $\partial P$, so they count each vertex of $P$ exactly once. Also, adding the starting endpoint recovers the closed-arc count. Thus
\[
 \sum_j r_j=N,\qquad k_j=r_j+\varepsilon_j,\qquad 0\le r_j\le2.
\]
If $r_j=2$, the cut must be proper, since an uncut arc contains at most one vertex in its half-open part. It then has $k_j=2$, so at most one $r_j$ can equal two. Since there are $N$ arcs and their counts sum to $N$, either every $r_j=1$, or there are unique indices $s,t$ such that
\[
 r_s=2,\qquad r_t=0,\qquad r_j=1\quad(j\notin\{s,t\}).
\]

Now suppose $r_s=2$ and $r_t=0$. Since $k_s=r_s+\varepsilon_s\le2$, we have $\varepsilon_s=0$: $y_s$ is not a vertex of $P$. We claim that $\varepsilon_{s+1}=1$, so $y_{s+1}$ is a vertex. Otherwise, the two vertices counted by $r_s$ would both lie strictly inside the arc $(y_s,y_{s+1})_{\partial P}$. They are consecutive vertices of $P$, and their joining side contains no vertex of $TP$.

But every side of $P$ meets $TP$. A linear functional defining this side would therefore attain its maximum over $TP$ at a vertex of $TP$, placing an image vertex on the side, a contradiction. Thus $\varepsilon_s=0$ and $\varepsilon_{s+1}=1$: as the indices run around the polygon, the sequence $(\varepsilon_j)$ changes from zero to one at $s$.

We next show that any change from one to zero must occur at $t$. Suppose $\varepsilon_j=1$ and $\varepsilon_{j+1}=0$ with $j\ne t$. Since $\varepsilon_s=0$, we also have $j\ne s$. Hence $r_j=1$ and $k_j=r_j+\varepsilon_j=2$.

This cut must be proper. If $Q_j=P$, the arc would start at the vertex $y_j$ and end at $y_{j+1}$ inside the same side, before that side's ending vertex. Excluding $y_j$ would leave no vertex in $(y_j,y_{j+1}]_{\partial P}$, giving $r_j=0$, a contradiction. Thus $Q_j\ne P$. We would then have two distinct proper cuts, $Q_j$ and $Q_s$, with $k_j=k_s=2$, contradicting the fact that at most one proper cut has count two.

In a sequence of zeros and ones read cyclically, every change from zero to one is matched by a change back to zero. Since a change from one to zero can occur only at $t$, there is exactly one change in each direction: from zero to one at $s$, and from one to zero at $t$. Now include the starting endpoint of each arc and exclude its ending endpoint. This adds $\varepsilon_j$ to the count and subtracts $\varepsilon_{j+1}$, giving
\[
 |\ext P\cap[y_j,y_{j+1})_{\partial P}|
 =r_j+\varepsilon_j-\varepsilon_{j+1}=1
 \quad\text{for every }j.
\]

Thus either every arc $(y_j,y_{j+1}]_{\partial P}$ contains exactly one vertex of $P$, or every arc $[y_j,y_{j+1})_{\partial P}$ does. In either case, call this vertex $u_j$; the vertices $u_j$ occur in positive cyclic order. In the second case, the arc $[y_j,y_{j+1})_{\partial P}$ contains $u_j$, so
\[
 y_j\in(u_{j-1},u_j].
\]
Each side, with its starting vertex excluded and its ending vertex included, therefore contains exactly one image vertex.

In the first case, we instead have $y_{j+1}\in[u_j,u_{j+1})$. Reflect $P$ across the real axis, replacing $P$ by $\overline P$ and $\zeta$ by $\overline\zeta$, and list the reflected vertices in positive cyclic order. Reflection reverses cyclic order, so each assigned side now excludes its starting vertex and includes its ending vertex. It preserves the extremality hypotheses \eqref{eq:extremal}, the count $b_T$, normalization, and area. We continue to write $P,\zeta$ for the resulting polygon and complex number. The map $Tz=\zeta z$ scales distances from the origin by $|\zeta|$ and rotates by $\arg\zeta$.

Multiplication by $\zeta$ preserves cyclic order, so the assignment of image vertices to sides shifts every index by the same amount. Label the vertices of $P$ as $v_i$ in positive cyclic order, and label the side $(v_{i-1},v_i]$ by its ending vertex $v_i$. There is a unique integer $1\le\kappa\le N$ such that the image of $v_{i-\kappa}$ lies on side $i$. Write this image as a convex combination of the side's endpoints:
\begin{equation}\label{eq:contacts}
 c_i=\zeta v_{i-\kappa}
 =\beta_i v_{i-1}+\alpha_i v_i\in(v_{i-1},v_i],
 \qquad \alpha_i=1-\beta_i>0,
\end{equation}
where $0\le\beta_i<1$ and all vertex indices are taken modulo $N$. Thus $\beta_i=0$ precisely when $c_i=v_i$.

If $\kappa=N$, then $c_i=\zeta v_i\ne v_i$, since $|\zeta|<1$ and $v_i\ne0$. Hence every image lies strictly inside its assigned side. Reflecting the polygon and relabelling in positive cyclic order gives the same representation with $\kappa=1$ and with $\alpha_i,\beta_i$ interchanged. We can therefore arrange $1\le\kappa<N$.

Choose real arguments $\Phi_i$ of the vertices that increase with $i$, and extend them to all integer indices by $\Phi_{i+N}=\Phi_i+2\pi$. Because $0$ lies inside $P$, each side spans an angular gap $\gamma_i=\Phi_i-\Phi_{i-1}$ in $(0,\pi)$. Argument increases along that side. Let $\theta_\zeta\in(0,2\pi)$ represent $\arg\zeta$. The side assignment therefore gives
\begin{equation}\label{eq:lift}
 \Phi_{i-1}<\Phi_{i-\kappa}+\theta_\zeta\le\Phi_i,
 \qquad \frac{\kappa-1}{N}<\frac{\theta_\zeta}{2\pi}<\frac\kappa N.
\end{equation}
For the first inequality, no additional multiple of $2\pi$ is needed: both $\theta_\zeta$ and the interval $(\Phi_{i-1}-\Phi_{i-\kappa},\Phi_i-\Phi_{i-\kappa}]$ lie in $(0,2\pi)$. Summing over all $N$ sides gives $2\pi(\kappa-1)<N\theta_\zeta\le2\pi\kappa$. Equality on the right would force $c_i=v_i$ for every $i$, hence $v_i=\zeta v_{i-\kappa}$. Repeatedly following these indices around the polygon would give $v_i=\zeta^L v_i$ for some positive integer $L$, contradicting $v_i\ne0$ and $|\zeta|<1$. This proves the strict upper bound.

\subsection{The local operation and the global normal form}
We modify the polygon by reversing the order of two matrix factors. Suppose $A=HS$, where $H,S$ are row-stochastic and $S$ is invertible. Then
\[
 A'=SH=SAS^{-1}
\]
is also row-stochastic and has the same eigenvalues as $A$, because the matrices are similar. For the vertex list $v=(v_1,\ldots,v_N)^T$, this turns $Av=\zeta v$ into $A'(Sv)=\zeta(Sv)$. We will choose $S$ so that $Sv$ replaces one vertex $v_i$ by a point on the preceding side $[v_{i-1},v_i]$ and leaves the other vertices fixed.

To write the contact identities in matrix form, define $C$ and $B(\beta)$ by their action on a column vector $x$:
\[
 (Cx)_r=x_{r+1},\qquad
 (B(\beta)x)_i=\beta_i x_{i-1}+\alpha_i x_i,
\]
with indices modulo $N$. Thus $C$ shifts the coordinates by one place, while row $i$ of $B(\beta)$ uses the two endpoints of side $i$. Both matrices are row-stochastic. Equation~\eqref{eq:contacts} becomes
\[
 A=C^\kappa B(\beta),\qquad Av=\zeta v.
\]
For the vertex replacement, define $E_i$ by
\[
 (E_i x)_r=
 \begin{cases}
 x_{i-1},&r=i,\\
 x_r,&r\ne i.
 \end{cases}
\]
It copies coordinate $i-1$ into coordinate $i$ and leaves all others fixed, so $E_i^2=E_i$. Set $S_i(t)=(1-t)I+tE_i$. For $0\le t<1$, this matrix is row-stochastic and invertible, and $S_i(t)v$ replaces $v_i$ by $(1-t)v_i+t v_{i-1}$, leaving all other vertices fixed.

Since $0<\alpha_i\le1$, the quantities $w_i=-\log\alpha_i$ are nonnegative. The following lemma sets $\alpha_i'=1$ and replaces $\alpha_{i+\kappa}$ by $\alpha_i\alpha_{i+\kappa}$. In terms of these weights, this is
\[
 w_i'=0,\qquad w_{i+\kappa}'=w_i+w_{i+\kappa}.
\]
Thus the weight at index $i$ is removed and added to the weight at index $i+\kappa$, modulo $N$. Reversing the matrix factors preserves stochasticity and the eigenvalue $\zeta$. The boundary-contact property then forces the vertices used by each changed row to lie on a single polygon side.

\begin{lemma}\label{lem:interchange}
Assume $\beta_i=t>0$ and $\beta_{i+1}=0$. Set $v_i'=c_i$ and $v_j'=v_j$ for $j\ne i$. The polygon $P'=\operatorname{conv}\{v_1',\ldots,v_N'\}$ has $N$ vertices and satisfies $\zeta P'\subseteq P'$. Every image vertex lies on $\partial P'$, and every side of $P'$ meets $\zeta P'$. The images $\zeta v_{j-\kappa}'$ lie on the sides $(v_{j-1}',v_j']$, with new coefficients satisfying
\begin{equation}\label{eq:coalescence}
 \alpha_i'=1,\qquad \alpha_{i+\kappa}'=\alpha_i\alpha_{i+\kappa},
 \qquad \alpha_j'=\alpha_j\quad(j\notin\{i,i+\kappa\}).
\end{equation}
The matrices $C^\kappa B(\beta')$ and $C^\kappa B(\beta)$ are similar, where $\beta_j'=1-\alpha_j'$.
\end{lemma}
\begin{proof}
Since $t=\beta_i$ and $\alpha_i=1-\beta_i>0$, we have $0<t<1$. Let $\widehat B$ be $B(\beta)$ with row $i$ replaced by the row that copies coordinate $i$. Thus $\widehat\beta_i=0$, $\widehat\alpha_i=1$, and all other coefficients are unchanged. Set $S=S_i(t)$ and $H=C^\kappa\widehat B$. Multiplying $\widehat B$ on the right by $S$ restores row $i$ of $B$. The only other row that could change is row $i+1$, but $\beta_{i+1}=0$. Hence
\[
 B=\widehat B S,\qquad A=HS,\qquad A'=SH=SAS^{-1}.
\]
The matrices $H,S$ are row-stochastic, and $S^{-1}=(I-tE_i)/(1-t)$. Thus $v'=Sv$ satisfies $A'v'=\zeta v'$, with $v_i'=c_i$ and all other vertices fixed. Row-stochasticity gives $\zeta P'\subseteq P'$. The polygon $P'$ has at most $N$ vertices, so minimality of $N$ forces it to have exactly $N$. These vertices retain their cyclic order, since $v_i$ moves into the side $[v_{i-1},v_i]$. The boundary-contact property established above applies to $P'$.

Put $j=i+\kappa$, with indices modulo $N$. Since $1\le\kappa<N$, we have $j\ne i$. The identity $S_i(t)C^\kappa=C^\kappa S_j(t)$ gives
\[
 A'=C^\kappa B',\qquad B'=S_j(t)\widehat B.
\]
Only row $j$ of $\widehat B$ changes. Its action on a vector $x$ is
\[
 (B'x)_j=t\widehat\beta_{j-1}x_{j-2}
 +\bigl(t\widehat\alpha_{j-1}+(1-t)\beta_j\bigr)x_{j-1}
 +(1-t)\alpha_j x_j.
\]
The coefficients of $x_{j-1}$ and $x_j$ are positive, since $0<t<1$ and every $\widehat\alpha$ is positive. Applying this row to $v'$ expresses the boundary point $\zeta v'_{j-\kappa}$ as a convex combination of vertices of $P'$. By the face rule, all vertices with positive coefficients must lie on one side. The vertices $v'_{j-1},v'_j$ already determine that side, so the coefficient of $v'_{j-2}$ must vanish: a third vertex cannot lie on it. These three vertices are distinct because $N\ge3$. Therefore $\widehat\beta_{j-1}=0$. This includes the case $j-1=i$, where that coefficient is zero by definition.

Row $j$ now uses only coordinates $j-1,j$, and its coefficient at $j$ is $(1-t)\alpha_j=\alpha_i\alpha_j$. Row $i$ still has coefficient one at $i$, and every other row is unchanged from $B$. This proves \eqref{eq:coalescence} and identifies $B'=B(\beta')$. In every row, the coefficient at the ending vertex of the assigned side remains positive. Thus the images retain the same half-open side assignment and shift $\kappa$. Similarity of the corresponding matrices was established above.
\end{proof}

Let
\[
 I=\{i:w_i>0\}=\{i:\beta_i>0\}.
\]
These are the indices whose image vertices lie inside their assigned sides, so $|I|=b_T(P)$. The lemma changes this set to $(I\setminus\{i\})\cup\{i+\kappa\}$. If $i+\kappa$ already belonged to $I$, two positive weights would combine at one index, decreasing $b_T$ by one. This contradicts the choice of $P$ with minimum $b_T$. The next proposition uses the minimum-area condition to show that $I$ consists of consecutive indices around the polygon.

\begin{proposition}\label{prop:interval}
Let $[r]_N$ denote the residue of $r$ in $\{0,\ldots,N-1\}$. By choosing the starting vertex in the cyclic labelling, we can arrange
\[
 I=\{1,\ldots,\varphi\},\qquad \varphi=|I|\ge\delta:=\gcd(N,\kappa).
\]
If $\varphi<N$, there is an integer $h>0$ such that
\begin{equation}\label{eq:record}
 [h\kappa]_N=N-\varphi,
 \qquad [t\kappa]_N<N-\varphi\quad(0\le t<h).
\end{equation}
Thus $[h\kappa]_N$ is a strict record: it exceeds every earlier residue. Its distance from $N$ is $\varphi$. For $\varphi=N$, take $h=0$, whose residue is zero and whose distance from $N$ is $N$.
\end{proposition}
\begin{proof}
Repeatedly adding $\kappa$ modulo $N$ visits all $N/\delta$ indices in one residue class modulo $\delta$. Each such class must meet $I$. Otherwise, all its contact coefficients $\beta_j$ would be zero, so the identities $\zeta v_r=v_{r+\kappa}$ would give $v_r=\zeta^{N/\delta}v_r$. This contradicts $v_r\ne0$ and $|\zeta|<1$. There are $\delta$ classes, so $|I|\ge\delta$.

If $I$ contains every index, the proposition follows with $\varphi=N$ and $h=0$. Otherwise, consider the ending index $i$ of any maximal run of consecutive indices in $I$: $i\in I$ but $i+1\notin I$. Then $\beta_i>0$ and $\beta_{i+1}=0$, so Lemma~\ref{lem:interchange} applies. It replaces $v_i$ by $c_i$, strictly inside the preceding side, and gives a proper subpolygon. Its count $b_T$ cannot increase by the update formula, and cannot decrease by minimality. If necessary, rescaling restores normalization without changing this count.

Apply this replacement separately at each run's ending index in the original polygon $P$. Since $|c_i|\le\rho<1$, any modulus-one vertex other than $v_i$ would remain. The new polygon would then still be normalized and minimize $b_T$, but have smaller area, contradicting the choice of $P$. Thus each run's ending vertex would have to be the unique modulus-one vertex. There can therefore be only one run. Choosing its first index as $1$ gives $I=\{1,\ldots,\varphi\}$.

To prove the record assertion, assume $\varphi<N$ and temporarily label this run as $F\cup\{0\}$, where $F=\{N-\varphi+1,\ldots,N-1\}$. The indices in $F$ carry the other positive weights $w_j$, which remain unchanged while the weight initially at index $0$ is moved. Set $a_t=[t\kappa]_N$, and let $h$ be the first positive time at which $a_h\in\{N-\varphi,\ldots,N-1\}$. This time exists: that interval contains $\varphi\ge\delta$ consecutive residues, hence a multiple of $\delta$, and repeatedly adding $\kappa$ from zero visits every such multiple.

Until time $h$, move the positive weight initially at zero by successive applications of the lemma. If its current index is $a_t$, the positive-weight indices are $F\cup\{a_t\}$. Since $a_t<N-\varphi$, the next index $a_t+1$ is not in this set, so the lemma applies at $a_t$. It moves the weight to $a_{t+1}$ and leaves the weights at $F$ fixed. If $a_{t+1}\in F$, two positive weights would combine, decreasing $b_T$ by one and contradicting minimality. Otherwise, the new positive-weight set is $F\cup\{a_{t+1}\}$. These later polygons need only to preserve the minimum of $b_T$; they need not minimize area.

At the first entrance into $\{N-\varphi,\ldots,N-1\}$, every index except $N-\varphi$ belongs to $F$ and is therefore forbidden. Hence $a_h=N-\varphi$, while $a_t<N-\varphi$ for every $t<h$. This proves the claimed record inequalities. Returning to the original cyclic labelling does not change these displacements by multiples of $\kappa$.
\end{proof}


\Needspace{12\baselineskip}
\section{Return geometry}
\subsection{Arithmetic suspension}\label{sec:return-arithmetic}
An endpoint contact, where $\beta_{i+\kappa}=0$, gives $\zeta v_i=v_{i+\kappa}$. Start at an index $j\in I$ and repeatedly add $\kappa$ modulo $N$ until the first return to $I$, at index $k$. The intermediate destinations lie outside $I$, so the endpoint identities apply at each of those steps. At the final step, the destination is in $I$, giving $\zeta^H v_j=c_k$, an interior contact on the side ending at $v_k$. We call the list of indices before this return a tower, and the number of steps $H$ its height. The starting index is its base, and each subsequent index is another level. For example, with $N=8$, $\kappa=3$, and $I=\{1,2\}$, the path $1\to4\to7\to2$ first returns to $I$ after three steps. Its tower consists of $1,4,7$, has height $3$, and closes with $\zeta^3v_1=c_2$.

We also retain the integer index $j+t\kappa$ before reducing it modulo $N$. An increase by $N$ records a full turn in the cyclic vertex order. Combining the endpoint identities along a tower gives a power of $\zeta$; its height and the number of turns will enter the product equation. We first determine these paths using only the index arithmetic.

The arithmetic follows the residues $a_t=[t\kappa]_N$. A record residue exceeds every earlier value in this sequence; the initial value $a_0=0$ is also a record. Its deficit $\nu=N-a_t$ is its gap to $N$. For $N=8$ and $\kappa=3$, the sequence begins $0,3,6,1,4,7,\ldots$, with records $0,3,6,7$ and deficits $8,5,2,1$. The preceding proposition shows that $\varphi$ is one of these deficits. We will use consecutive records to determine the tower heights and return destinations.

The calculation with record residues below reformulates Karpelevi\v{c}'s analysis of Euclidean remainders \cite[\S3, Lemma~2 and Theorems~III--IV]{Karp}. Slater studies related gaps and return times for rotations \cite{Slater}, and Hardy and Wright provide the background on Farey fractions \cite[Chapter~III]{HardyWright}. We prove here how the towers cover all $N$ indices, specifying which interval endpoints are included.

Fix $1\le\kappa<N$ and $\delta=\gcd(N,\kappa)$. Define $L:\mathbb Z^2\to\mathbb Z$ by $L(a,b)=a\kappa-bN$. For each record time $h$, let $\nu=N-[h\kappa]_N$ be its deficit and encode the record by $V=(h,b)$, where
\[
 b=\begin{cases}
 1,&h=0,\\
 \lceil h\kappa/N\rceil,&h>0.
 \end{cases}
\]
This choice gives $L(V)=-\nu$, including the initial record $V=(0,1)$ with $\nu=N$. The attainable residues are the multiples of $\delta$, whose largest value is $N-\delta$, so the last record has deficit $\delta$.

\begin{lemma}\label{lem:record-basis}
Let $V=(h,b)$ and $V'=(h',b')$ encode two consecutive records, with deficits $\nu>\nu'$. Set $U=V'-V=(q,p)$ and $\Delta=\nu-\nu'$. Then
\begin{equation}\label{eq:record-basis}
 \det(U,V)=1,\quad L(U)=\Delta,\quad L(V)=-\nu,
 \quad q\nu+h\Delta=N,\quad \gcd(\Delta,\nu)=\delta.
\end{equation}
Here $\det(U,V)=qb-ph$. Thus every pair of integers can be written uniquely as an integer combination of $U$ and $V$. In particular, $q>0$ and $0<\Delta<\nu$.
\end{lemma}
\begin{proof}
The attainable residues are the multiples of $\delta$, so the last deficit is $\delta$. Each record vector $V=(h,b)$ is primitive, meaning $\gcd(h,b)=1$. This holds for the initial vector $(0,1)$. Otherwise, if $V=gW$ for an integer $g>1$, the integer vector $W$ would have time $h/g<h$ and positive deficit $\nu/g<\nu$. Its residue would exceed the record at time $h$, a contradiction.

Since $h'>h$ and $\nu>\nu'>0$, the integer
\[
 D=-\det(V,V')=\frac{h'\nu-h\nu'}{N}
\]
is positive. Suppose $D>1$. B\'ezout's identity says that, because $\gcd(h,b)=1$, there are integers $r,s$ with $hr+bs=1$. Set $W_0=(s,-r)$. Then $\det(V,W_0)=-hr-bs=-1$, so every integer vector has unique integer coordinates in the basis $(V,W_0)$. Write $V'=aV+DW_0$. Since $V'$ is primitive, $\gcd(a,D)=1$. Consequently
\[
 W=\left(\lceil a/D\rceil-a/D\right)V+D^{-1}V'
   =\lceil a/D\rceil V+W_0
\]
is an integer vector of the form $W=\alpha V+\beta V'$ with $0<\alpha,\beta<1$. If $\alpha+\beta>1$, replace $W$ by $V+V'-W$; the new coefficients still lie in $(0,1)$ and have sum at most one. The time $t$ and deficit of $W$ then satisfy
\[
 0<t=\alpha h+\beta h'<h',\qquad
 0<-L(W)=\alpha\nu+\beta\nu'<\nu.
\]
Thus a positive integer time before $h'$ would have a residue exceeding the record at $h$. This is impossible for consecutive records. Hence $D=1$, and $\det(U,V)=-\det(V,V')=1$.

By linearity, $L(U)=L(V')-L(V)=\Delta$, while $L(V)=-\nu$ by definition. Expanding the determinant gives
\[
 q\nu+h\Delta
 =q(bN-h\kappa)+h(q\kappa-pN)
 =N(qb-ph)=N.
\]
Since $U,V$ form an integer basis, the values of $L$ are exactly the integer combinations of $\Delta$ and $\nu$. By its definition, they are also the integer combinations of $\kappa$ and $N$. Therefore $\gcd(\Delta,\nu)=\gcd(\kappa,N)=\delta$. Finally, $q=h'-h>0$ and $\Delta=\nu-\nu'$ lies in $(0,\nu)$.
\end{proof}

The identity $q\nu+h\Delta=N$ can be written
\[
 (\nu-\Delta)q+\Delta(q+h)=N,
\]
suggesting $\nu-\Delta$ towers of height $q$ and $\Delta$ towers of height $q+h$. The congruences specify where these towers are expected to return. To show that their lists partition all $N$ indices, we must also exclude repetitions. The next lemma does this by following cycles of the proposed return map on the tower bases: a cycle ends when repeated returns reach the starting base again. The argument uses only the stated integer relations, so it also applies later to a different choice of bases.

\Needspace{19\baselineskip}
\begin{lemma}\label{lem:suspension}
Let $\nu,\Delta,q$ be positive integers and $h$ a nonnegative integer. Assume $1\le\Delta\le\nu\le N$ and
\[
 \begin{gathered}
 q\kappa\equiv\Delta\pmod N,\qquad h\kappa\equiv-\nu\pmod N,\\
 q\nu+h\Delta=N,\qquad \gcd(\Delta,\nu)=\gcd(\kappa,N)=\delta.
 \end{gathered}
\]
Take $J=\{1,\ldots,\nu\}$ as the starting indices, interpreted modulo $N$. For each $j\in J$, define a tower height $H_j$ and a return index $\sigma(j)$ by
\[
 H_j=\begin{cases}
 q,&1\le j\le\nu-\Delta,\\
 q+h,&\nu-\Delta<j\le\nu,
 \end{cases}
 \qquad
 \sigma(j)=\begin{cases}
 j+\Delta,&1\le j\le\nu-\Delta,\\
 j+\Delta-\nu,&\nu-\Delta<j\le\nu.
 \end{cases}
\]
Thus $\sigma(j)$ is addition by $\Delta$ modulo $\nu$, with values in $J$. The map
\[
 \pi(t,j)=[j+t\kappa]_N,\qquad j\in J,\quad t=0,\ldots,H_j-1,
\]
lists each residue in $\{0,\ldots,N-1\}$ exactly once. Moreover, starting from $j$, repeated addition by $\kappa$ modulo $N$ first returns to $J$ after $H_j$ steps, at index $\sigma(j)$. The towers therefore partition all $N$ indices, with heights and first returns as displayed.
\end{lemma}
\begin{proof}
First treat the pairs $(t,j)$ as distinct positions in the towers, before checking whether their images under $\pi$ repeat. From $(t,j)$ move to $(t+1,j)$ if $t<H_j-1$, and from the top $(H_j-1,j)$ move to the next base $(0,\sigma(j))$. Each move adds $\kappa$ modulo $N$ under $\pi$. This is immediate inside a tower. At its top, the required relation is
\[
 j+H_j\kappa\equiv\sigma(j)\pmod N.
\]
For $j\le\nu-\Delta$, it follows from $q\kappa\equiv\Delta$. For $j>\nu-\Delta$, it follows from $(q+h)\kappa\equiv\Delta-\nu$ and $\sigma(j)=j+\Delta-\nu$.

Since $\gcd(\Delta,\nu)=\delta$, repeated application of $\sigma$ visits every base in one residue class modulo $\delta$. There are $\delta$ such cycles, each containing $\nu/\delta$ bases. The last $\Delta$ bases contain exactly $\Delta/\delta$ members of each class. Thus, joining the towers according to one base cycle gives a cycle of
\[
 q\frac{\nu}{\delta}+h\frac{\Delta}{\delta}=\frac N\delta
\]
positions. Addition by $\kappa$ modulo $N$ also has period $N/\delta$. The images of these positions therefore visit every residue in one class modulo $\delta$, without repetition. Different base cycles give different classes, so $\pi$ lists every residue exactly once.

For $1\le t<H_j$, the residue $\pi(t,j)$ cannot belong to $J$: it would already occur at the base position $(0,k)$ for some $k\in J$, contradicting the absence of repetitions. After $H_j$ steps, the top relation gives the return index $\sigma(j)$. Hence $H_j$ is the first return time and $\sigma(j)$ its destination, including when $\nu=\Delta=1$.
\end{proof}

Recall that the deficit of a record at time $h$ is its gap $N-[h\kappa]_N$ to $N$. By Proposition~\ref{prop:interval}, $\varphi$ is such a deficit. Assume $\varphi>\delta$, so the record has a successor. Let $V=(h,b)$ be its record vector, $V'$ the next one, and $U=V'-V=(q,p)$. Let $\Delta$ be the drop in deficit between these records. Lemma~\ref{lem:record-basis} proves that these data satisfy
\begin{equation}\label{eq:tower-data}
 q\kappa-pN=\Delta,\qquad h\kappa-bN=-\varphi,
 \qquad q\varphi+h\Delta=N,\qquad
 1\le\Delta<\varphi,\quad\gcd(\Delta,\varphi)=\delta.
\end{equation}
This includes $\varphi=N$: the initial record is $V=(0,1)$ and the next is $V'=(1,1)$.

Reducing the first two identities in \eqref{eq:tower-data} modulo $N$ gives
\[
 q\kappa\equiv\Delta\pmod N,\qquad h\kappa\equiv-\varphi\pmod N.
\]
Together with the remaining relations, $q>0$, $h\ge0$, and $\varphi\le N$, these verify the hypotheses of Lemma~\ref{lem:suspension} with $\nu=\varphi$ and bases $J=I=\{1,\ldots,\varphi\}$. We can therefore apply that lemma to obtain the tower partition and first returns.

Every intermediate destination $k$ lies outside $I$, so $\beta_k=0$ and the endpoint identity applies. At the first return, $\sigma(j)\in I$ and $\beta_{\sigma(j)}>0$, giving the interior contact $c_{\sigma(j)}$ on the side ending at $v_{\sigma(j)}$. For each base $j\in I$, these steps give
\begin{equation}\label{eq:towers}
 \zeta^t v_j=v_{\pi(t,j)}\ (0\le t<H_j),\qquad
 \zeta^{H_j}v_j=c_{\sigma(j)},\qquad
 \zeta^h v_\varphi=v_0,\quad \zeta^qv_0=c_\Delta.
\end{equation}
Here $v_0=v_N$. The tower starting at $\varphi$ has height $H_\varphi=q+h$ and returns to $\sigma(\varphi)=\Delta$. Since $\varphi+h\kappa\equiv0\pmod N$, its level $h$ is $\pi(h,\varphi)=0$. As $q>0$, this level occurs before the return. The first $h$ steps therefore reach $v_0$, and the remaining $q$ steps end at $c_\Delta$, proving the last two identities.

Continue the example $N=8$, $\kappa=3$. The record at $h=2$ has residue $6$ and deficit $\varphi=2$. The next record, at time $5$, has residue $7$ and deficit $1$. Thus $q=5-2=3$ and $\Delta=2-1=1$, giving the towers
\[
\begin{array}{c|c|c|c}
\text{base }j&H_j&\text{indices before return}&\sigma(j)\\ \hline
1&3&1,4,7&2\\
2&5&2,5,0,3,6&1
\end{array}
\]
Each residue occurs once. If an invariant polygon has these contact data, its tower identities give $\zeta^3v_1=c_2$ and $\zeta^5v_2=c_1$. The second tower also gives $\zeta^2v_2=v_0$. We return to this example after Theorem~\ref{thm:product} to derive its polynomial equation and the condition tracking the accumulated argument. Existence of a polygon with these contact data still requires a geometric argument.

Reducing indices modulo $N$ identifies the vertices but loses the multiples of $2\pi$ in their continued arguments. To retain these full turns, we use the real vertex arguments from \eqref{eq:lift}, with $\Phi_{i+N}=\Phi_i+2\pi$. If a path starts at $v_a$ and its next $t$ images are endpoints, then
\begin{equation}\label{eq:path-lift}
 \Phi_a+t\theta_\zeta=\Phi_{a+t\kappa}.
\end{equation}
If the last image is inside its assigned side and all previous images are endpoints, this becomes
\[
 \Phi_{a+t\kappa-1}<\Phi_a+t\theta_\zeta<\Phi_{a+t\kappa}.
\]
Here $a+t\kappa$ is kept as an integer, without reduction modulo $N$. Joining endpoint paths adds their lengths and index advances. When we combine their endpoint identities into one equation involving a power of $\zeta$, we retain the real argument identity as well. This preserves the total argument change, including its multiples of $2\pi$, which are needed to select the boundary root.


\subsection{Face persistence and projective transfer}

The base vertices are $v_j$ with $j\in I=\{1,\ldots,\varphi\}$, listed in cyclic order as $v_1,\ldots,v_\varphi,v_1$. The first return from $v_j$ lands at the interior contact $c_{\sigma(j)}$ on the side ending at $v_{\sigma(j)}$. Here $\sigma(j)=j+\Delta$ modulo $\varphi$, with values in $I$. When $\Delta=1$, this is the next base in the list. When $\Delta>1$, the return skips the $\Delta-1$ bases strictly between its starting and return indices; these are the intervening bases. For example, if $\varphi\ge3$, a return from $v_1$ to the side ending at $v_3$ skips the intervening base $v_2$.

We follow Karpelevi\v{c}'s strategy of moving polygon vertices to exclude such returns \cite{Karp}. Under our extremality assumptions, Theorem~\ref{thm:saturation} states that every invariant $N$-gon $Q$, with $\zeta Q\subseteq Q$, has every vertex image $\zeta u$ on $\partial Q$, and every side of $Q$ meets $\zeta Q$. These are the boundary-contact properties used here. The intended motion keeps $\zeta$ fixed and preserves $\zeta Q\subseteq Q$, but puts one vertex image in $\operatorname{int}Q$, contradicting the first property.

The motion is built from projections between lines, each following lines through a fixed point. We first need supporting lines: lines that meet the polygon and leave it entirely in one closed half-plane. A supporting line that meets the polygon at just one vertex is said to expose that vertex. We obtain these lines from the behavior of faces under repeated images. The required projection formulas are derived below; Coxeter \cite{Coxeter} gives the background on these projections and their compositions.

A polygon's nonempty faces are its vertices, viewed as single points, its closed sides, and the whole polygon. For $x\in P$, the smallest face $G_P(x)$ is the face contained in every other face that contains $x$. Thus
\[
 G_P(x)=\begin{cases}
 \{x\},&x\text{ is a vertex},\\
 [v_{i-1},v_i],&x\in(v_{i-1},v_i),\\
 P,&x\in\operatorname{int}P.
 \end{cases}
\]
These faces have dimensions zero, one, and two, respectively; ``smallest'' refers to containment.

Let $F(x)=Mx+a$ be a one-to-one affine map, where $M$ is a matrix and $a$ a fixed vector, with $F(P)\subseteq P$. Then
\[
 F(G_P(x))\subseteq G_P(Fx),\qquad
 \dim G_P(x)\le\dim G_P(Fx).
\]
If $Fx$ is interior, the inclusion follows immediately from $G_P(Fx)=P$. Otherwise, take a supporting linear functional $\ell$ maximal at $Fx$. Since $F(P)\subseteq P$, we have $\ell(Fy)\le\ell(Fx)$ for every $y\in P$. Thus $\ell\circ F$ is maximal at $x$. Because $x$ lies in the interior of $G_P(x)$ relative to that face, this affine functional must be constant on the whole face. Consequently $F(G_P(x))$ lies in every supporting face containing $Fx$, and hence in their intersection, $G_P(Fx)$. Since $F$ is one-to-one, it preserves the dimension of $G_P(x)$, giving the dimension inequality.

In particular, a point that is not a vertex cannot become a vertex under any later application of $F$: its containing face cannot drop to dimension zero. A point strictly inside a side may move into the polygon's interior, but cannot move to a vertex. The next lemma combines this fact with the half-open side assignment to identify which endpoint must remain a vertex when an entire side stays on the boundary.

\begin{lemma}\label{lem:boundary-persistence}
Let $F$ be an invertible affine map with $F(P)\subseteq P$ and the half-open side assignment \eqref{eq:contacts}, namely
\[
 Fv_i\in(v_{i+\kappa-1},v_{i+\kappa}]
 \qquad(i\bmod N).
\]
Write $E_i=[v_{i-1},v_i]$. If an integer $r\ge1$ satisfies $F^r(E_i)\subseteq\partial P$, then the starting endpoint remains a vertex at every step:
\[
 F^t v_{i-1}=v_{i-1+t\kappa}\qquad(0\le t\le r).
\]
Also, suppose $F^rv_i$ lies strictly between the endpoints of a side whose containing line is $\ell$, and $F^tv_{i-1}$ is not a vertex for some $1\le t\le r$. Then the preimage line
\[
 L=\{x:F^rx\in\ell\}
\]
supports $P$ and satisfies $L\cap P=\{v_i\}$; thus it exposes $v_i$.
\end{lemma}
\begin{proof}
An invertible affine map sends open sets to open sets. Since $F(P)\subseteq P$, it follows that $F(\operatorname{int}P)\subseteq\operatorname{int}P$. If any point of $F(E_i)$ were interior, its image under $F^{r-1}$ would remain interior, contradicting $F^r(E_i)\subseteq\partial P$. Thus $F(E_i)$ is a nondegenerate boundary segment and lies on one side of $P$.

The side assignments give
\[
 Fv_{i-1}\in(v_{i+\kappa-2},v_{i+\kappa-1}],\qquad
 Fv_i\in(v_{i+\kappa-1},v_{i+\kappa}].
\]
The second image lies beyond the common vertex $v_{i+\kappa-1}$, which is excluded from its assigned side. For the segment joining these images to lie on one side, the first image must equal that common vertex. Hence
\[
 Fv_{i-1}=v_{i+\kappa-1},\qquad
 \relint(F(E_i))\subseteq\relint E_{i+\kappa},
\]
where the relative interior of a segment means the segment with both endpoints excluded.

To continue, we need the whole side $E_{i+\kappa}$ to stay on the boundary for the remaining $r-1$ steps. Choose $y$ strictly between the endpoints of $F(E_i)$. It also lies strictly between the endpoints of $E_{i+\kappa}$, while $F^{r-1}y\in\partial P$. Affinity expresses this boundary point as a convex combination, with both coefficients positive, of the images of those endpoints. A supporting functional maximal at the combination must be maximal at both images. Their whole joining segment therefore lies on the boundary, giving $F^{r-1}(E_{i+\kappa})\subseteq\partial P$. Induction on $r$ proves all the stated endpoint identities.

For the second conclusion, the line $\ell$ supports $P$ at $F^rv_i$. Since $F^r(P)\subseteq P$, its preimage $L$ supports $P$ at $v_i$. If $L$ met $P$ anywhere else, it would contain one of the two sides incident to $v_i$. If it contained $E_{i+1}$, the first conclusion would make $F^rv_i$ a vertex, contradicting its location inside a side. If it contained $E_i$, that conclusion would make $F^rv_{i-1}$ a vertex. But $v_{i-1}$ has a nonvertex image by time $r$, and such an image cannot later become a vertex. Both possibilities are excluded, so $L\cap P=\{v_i\}$.
\end{proof}

The previous lemma supplies exposing lines along which vertices may move. After moving the first vertex slightly along its side's line, we find each subsequent vertex by intersecting its exposing line with the line through the preceding moved vertex and a fixed side contact. This keeps the prescribed contacts on the corresponding sides. We must choose the initial motion so that the closing contact moves to the inward side of the new final side: the half-plane containing the polygon's interior.

Although the motion uses only points near the original vertices, its direction is determined by evaluating the same projection chain at another starting point. During that comparison, an intersection may occur at infinity. We therefore use completed lines: each ordinary line is supplemented by the point representing its direction, so parallel lines meet at that point.

We express these projections in homogeneous coordinates. A nonzero vector $X=(x,y,w)^T$ represents a point, and every nonzero scalar multiple represents the same point. When $w\ne0$, the ordinary point is $(x/w,y/w)$; vectors with $w=0$ represent points at infinity. The part with $w\ne0$ is an affine chart, where we can normalize $w=1$. A completed line has equation $\eta^TX=0$, with $\eta$ a nonzero column vector in $\mathbb R^3$; this equation defines the two-dimensional subspace representing the line.

Let $c$ represent the projection centre. Suppose the source line does not contain $c$ and the target line $\eta^TX=0$ satisfies $\eta^Tc\ne0$. Projection onto the target is induced by
\[
 X\longmapsto\left(I-\frac{c\eta^T}{\eta^Tc}\right)X,
\]
where $I$ is the $3\times3$ identity matrix. The result is a linear combination of $X,c$ and satisfies the target equation, so it represents the intersection of that line with the line through the original point and the centre. The map is an isomorphism between the two representing subspaces. In scalar coordinates along the lines, it has the fractional-linear form
\[
 \tau\longmapsto\frac{a\tau+b}{d\tau+e},\qquad ae-bd\ne0,
\]
with real constants $a,b,d,e$. A zero denominator represents an image at infinity. These maps and their compositions are projectivities.

An invertible $3\times3$ matrix changes the homogeneous coordinates. Dividing the first two resulting coordinates by the third gives ordinary plane coordinates. This third coordinate is an affine function of the original point, called the denominator. On a region where it is nonzero and keeps one sign, the transformation preserves segments and their interiors, because the transformed convex combinations still have positive weights.

\Needspace{20\baselineskip}
\begin{lemma}\label{lem:chain-holonomy}
Let $X_0,\ldots,X_{\ell+1}$, with $\ell\ge2$, be distinct consecutive vertices of a convex polygon $P$, listed in either direction around its boundary. Assume that at least one side of $P$ lies outside this chain. Choose contacts and exposing lines satisfying
\[
 C_i\in(X_{i-1},X_i)\quad(2\le i\le\ell+1),
 \qquad P\cap L_i=\{X_i\}\quad(2\le i\le\ell).
\]
Write $\aff(A,B)$ for the full line through distinct points $A,B$, and set $\Lambda_1=\aff(X_0,X_1)$ and $K=\aff(C_\ell,C_{\ell+1})$. An overbar on a line denotes its completion by a point at infinity.

Starting on $\overline\Lambda_1$, project through $C_2$ onto $L_2$, then through $C_i$ onto $L_i$ for $i=3,\ldots,\ell$, and finally through $X_{\ell+1}$ onto $K$. The resulting projectivity $\Pi:\overline\Lambda_1\to\overline K$ satisfies
\begin{equation}\label{eq:strict-holonomy}
 \Pi(X_1)=C_{\ell+1},\qquad
 \Pi(X_0)\in(C_{\ell+1},C_\ell).
\end{equation}
Let $S:\overline\Lambda_1\to\overline K$ be any projectivity with $S(X_1)=C_{\ell+1}$ and $S(X_0)=C_\ell$. Keep $X_0,X_{\ell+1}$ fixed. For small real $t$, define
\[
 \begin{gathered}
 Y(t)=(1-t)C_{\ell+1}+tC_\ell,\qquad X_1(t)=S^{-1}(Y(t)),\\
 X_i(t)=L_i\cap\aff(X_{i-1}(t),C_i)\qquad(2\le i\le\ell).
 \end{gathered}
\]
At $t=0$, these are the original vertices and $Y(0)=C_{\ell+1}$. Let $\mathcal D(t)$ be the signed distance of $Y(t)$ from the moving line $\aff(X_\ell(t),X_{\ell+1})$, positive on the side containing $X_0$. Then either $\mathcal D'(0)\ne0$, or $\mathcal D'(0)=0$ and $\mathcal D''(0)>0$. Thus there are arbitrarily small nonzero $t$ for which $Y(t)$ lies strictly on the inward side of the moving final line.
\end{lemma}

\begin{figure}[htbp]
\centering
\begin{tikzpicture}[x=2cm,y=1.4cm, font=\small]
 \coordinate (X0) at (0,0);
 \coordinate (X1) at (1,0);
 \coordinate (X2) at (2,1);
 \coordinate (X3) at (3,3);
 \coordinate (C2) at (1.5,.5);
 \coordinate (C3) at (2.5,2);
 \coordinate (Z2) at ({12/7},{4/7});
 \coordinate (W) at ({33/14},{25/14});
 \fill[black!5] (X0)--(X1)--(X2)--(X3)--cycle;
 \draw[black!35] (X3)--(X0);
 \draw[thick] (X0)--(X1)--(X2)--(X3);
 \draw[dashed,black!65] (1.3,-.05)--(2.65,1.975)
   node[right] {$L_2$};
 \draw[blue!65!black] (X0)--(Z2)--(X3);
 \draw[orange!75!black] (1.3,.2)--(2.8,2.45)
   node[right] {$K$};
 \foreach \p in {X0,X1,X2,X3,C2,C3,Z2,W}
   \fill (\p) circle[radius=1.25pt];
 \node[below left] at (X0) {$X_0$};
 \node[below] at (X1) {$X_1$};
 \node[right] at (X2) {$X_2$};
 \node[above] at (X3) {$X_3$};
 \node[above left] at (C2) {$C_2$};
 \node[above left] at (C3) {$C_3$};
 \node[below right] at (Z2) {$Z_2$};
 \draw[black!60] (W)--(1.93,1.86);
 \node[left] at (1.93,1.86) {$\Pi(X_0)$};
 \node[below] at (.4,0) {$\Lambda_1$};
\end{tikzpicture}
\caption{The two projection chains for $\ell=2$. Starting at $X_0$, projection through $C_2$ meets $L_2$ at $Z_2$; projection through $X_3$ then meets $K$ at $\Pi(X_0)$, strictly between $C_2$ and $C_3$. Starting at $X_1$ gives $X_2$ and then $C_3=\Pi(X_1)$.}
\end{figure}

\begin{proof}
We first locate $\Pi(X_0)$ by putting the chain in coordinates where its side slopes are ordered. We then use that location to choose the direction of the local motion.

Projection from a centre $c$ onto a line $B$ sends a point $X\ne c$ to the intersection of $B$ with the line through $c,X$. We use completed lines, so parallel ordinary lines intersect at a point at infinity.

A projectivity between completed lines is an invertible map that, in scalar line coordinates, has the fractional-linear form $\tau\mapsto(a\tau+b)/(d\tau+e)$, with $ae-bd\ne0$. Projection restricted to a completed source line $A$ is a projectivity when $c$ lies on neither $A$ nor $B$. This condition makes every intersection unique, and projection from the same centre back onto $A$ gives the inverse. The homogeneous formula above gives the fractional-linear form. Compositions of projectivities are again projectivities.

These conditions hold at every step here. The first projection has source $\overline\Lambda_1$, centre $C_2$, and target $\overline L_2$. Since adjacent sides are not collinear, $C_2\notin\Lambda_1$; since $P\cap L_2=\{X_2\}$, also $C_2\notin L_2$. For $3\le i\le\ell$, the source is $\overline L_{i-1}$, the centre is $C_i$, and the target is $\overline L_i$. The contact $C_i$ belongs to $P$ and differs from both $X_{i-1}$ and $X_i$, so exposure excludes it from both lines. Finally, the projection with source $\overline L_\ell$, centre $X_{\ell+1}$, and target $\overline K$ is well defined: exposure excludes $X_{\ell+1}$ from $L_\ell$, and $X_{\ell+1}\notin K$ because $C_\ell,C_{\ell+1}$ lie strictly inside different adjacent sides. Thus every projection, and their composition $\Pi$, is a projectivity.

Starting at $X_1$ gives $X_2,\ldots,X_\ell$ and finally $C_{\ell+1}$. Since $S^{-1}(Y(0))=X_1$ and these intersections are finite at $t=0$, all moving points remain finite and depend smoothly on $t$ near zero.

To locate the image of $X_0$, choose affine functions $\eta_0,\eta_1\ge0$ on $P$ that vanish there only at $X_0,X_{\ell+1}$, respectively. Such a function is obtained by adding the two affine functions that vanish on the endpoint's incident side lines and are positive inside $P$. Their sum $\eta=\eta_0+\eta_1$ is strictly positive on $P$. The functions $\eta_0,\eta$ are linearly independent, as their values at the two endpoints show. Choose an affine function $\psi$ so that $\eta_0,\psi,\eta$ form a basis of the affine functions on the plane. The change of coordinates
\[
 \Phi(x)=\left(\frac{\eta_0(x)}{\eta(x)},
                  \frac{\psi(x)}{\eta(x)}\right)
\]
is invertible and finite throughout $P$. For $x,y\in P$ and $0\le s\le1$,
\[
 \Phi((1-s)x+sy)=
 \frac{(1-s)\eta(x)\Phi(x)+s\eta(y)\Phi(y)}
 {(1-s)\eta(x)+s\eta(y)}.
\]
These positive weights show that the transformation preserves convexity, segments, and their interiors. We use these coordinates until the location of $\Pi(X_0)$ has been established.

The first coordinate has unique minimum $0$ at $X_0$ and unique maximum $1$ at $X_{\ell+1}$. Every vertical slice of the polygon is an interval, so its two boundary chains between these endpoints are graphs. There are no vertical sides or vertical exposing lines at the internal vertices: such a supporting line would attain an extreme first coordinate, which occurs only at an endpoint. Reflect the second coordinate if necessary so that the selected chain is the lower graph. Convexity then makes its side slopes strictly increasing, and the slope of each exposing line lies strictly between those of the two incident sides. Write
\[
 \begin{gathered}
 X_i=(t_i,f_i),\qquad t_0<\cdots<t_{\ell+1},\\
 d_1<\cdots<d_{\ell+1},\qquad d_i<s_i<d_{i+1}\quad(2\le i\le\ell),
 \end{gathered}
\]
where $d_i$ is the slope of $[X_{i-1},X_i]$ and $s_i$ that of $L_i$. Set $s_1=d_1$ and $h_i=t_i-t_{i-1}>0$.

Write each contact as
\[
 C_i=X_i-\gamma_i h_i(1,d_i),\qquad 0<\gamma_i<1
 \quad(2\le i\le\ell+1).
\]
Starting at $Z_1=X_0$, let $Z_i$ be the successive projections onto $L_i$. We seek them in the form
\[
 Z_i=X_i-r_i(1,s_i),\qquad r_1=h_1.
\]
Thus $r_i$ measures the horizontal displacement from $Z_i$ to $X_i$ along the exposing line. The condition that $Z_{i-1},C_i,Z_i$ lie on one line is
\[
 \det(C_i-Z_{i-1},Z_i-C_i)=0,
\]
with
\[
 \begin{split}
 C_i-Z_{i-1}&=(1-\gamma_i)h_i(1,d_i)+r_{i-1}(1,s_{i-1}),\\
 Z_i-C_i&=\gamma_i h_i(1,d_i)-r_i(1,s_i).
 \end{split}
\]
Expanding the determinant and using the reciprocals $x_i=1/r_i$ gives the recurrence
\begin{equation}\label{eq:positive-transfer}
 x_i=\frac{(1-\gamma_i)(s_i-d_i)}{\gamma_i(d_i-s_{i-1})}x_{i-1}
       +\frac{s_i-s_{i-1}}{\gamma_i h_i(d_i-s_{i-1})}.
\end{equation}
Both coefficients are strictly positive because $s_{i-1}<d_i<s_i$. Starting with $x_1=1/h_1$, the recurrence therefore gives finite positive $x_i$. The resulting $r_i=1/x_i$ satisfy the collinearity equations, so the constructed points are the unique projective intersections. This also proves that none of these intersections is at infinity in the chosen coordinates. The positive first term gives
\[
 0<r_i<\gamma_i h_i\frac{d_i-s_{i-1}}{s_i-s_{i-1}}
          <\gamma_i h_i\qquad(2\le i\le\ell).
\]

It remains to project $Z_\ell$ through $X_{\ell+1}$ onto $K$. Put
\[
 r=r_\ell,\quad g=\gamma_\ell h_\ell,\quad H=h_{\ell+1},
 \quad d=d_\ell,\quad s=s_\ell,\quad D=d_{\ell+1}.
\]
The affine function
\[
 G(W)=\det(X_{\ell+1}-Z_\ell,W-Z_\ell)
\]
vanishes precisely on the projection line through $Z_\ell,X_{\ell+1}$. Direct substitution gives
\[
 \begin{split}
 G(C_\ell)&=H(g-r)(D-s)+g(H+r)(s-d)>0,\\
 G(C_{\ell+1})&=-\gamma_{\ell+1}Hr(D-s)<0.
 \end{split}
\]
Here the first sign uses $0<r<g$ and $d<s<D$. The opposite signs place the intersection strictly between the contacts. Together with $\Pi(X_1)=C_{\ell+1}$, this proves \eqref{eq:strict-holonomy}.

Return to the original affine coordinates and use $Y(t)=(1-t)C_{\ell+1}+tC_\ell$ to parametrize $K$. Define $u(t)$ by
\[
 \Pi\bigl(S^{-1}(Y(t))\bigr)=Y(u(t)).
\]
This projectivity fixes coordinate $0$, so it has the form
\[
 u(t)=\frac{at}{1+bt},\qquad a\ne0.
\]
The endpoint conditions give $u(0)=0$ and $0<u(1)<1$. Hence
\begin{equation}\label{eq:holonomy-defect}
 t-u(t)=\frac{(1-a)t+bt^2}{1+bt}.
\end{equation}
The moving final line meets $K$ at $Y(u(t))$. At $t=0$, increasing the coordinate from $C_{\ell+1}$ toward $C_\ell$ enters the inward half-plane of the final side. Its signed-distance function, restricted to $K$, is affine and vanishes at $Y(u(t))$. Thus the signed distance defined in the statement satisfies
\[
 \mathcal D(t)=\Gamma(t)(t-u(t)),
\]
where $\Gamma$ is smooth near zero and $\Gamma(0)>0$. This factor remains positive for small $t$. Equation~\eqref{eq:holonomy-defect} now gives
\[
 \mathcal D'(0)=\Gamma(0)(1-a).
\]
If $a\ne1$, choose the sign of $t$ so that $(1-a)t>0$. If $a=1$, the inequality $0<u(1)<1$ forces $b>0$, and
\[
 \mathcal D''(0)=2\Gamma(0)b>0.
\]
In that case either sufficiently small nonzero sign of $t$ gives inward motion. The comparison at parameter $1$ therefore determines the direction of the local motion near $0$; the construction need not remain finite throughout $[0,1]$.
\end{proof}

Lemma~\ref{lem:chain-holonomy} shows how to move a chain so that its closing contact lies strictly inward of the final side, even when the first derivative of the signed distance vanishes. To obtain an invariant polygon, we must extend this motion to all vertices using the tower identities and keep every other returning image on its corresponding side. The next proof checks these conditions by comparing the indices of the moving bases with those of their target sides.

\begin{theorem}\label{thm:no-skipping}
For the fixed complex number $\zeta$, assume the conclusion of Theorem~\ref{thm:saturation} holds for every polygon $Q$ with at most $N$ vertices and $\zeta Q\subseteq Q$. In particular, $\zeta u\in\partial Q$ for every vertex $u$ of every such $Q$.

Let $P$ be an invariant $N$-gon, with $N\ge4$, whose interior contacts occur exactly on the sides $E_j$ with $j\in B=\{1,\ldots,\varphi\}$. Assume the return data \eqref{eq:tower-data}--\eqref{eq:towers} and $\varphi>\gcd(N,\kappa)$. Then $\Delta=1$: the first return from each base lands at the contact on the side ending at the next base, with base indices taken modulo $\varphi$.
\end{theorem}
\begin{proof}
Let $F(x)=\zeta x$, where $\zeta$ is fixed as in the theorem. Recall the return map and times:
\[
 \sigma(j)=1+[j+\Delta-1]_\varphi,\qquad
 H_j=q+h\mathbf1_{j>\varphi-\Delta},\qquad
 F^{H_j}v_j=c_{\sigma(j)},\quad F^hv_\varphi=v_0.
\]
Thus $H_j=q$ for the first $\varphi-\Delta$ bases and $H_j=q+h$ for the remaining $\Delta$. Every image before time $H_j$ is a vertex; the image at time $H_j$ lies strictly between the endpoints of the side $E_{\sigma(j)}$.

For $j\ge2$, the preceding base $v_{j-1}$ has a nonvertex image at time $H_{j-1}\le H_j$. Lemma~\ref{lem:boundary-persistence} therefore shows that the preimage line
\[
 \mathcal L_j:=F^{-H_j}(\aff E_{\sigma(j)})
       =\{x:F^{H_j}x\in\aff E_{\sigma(j)}\}
\]
exposes $v_j$: $P\cap\mathcal L_j=\{v_j\}$. This also covers the change from return time $q$ to $q+h$, because a nonvertex image cannot become a vertex at a later step.

Suppose $\Delta>1$. We will move some vertices while keeping $F(P)\subseteq P$, but make one vertex image strictly interior. The boundary contact assumption will then give a contradiction. First choose a chain of consecutive vertices on which to apply the local motion; then extend that motion to all tower levels.

The first vertex to be moved is $X_1$. We choose between two chains, going in opposite directions around the polygon, so that the selected chain is the shorter one. Write $\varepsilon$ for its direction and $a$ for the first return time of $X_1$:
\[
 (\varepsilon,\ell,X_i,a)=
 \begin{cases}
 (1,\Delta,v_i,q),&2\Delta\le\varphi+1,\\
 (-1,\varphi-\Delta+1,v_{\varphi-i},q+h),&2\Delta>\varphi+1,
 \end{cases}
 \qquad 0\le i\le\ell+1.
\]
Since $1<\Delta<\varphi$, we have $2\le\ell\le(\varphi+1)/2$. The chain uses $\ell+1$ sides. If $\varphi<N$, then $\ell+1\le\varphi<N$; if $\varphi=N\ge4$, then $\ell+1\le(N+3)/2<N$. Thus at least one polygon side lies outside the chain, as required by Lemma~\ref{lem:chain-holonomy}.

We keep $X_0,X_{\ell+1}$ fixed. They are bases outside the list to be moved, except possibly $X_0=v_0$. In that case $v_0=F^hv_\varphi$ remains fixed because its base $v_\varphi$ remains fixed. For $2\le i\le\ell$, take $L_i=\mathcal L_j$ when $X_i=v_j$. These indices belong to $\{2,\ldots,\varphi\}$, so $L_i$ exposes $X_i$. Let $C_i$ be the contact on $[X_{i-1},X_i]$ for $2\le i\le\ell+1$. In the two orientations, respectively,
\[
 (C_\ell,C_{\ell+1})=(c_\Delta,c_{\Delta+1})
 \quad\hbox{or}\quad(c_\Delta,c_{\Delta-1}),
\]
and the tower identities give
\[
 F^aX_0=C_\ell,\qquad F^aX_1=C_{\ell+1}.
\]
Consequently $F^a$ sends $\Lambda_1=\aff(X_0,X_1)$ onto $K=\aff(C_\ell,C_{\ell+1})$. Its restriction extends to a projectivity between the completed lines, since $F^a$ is invertible. This supplies the map $S$ required for the local motion.

Before moving the chain, we check which other returns must be preserved. Only the sides $E_i$ with $i\in B$ carry return contacts. Their endpoints are bases, except for $v_0$, which stays fixed. Let $b$ satisfy $v_b=X_1$, let $M$ be the base indices of $X_1,\ldots,X_\ell$, and let $k_*:=\sigma(b)$ be the index of the final chain side. Let $J$ be the indices of the internal chain sides $[X_{i-1},X_i]$, $2\le i\le\ell$, each indexed by its ending vertex. Taking base indices modulo $\varphi$, with representatives in $B$, gives
\begin{equation}\label{eq:network-disjointness}
 \sigma(M)=\{k_*+\varepsilon t:0\le t<\ell\},\qquad
 J=\{k_*-\varepsilon r:1\le r<\ell\}.
\end{equation}
These sets are disjoint. An intersection would require $\varphi\mid(t+r)$, whereas $1\le t+r\le2\ell-2<\varphi$. Hence every return to an internal chain side comes from a base outside $M$, which stays fixed. Also, because $\sigma$ is a permutation, $b$ is the only moving base whose return lands on the final side $E_{k_*}$. Every other moving base returns to a side with a fixed containing line. Figure~\ref{fig:global-motion} illustrates this separation and the extension along towers.

\begin{figure}[htbp]
\centering
\begin{tikzpicture}[x=1.15cm,y=.88cm,font=\small,
 vertex/.style={circle,draw=black!45,fill=white,minimum size=6mm,inner sep=1pt},
 moving/.style={vertex,draw=blue!65!black,fill=blue!9},
 seed/.style={vertex,draw=orange!80!black,fill=orange!14},
 side/.style={draw=black!45,fill=white,rounded corners=1pt,
              minimum width=7.8mm,minimum height=5.5mm,inner sep=1pt},
 internal/.style={side,draw=violet!70!black,fill=violet!10},
 target/.style={side,draw=blue!65!black,fill=blue!9},
 exceptional/.style={side,draw=orange!80!black,fill=orange!14}]
 \node[anchor=west] at (0,7.55) {(a) Cyclic indices modulo $\varphi=7$};
 \node[anchor=east] at (.35,6.15) {bases $j$};
 \node[seed] at (1,6.15) {$1$};
 \foreach \j in {2,3} \node[moving] at (\j,6.15) {$\j$};
 \foreach \j in {4,5,6,7} \node[vertex] at (\j,6.15) {$\j$};
 \draw[blue!65!black] (.65,6.53)--(.65,6.65)--(3.35,6.65)--(3.35,6.53);
 \node[blue!65!black,above] at (2,6.65) {$M$};
 \node[anchor=east] at (.35,5.15) {sides $E_k$};
 \foreach \k in {1,7} \node[side] at (\k,5.15) {$\k$};
 \foreach \k in {2,3} \node[internal] at (\k,5.15) {$\k$};
 \node[exceptional] at (4,5.15) {$4=k_*$};
 \foreach \k in {5,6} \node[target] at (\k,5.15) {$\k$};
 \draw[violet!70!black] (1.65,4.77)--(1.65,4.65)--(3.35,4.65)--(3.35,4.77);
 \node[violet!70!black,below] at (2.5,4.65) {$J$};
 \draw[blue!65!black] (3.65,4.77)--(3.65,4.65)--(6.35,4.65)--(6.35,4.77);
 \node[blue!65!black,below] at (5,4.65) {$\sigma(M)$};
 \node[anchor=west] at (0,3.85) {(b) Motion propagates down each tower};
 \node[anchor=east] at (.35,3.15) {$t=0$};
 \node[anchor=east] at (.35,2.25) {$t=1$};
 \node[anchor=east] at (.35,1.35) {$t=2$};
 \node[seed] (B1) at (1,3.15) {$v_1$};
 \node[seed] (T1) at (1,2.25) {$v_{11}$};
 \foreach \j/\k in {2/12,3/13} {
   \node[moving] (B\j) at (\j,3.15) {$v_{\j}$};
   \node[moving] (T\j) at (\j,2.25) {$v_{\k}$};
 }
 \foreach \j/\k in {4/14,5/15,6/16,7/0} {
   \node[vertex] (B\j) at (\j,3.15) {$v_{\j}$};
   \node[vertex] (T\j) at (\j,2.25) {$v_{\k}$};
 }
 \foreach \j/\k in {5/8,6/9,7/10}
   \node[vertex] (U\j) at (\j,1.35) {$v_{\k}$};
 \draw[->,orange!80!black] (B1)--(T1);
 \foreach \j in {2,3} \draw[->,blue!65!black] (B\j)--(T\j);
 \foreach \j in {4,5,6,7} \draw[->,black!45] (B\j)--(T\j);
 \foreach \j in {5,6,7} \draw[->,black!45] (T\j)--(U\j);
 \node[exceptional] (C1) at (1,.35) {$c_4$};
 \foreach \j/\k in {2/5,3/6} \node[target] (C\j) at (\j,.35) {$c_{\k}$};
 \foreach \j/\k in {4/7,5/1} \node[side] (C\j) at (\j,.35) {$c_{\k}$};
 \foreach \j/\k in {6/2,7/3} \node[internal] (C\j) at (\j,.35) {$c_{\k}$};
 \draw[->,dashed,orange!80!black] (T1)--(C1);
 \foreach \j in {2,3} \draw[->,dashed,blue!65!black] (T\j)--(C\j);
 \draw[->,dashed,black!45] (T4)--(C4);
 \draw[->,dashed,black!45] (U5)--(C5);
 \foreach \j in {6,7} \draw[->,dashed,violet!70!black] (U\j)--(C\j);
 \node[anchor=east,align=right] at (.35,.35) {terminal\\contacts};
\end{tikzpicture}
\caption{Index and contact relations for $(N,\kappa,\varphi,\Delta,q,h)=(17,10,7,3,2,1)$, with $\ell=3$ and $\sigma(j)=1+[j+2]_7$. The moving bases are $M=\{1,2,3\}$, their return sides are $\sigma(M)=\{4,5,6\}$, and the internal chain sides are $J=\{2,3\}$. Solid arrows join successive tower vertices under $F$; dashed arrows show the contacts reached at the first returns. Coloured towers move according to \eqref{eq:transported-motion}; the others stay fixed. The violet contacts have fixed sources and moving side lines, the blue contacts have moving sources and fixed side lines, and the orange contact has both its source and side line moving. The diagram records these index relations only; it does not assert the existence of an invariant polygon with $\Delta>1$.}
\label{fig:global-motion}
\end{figure}

Apply Lemma~\ref{lem:chain-holonomy} with $S=F^a|_{\Lambda_1}$. Denote the motion parameter by $\tau$ to distinguish it from the integer tower time $t$. Set $\widehat v_j(\tau)=X_i(\tau)$ when $v_j=X_i$ for $1\le i\le\ell$, and keep all other bases fixed. Since the tower lists partition the vertex indices, these choices extend uniquely to every vertex by
\begin{equation}\label{eq:transported-motion}
 \widehat v_{\pi(t,j)}(\tau)=F^t\widehat v_j(\tau)
 \qquad(0\le t<H_j).
\end{equation}
Within each tower, applying $F$ therefore gives the next moved vertex. Applying $F$ to the last vertex gives the returning image $F^{H_j}\widehat v_j(\tau)$. We must check that each such image remains in the polygon.

For an internal chain side, indexed by $J$, the base producing its contact is fixed, by \eqref{eq:network-disjointness}. Its returning image is therefore the original contact $C_i$, and the construction keeps $C_i$ on the line through $X_{i-1}(\tau),X_i(\tau)$.

For a moving base $X_i$ with $i\ge2$, the return side has index in $\sigma(M)\setminus\{k_*\}$. Its containing line stays fixed: it is neither an internal chain line nor the final line. Among the other sides carrying return contacts, the initial line remains $\Lambda_1$, and the remaining endpoints are fixed bases or $v_0$. Since $X_i(\tau)$ lies on the preimage line $L_i$, its returning image stays on its target line.

Every return side outside $J\cup\sigma(M)$ has both a fixed returning base and a fixed containing line, so its contact is unchanged. The only remaining return is from $X_1$. Its image is
\[
 Y(\tau)=F^aX_1(\tau),
\]
and its target side is $[X_\ell(\tau),X_{\ell+1}]$. Lemma~\ref{lem:chain-holonomy} places this image strictly on the inward side of the moving line for arbitrarily small nonzero $\tau$.

Define
\[
 P_\tau=\operatorname{conv}\{\widehat v_i(\tau):1\le i\le N\}.
\]
For sufficiently small $\tau$, the moved vertices remain distinct and form a convex $N$-gon in the same cyclic order. Every preserved contact remains strictly between its side endpoints, since it starts there and both the contact and endpoints vary continuously. The exceptional image $Y(\tau)$ starts strictly between the endpoints of the final side, so it remains strictly inside the inward half-planes of all the other sides for small $\tau$. Choose such a nonzero $\tau$ from Lemma~\ref{lem:chain-holonomy} to put it inward of the final side as well. Then $Y(\tau)\in\operatorname{int}P_\tau$.

Thus every vertex image belongs to $P_\tau$, so $F(P_\tau)\subseteq P_\tau$, and exactly one vertex image is strictly interior. That image is $Y(\tau)$, obtained by applying $F$ to the last vertex in the tower starting at $X_1$. Since $F$ is invertible, it is also a vertex of $F(P_\tau)$. This contradicts the assumed boundary contact property for every invariant polygon with at most $N$ vertices.
\end{proof}


\section{The product and its winding}
The index shift $i\mapsto i+\kappa$ modulo $N$ has $\delta=\gcd(N,\kappa)$ cycles, consisting of the residue classes modulo $\delta$. By Theorem~\ref{thm:no-skipping} and $\gcd(\Delta,\varphi)=\delta$ in \eqref{eq:tower-data}, $\varphi>\delta$ forces $\Delta=\delta=1$. Since $\varphi\ge\delta$, only two cases remain. When $\delta=1$, including the case of a single interior contact, we use the integer $q$ satisfying $q\kappa\equiv1\pmod N$ to build towers whose returns go to the next base, adding bases along paths of vertex images if needed. When $\delta\ge2$, we have $\varphi=\delta$, so each index cycle contains exactly one interior contact. The recurrence then links each base to the preceding one; conjugating and reversing the order puts it in the same form as in the first case.

In both cases we obtain relations between successive base vertices and a relation closing the list, then multiply them to eliminate the vertices. We also add the real angular increments along the same paths, using the continued vertex arguments from \eqref{eq:lift}. Consecutive differences of vertex arguments cancel in the sum, retaining the accumulated full turns. This real argument equality is needed to identify the intended boundary arc: the complex product alone determines an argument only modulo $2\pi$.

The cyclic recurrence with coefficients varying from side to side already appears in Karpelevi\v{c} \cite[Theorem~VI and equation~(13), pp.~375--376 of the Russian original]{Karp}. There the coefficients are first shown to be equal, and only then are the vertices eliminated. Here we allow the coefficients to differ while deriving the product. Equality of the coefficients will follow from the later convexity argument that determines the boundary modulus.

We will use the following Farey criterion: fractions $p/q<r/s$ in $F_N$ are consecutive if and only if
\[
 rq-ps=1,\qquad q+s>N.
\]
We prove both directions, keeping track of the denominator of any intermediate fraction.

Suppose first that the fractions are consecutive. Put $U=(q,p)$, $V=(s,r)$, and $D=\det(U,V)=rq-ps>0$. Since $p/q$ is reduced, B\'ezout's identity gives integers $u,v$ with $qv-pu=1$. Set $W=(u,v)$, so $\det(U,W)=1$. Consequently every integer vector is an integer combination of $U,W$; in particular,
\[
 V=aU+DW\qquad\text{for some integer }a.
\]
Since $r/s$ is reduced, $\gcd(a,D)=1$: a common divisor would divide both coordinates of $V$.

If $D>1$, then $a/D$ is not an integer. The point
\[
 Q=\left(\left\lceil\frac aD\right\rceil-\frac aD\right)U
       +\frac1D V
   =\left\lceil\frac aD\right\rceil U+W
\]
is therefore an integer vector whose two coefficients in terms of $U,V$ lie strictly between $0$ and $1$. If their sum exceeds $1$, replace $Q$ by $U+V-Q$. In either case we obtain
\[
 Q=\xi U+\eta V,\qquad \xi,\eta>0,\qquad \xi+\eta\le1.
\]
Write $Q=(d,c)$. The positive coefficients give $p/q<c/d<r/s$, while their sum bounds the denominator:
\[
 0<d=\xi q+\eta s\le\max(q,s)\le N.
\]
Reducing $c/d$ can only decrease its denominator, so this gives a fraction of $F_N$ strictly between the two consecutive fractions, a contradiction. Hence $D=1$. Also, if $q+s\le N$, the fraction $(p+r)/(q+s)$ lies strictly between them and has reduced denominator at most $N$, again a contradiction. Thus $q+s>N$.

Conversely, suppose $rq-ps=1$ and $q+s>N$. Any intermediate fraction $c/d$, with $d>0$, satisfies
\[
 (d,c)=(rd-sc)(q,p)+(qc-pd)(s,r).
\]
Both coefficients are positive integers because $p/q<c/d<r/s$. Comparing first coordinates gives $d\ge q+s>N$. No fraction of $F_N$ can therefore lie between the pair, so they are consecutive.

\Needspace{18\baselineskip}
\begin{theorem}[Product reduction]\label{thm:product}
Assume the extremality conditions \eqref{eq:extremal} with $N\ge4$ and nonreal $\zeta$, where $0<\rho=|\zeta|<1$. We can choose $\omega$ to be $\zeta$ or its conjugate so that the left Farey endpoint has the smaller denominator. Writing $\omega=\rho e^{i\vartheta}$ with $0<\vartheta<2\pi$, there are consecutive fractions in $F_N$ satisfying
\[
 \frac pq<\frac{\vartheta}{2\pi}<\frac rs,\qquad q<s.
\]
Set $m=\lfloor N/q\rfloor$ and $e=s-mq$. There are coefficients $0\le\beta_j<1$, with $\alpha_j=1-\beta_j>0$, such that
\begin{equation}\label{eq:product}
 \omega^e\prod_{j=1}^m(\omega^q-\beta_j)=\prod_{j=1}^m\alpha_j.
\end{equation}
Let $\Arg$ denote the principal argument in $(-\pi,\pi]$, and put
\[
 A=q\vartheta-2\pi p,\qquad M=\Arg(\omega^q-1),
 \qquad u_j=\Arg(\omega^q-\beta_j).
\]
Then $u_j\in[A,M)$, and the same factors satisfy the real argument identity
\begin{equation}\label{eq:product-phase}
 0<A<M<\pi,\qquad
 e\vartheta+\sum_{j=1}^m u_j=2\pi(r-mp).
\end{equation}
If $e<0$, multiplying \eqref{eq:product} by $\omega^{mq}$ removes the negative power and gives the equivalent polynomial equation
\[
 \omega^s\prod_{j=1}^m(\omega^q-\beta_j)
       =\omega^{mq}\prod_{j=1}^m\alpha_j.
\]
\end{theorem}
\begin{proof}
Use the polygon chosen above to minimize the number $\varphi$ of interior contacts, with contact indices $I=\{1,\ldots,\varphi\}$. Put $\delta=\gcd(N,\kappa)$. By Theorem~\ref{thm:no-skipping}, $\varphi>\delta$ forces $\Delta=1$, and then \eqref{eq:tower-data} gives $\delta=1$. Since $\varphi\ge\delta$, it remains to treat $\delta=1$ and $\varphi=\delta\ge2$. In both cases we multiply relations between successive vertices and add their real argument differences. Recall that $\Phi_{i+N}=\Phi_i+2\pi$ and that $\gamma_i=\Phi_i-\Phi_{i-1}\in(0,\pi)$ is the angular gap along side $i$.

\emph{Case $\delta=1$.} Now $\kappa$ and $N$ are coprime. Choose $1\le q<N$ and an integer $p$ such that $q\kappa-pN=1$. If $\varphi>1$, these are the $q,p$ from \eqref{eq:tower-data}: there $\Delta=1$, and $q\varphi+h=N$ gives $q<N$ and $\lfloor N/q\rfloor\ge\varphi$. When $\varphi=1$, the latter inequality is automatic. Set
\[
 m=\lfloor N/q\rfloor,\qquad e=N-mq,\qquad g=\kappa-mp.
\]
Then $m\ge\varphi$, $0\le e<q$, and
\[
 q\kappa-pN=1,\qquad e\kappa-gN=-m,\qquad qm+e=N.
\]
These identities verify the hypotheses of Lemma~\ref{lem:suspension} with $(\nu,\Delta,h)=(m,1,e)$. The bases are $v_1,\ldots,v_m$, and their indices include every index in $I$. For $2\le j\le m$, the return from $v_{j-1}$ has length $q$ and ends at $c_j$, on the side ending at $v_j$. The return from $v_m$ to $c_1$ has length $q+e$ and reaches $v_0$ after $e$ steps. Splitting this last path at $v_0$ and using the contact formula \eqref{eq:contacts} gives
\begin{equation}\label{eq:successor-product}
 (\zeta^q-\beta_j)v_{j-1}=\alpha_jv_j\quad(1\le j\le m),
 \qquad \zeta^ev_m=v_0.
\end{equation}
This also covers $m=1$. Any added bases $j>\varphi$ have $\beta_j=0$; they only subdivide paths whose images are vertices, without adding interior contacts.

The $q$-step path from $v_0$ ends strictly inside side $1$, and all its earlier images are vertices. Its index advance is $q\kappa=pN+1$. Keeping the integer indices in the argument lift \eqref{eq:path-lift} therefore gives
\[
 \Phi_0+q\theta_\zeta\in(\Phi_0+2\pi p,\Phi_1+2\pi p),
 \qquad 0<q\theta_\zeta-2\pi p<\gamma_1.
\]
Together with the upper bound in \eqref{eq:lift}, this yields
\[
 \frac pq<\frac{\theta_\zeta}{2\pi}<\frac\kappa N.
\]
Since $q\kappa-pN=1$ and $q+N>N$, the Farey criterion proved above shows that $p/q$ and $\kappa/N$ are consecutive in $F_N$. Thus take $\omega=\zeta$, $\vartheta=\theta_\zeta$, $r=\kappa$, and $s=N$; their denominators satisfy $q<s$.

Because $\alpha_j>0$, \eqref{eq:successor-product} identifies the principal argument of $\zeta^q-\beta_j$ with
\[
 u_j=\Phi_j-\Phi_{j-1}=\gamma_j\in(0,\pi).
\]
The closing path has index advance $m+e\kappa=gN$, so its real argument equality is
\[
 \Phi_m+e\theta_\zeta=\Phi_{gN}=\Phi_0+2\pi g.
\]
For $e=0$, this is the periodic extension of the vertex arguments. Multiplying the recurrences in \eqref{eq:successor-product} and using the closing relation cancels the nonzero vertices and gives \eqref{eq:product}. Adding the real differences gives
\[
 e\vartheta+\sum_{j=1}^m u_j
 =e\theta_\zeta+\Phi_m-\Phi_0
 =2\pi g=2\pi(r-mp).
\]

\emph{Case $\varphi=\delta\ge2$.} Put $L=N/\delta$ and $K=\kappa/\delta$, so $\gcd(K,L)=1$. Each cycle of the index shift has length $L$ and contains exactly one interior contact. The return from $v_i$ therefore ends at $c_i$ after $L$ steps. Choose $1\le h<L$ with $Kh\equiv-1\pmod L$, and set
\[
 b=\frac{Kh+1}{L},\qquad s_0=N-h,\qquad r_0=\kappa-b.
\]
The congruence makes $b$ an integer and gives $\delta+h\kappa=bN$. Hence the path from $v_\delta$ reaches $v_0$ after $h$ steps. Its indices stay in the residue class $0$ modulo $\delta$, whose only interior contact has index $\delta$. Since $h<L$, the path does not return to that contact, so every image along it is a vertex. Using the contact formula, we obtain
\begin{equation}\label{eq:orbit-product}
 (\zeta^L-\alpha_i)v_i=\beta_i v_{i-1}\quad(1\le i\le\delta),
 \qquad \zeta^h v_\delta=v_0.
\end{equation}
All these contacts are interior, so $\alpha_i,\beta_i\in(0,1)$.

The closing path and the $L$-step returns give the real argument relations
\[
 \Phi_\delta+h\theta_\zeta=\Phi_0+2\pi b,
 \qquad 0<2\pi K-L\theta_\zeta<\gamma_i\quad(1\le i\le\delta).
\]
Indeed, the index advance of each return is $L\kappa=KN$, and its final image lies strictly between $v_{i-1}$ and $v_i$. Put $\Gamma=\sum_{i=1}^\delta\gamma_i=\Phi_\delta-\Phi_0$. The displayed relations imply
\[
 s_0\theta_\zeta-2\pi r_0
 =\Gamma-\delta(2\pi K-L\theta_\zeta)>0,
 \qquad \frac{r_0}{s_0}<\frac{\theta_\zeta}{2\pi}<\frac KL.
\]
Here $1\le b\le K$, and the definitions give
\[
 Ks_0-r_0L=1,\qquad L+s_0=N+(L-h)>N.
\]
Thus $r_0/s_0$ and $K/L$ are reduced consecutive fractions in $F_N$.

The smaller denominator $L$ is on the right. Conjugation reflects this interval and puts it on the left; reversing the vertex list gives relations in the same order as \eqref{eq:successor-product}. Explicitly, set
\[
 \omega=\overline\zeta,\quad\vartheta=2\pi-\theta_\zeta,\quad
 W_j=\overline v_{\delta-j}\ (0\le j\le\delta),\quad
 (p,q,r,s)=(L-K,L,s_0-r_0,s_0).
\]
Since $\delta\ge2$ and $h<L$, we have $q=L<s_0=s$, $m=\delta$, and $e=-h$. With
\[
 \widetilde\beta_j=\alpha_{\delta-j+1},\qquad
 \widetilde\alpha_j=\beta_{\delta-j+1}\quad(1\le j\le m),
\]
conjugating \eqref{eq:orbit-product} gives
\[
 (\omega^q-\widetilde\beta_j)W_{j-1}=\widetilde\alpha_jW_j,
 \qquad \omega^{-h}W_m=W_0.
\]
The new coefficients satisfy $0<\widetilde\beta_j<1$ and $\widetilde\alpha_j=1-\widetilde\beta_j>0$. Reversal and conjugation give successive argument differences $u_j=\gamma_{\delta-j+1}$, so their sum is $\Gamma=2\pi b-h\theta_\zeta$. Consequently
\[
 e\vartheta+\sum_{j=1}^m u_j
 =-h(2\pi-\theta_\zeta)+2\pi b-h\theta_\zeta
 =2\pi(b-h)=2\pi(r-mp).
\]
Multiplying the recurrences and using the closing relation gives \eqref{eq:product}, with the tildes dropped.

In both cases, Farey adjacency gives $0<A<2\pi/s<\pi$; here $s\ge3$ because $q<s$, $q+s>N$, and $N\ge4$. Thus $\omega^q=\rho^q e^{iA}$ lies in the upper half-plane. Subtracting $t\in[0,1]$ moves this point horizontally to the left, so its principal argument increases strictly from $A$ to $M=\Arg(\omega^q-1)<\pi$. Since $0\le\beta_j<1$, each factor argument belongs to $[A,M)$, as required. If $e<0$, multiplying the product equation by $\omega^{mq}$ gives the stated polynomial equation because $e+mq=s$.
\end{proof}

\subsection*{The eight-state example, continued}
Return to the example in Section~\ref{sec:return-arithmetic}. Suppose an invariant polygon has eight vertices $v_0,\ldots,v_7$, with
\[
 \zeta v_i\in(v_{i+2},v_{i+3}],
\]
where indices are taken modulo $8$. Its only images lying strictly inside sides are $c_1\in(v_0,v_1)$ and $c_2\in(v_1,v_2)$; all other vertex images are vertices. This is the configuration specified by $N=8$, $\kappa=3$, and $\varphi=2$.

Here $q=3$, $e=2$, and $m=\lfloor8/3\rfloor=\varphi=2$. The product construction requires $m=2$ starting vertices. These are the same vertices $v_1,v_2$ from which the two earlier towers start; no further vertices need to be chosen. Each arrow denotes one multiplication by $\zeta$:
\[
 \begin{aligned}
 v_1&\longrightarrow v_4\longrightarrow v_7\longrightarrow c_2,\\
 v_2&\longrightarrow v_5\longrightarrow v_0\longrightarrow v_3
       \longrightarrow v_6\longrightarrow c_1.
 \end{aligned}
\]
Here a return means the first subsequent image assigned to $(v_0,v_1]$ or $(v_1,v_2]$, the sides ending at the chosen starting vertices. The intermediate vertex $v_0$ is excluded from $(v_0,v_1]$, so reaching it is not yet a return. The second path takes five steps to reach $c_1$, on the side ending at $v_1$. It reaches $v_0$ after two steps, leaving three steps to $c_1$. Using $c_j=\beta_jv_{j-1}+\alpha_jv_j$, with $\alpha_j=1-\beta_j$, the first path and the split at $v_0$ give
\[
 (\zeta^3-\beta_1)v_0=\alpha_1v_1,\qquad
 (\zeta^3-\beta_2)v_1=\alpha_2v_2,\qquad
 \zeta^2v_2=v_0.
\]
Multiplying these relations and cancelling the vertex factors gives
\[
 \zeta^2(\zeta^3-\beta_1)(\zeta^3-\beta_2)=\alpha_1\alpha_2.
\]
Set $u_j=\Phi_j-\Phi_{j-1}$ for $j=1,2$. The closing path gives
\[
 \Phi_2+2\theta_\zeta=\Phi_8=\Phi_0+2\pi.
\]
Adding these real argument differences therefore yields
\[
 2\theta_\zeta+u_1+u_2=2\pi.
\]
This records one full turn, agreeing with $r-mp=1$ for the endpoint data $(p,q,r,s)=(1,3,3,8)$ and the Farey interval $(1/3,3/8)$.

For this example, choose $\theta_\zeta=5\pi/7$. The angle parameters used in the next section are then
\[
 A=3\theta_\zeta-2\pi=\frac\pi7,\qquad
 B=\frac{6\pi-8\theta_\zeta}{2}=\frac\pi7.
\]
The equation for $\rho=|\zeta|$ derived there reduces to
\[
 \rho^4+\rho^3=2\cos(\pi/7).
\]
Its left-hand side increases strictly from $0$ to $2$ on $[0,1]$, so there is a unique solution $\rho\in(0,1)$. Set
\[
 \beta_1=\beta_2=\beta=\frac1{1+\rho},\qquad
 \alpha_1=\alpha_2=\alpha=\frac{\rho}{1+\rho}.
\]
For $\zeta=\rho e^{5\pi i/7}$, the equation $\rho^3(1+\rho)=2\cos(\pi/7)$ gives
\[
 \zeta^3-\beta=\beta e^{2\pi i/7}.
\]
Thus $u_1=u_2=2\pi/7$. The product has modulus $\rho^2\beta^2=\alpha^2$ and total argument $2\theta_\zeta+u_1+u_2=2\pi$, so both identities hold. Numerically, $\rho\simeq0.97061308$; this approximation is not used in the proof.

We can realize the same number $\zeta$ as an eigenvalue of an $8\times8$ row-stochastic matrix $A$. Index its rows and columns by $0,\ldots,7$. An arrow $i\xrightarrow{a}j$ means that $A_{ij}=a$, and all unlisted entries are zero. Give each arrow in the paths
\[
 0\longrightarrow3\longrightarrow6,\qquad
 1\longrightarrow4\longrightarrow7,\qquad
 2\longrightarrow5\longrightarrow0
\]
weight $1$, and set the remaining transitions to
\[
 6\xrightarrow{\beta}0,\quad6\xrightarrow{\alpha}1,
 \qquad
 7\xrightarrow{\beta}1,\quad7\xrightarrow{\alpha}2.
\]
Every row has either one entry equal to $1$ or two entries $\alpha,\beta$ whose sum is $1$, so $A$ is row-stochastic.

To construct an eigenvector $v=(v_0,\ldots,v_7)^T$, set
\[
 v_0=1,\qquad v_1=\frac{\zeta^3-\beta}{\alpha},
 \qquad v_2=\zeta^{-2}.
\]
The product identity above ensures $(\zeta^3-\beta)v_1=\alpha v_2$; the other two recurrences follow directly. Define the remaining coordinates along the arrows of weight $1$ by $v_j=\zeta v_i$. The six corresponding rows then satisfy $(Av)_i=\zeta v_i$, while the recurrences $(\zeta^3-\beta)v_i=\alpha v_{i+1}$ for $i=0,1$ verify rows $6$ and $7$, respectively. Thus $Av=\zeta v$, and $v\ne0$ because $v_0=1$.

This places $\zeta$ in $\Theta_8$. The later convexity bound and Farey comparison, proved independently of this example, show that $\rho$ is the largest possible modulus at this argument. The eigenvector construction does not establish that its coordinates are all polygon vertices with the image assignment assumed in the earlier illustration.



\section{The bound on the modulus and its attainment}
The product reduction allows the coefficients $\beta_j$ to vary. Its real argument identity determines the average $m^{-1}\sum_j u_j$. For the factors normalized by $1-\beta_j$, we will show that the logarithm of the modulus is a strictly convex function of the argument $u_j$. Jensen's inequality states that the average of these logarithms is at least the function's value at the average argument. The resulting inequality gives an upper bound on $|\omega|$, with equality precisely when the $\beta_j$ are equal. We then construct stochastic matrices attaining the bound.

For fixed $\omega$, write
\[
 g_j=\frac{\omega^q-\beta_j}{1-\beta_j},\qquad
 u_j=\operatorname{Arg}g_j.
\]
Then $\log|g_j|$ is a strictly convex function of $u_j$. This calculation and Jensen's inequality were already used in \cite[Lemma~2.2 and Proposition~3.1]{VGNCycle} for matrices whose nonzero entries lie on a directed $n$-cycle and the diagonal. Here Theorem~\ref{thm:product} obtains the product and a real argument identity from an extremal invariant polygon in the general stochastic eigenvalue problem. Applying the same calculation to these identities gives a bound on $|\omega|$, with equality precisely when all $\beta_j$ agree. We include the short proof to keep the argument self-contained. The resulting equation agrees with the boundary description at a fixed argument in \cite[Theorems~1.2 and~4.3; Lemma~4.4]{KirklandLaffeySmigoc}. Stochastic realization and the comparison between orders will then show that this equation describes the boundary.

Let $f<g$ be consecutive fractions in $F_n\cap[0,1/2]$, and fix an angle $\theta$ with $x=\theta/(2\pi)\in(f,g)$. Set
\[
 (p/q,r/s,y)=
 \begin{cases}
 (f,g,x),&f\text{ has the smaller denominator},\\
 (1-g,1-f,1-x),&g\text{ has the smaller denominator}.
 \end{cases}
\]
In the chosen interval $(p/q,r/s)$, the lower endpoint $p/q$ has denominator $q$, and the upper endpoint $r/s$ has denominator $s>q$. In the second case, $y=1-x$ corresponds to complex conjugation: the angle becomes $2\pi-\theta$ and the modulus is unchanged. In both cases, $p/q$ and $r/s$ are consecutive in $F_n$, with
\[
 \frac pq<y<\frac rs,\qquad q<s,\qquad rq-ps=1.
\]
Define
\begin{equation}\label{eq:scalar-data}
 m=\lfloor n/q\rfloor,\qquad
 A=2\pi q\left(y-\frac pq\right),\qquad
 B=\frac{2\pi s}{m}\left(\frac rs-y\right).
\end{equation}
Thus $A$ is $q$ times the difference between $2\pi y$ and the left endpoint angle, while $B$ is $s/m$ times the difference between the right endpoint angle and $2\pi y$.

Set
\[
 t=\frac{y-p/q}{r/s-p/q}\in(0,1).
\]
As $y$ moves from the left endpoint to the right, $t$ increases from $0$ to $1$. Since $r/s-p/q=1/(qs)$, we obtain
\begin{equation}\label{eq:angle-budget}
 A=\frac{2\pi t}s,\qquad B=\frac{2\pi(1-t)}{mq},\qquad
 \frac{A+B}{2\pi}=\frac ts+\frac{1-t}{mq}.
\end{equation}
Hence $A,B>0$. For $n\ge3$, we have $s\ge3$ and $mq\ge2$, so
\[
 \frac{A+B}{2\pi}\le\frac t3+\frac{1-t}{2}<\frac12.
\]
Thus $A+B<\pi$.

\begin{definition}\label{def:candidate}
Let $n\ge3$, and suppose $\theta/(2\pi)$ lies strictly between consecutive fractions in $F_n\cap[0,1/2]$. With $m,A,B$ defined by \eqref{eq:scalar-data}, let $K_n(\theta)$ be the unique solution $\rho\in(0,1)$ of
\begin{equation}\label{eq:radial}
 \rho^{s/m}\sin A+\rho^q\sin B=\sin(A+B).
\end{equation}
For $n\ge4$, set $K_n(\theta)=1$ whenever $\theta/(2\pi)\in F_n\cap[0,1/2]$. For $n=3$, we use the solutions inside the Farey intervals and leave the value at $\theta=\pi$, the negative real direction, unspecified.
\end{definition}
The left-hand side of \eqref{eq:radial} is continuous and strictly increasing for $\rho\ge0$. At $\rho=0$ it is zero; at $\rho=1$, its excess over the right-hand side is
\[
 \sin A+\sin B-\sin(A+B)
 =4\sin(A/2)\sin(B/2)\sin((A+B)/2)>0.
\]
Since $\sin(A+B)>0$, there is exactly one solution in $(0,1)$.

The solution varies continuously inside each Farey interval. Indeed, the roots lie in $[0,1]$, and any limit of roots for angles tending to an interior angle must satisfy the equation at that angle. Uniqueness determines this limit.

For $n\ge4$, we have $s\ge3$ and $mq\ge3$. At the endpoint $y=p/q$, $A\to0$ and $B\to2\pi/(mq)\in(0,\pi)$, so the limiting equation reduces to $\rho^q=1$. At $y=r/s$, $B\to0$ and $A\to2\pi/s\in(0,\pi)$, giving $\rho^{s/m}=1$. Thus $\rho\to1$ at both endpoints, and the assigned endpoint values make $K_n$ positive and continuous on $[0,\pi]$ for $n\ge4$.

For $n=3$, the solutions instead tend to $1/2$ as $\theta\uparrow\pi$. On the adjacent interval, $q=2$, $s=3$, $m=1$, $A=2\pi-2\theta$, and $B=3\theta-2\pi$. Dividing \eqref{eq:radial} by $\sin\theta$ and taking the limit gives $2\rho^3+3\rho^2=1$, whose unique solution in $[0,1]$ is $1/2$. Since $-1\in\Theta_3$, the maximal modulus at $\theta=\pi$ is $1$, explaining why no continuous endpoint value is assigned there.

\begin{theorem}[The bound on the modulus]\label{thm:convex-product}\label{thm:jensen}
Let $p,r\in\mathbb Z$ and let $q,s,m$ be positive integers with $q<s$ and $rq-ps=1$. Choose an angle $\vartheta$ satisfying
\[
 \frac pq<\frac{\vartheta}{2\pi}<\frac rs,
\]
and put
\[
 A=q\vartheta-2\pi p,\qquad
 B=\frac{2\pi r-s\vartheta}{m},\qquad e=s-mq.
\]
Assume $A+B<\pi$.

Let $\omega=\rho e^{i\vartheta}$ with $0<\rho<1$, and let $\beta_1,\ldots,\beta_m\in[0,1)$. Define
\[
 u_j=\operatorname{Arg}(\omega^q-\beta_j),\qquad
 M=\operatorname{Arg}(\omega^q-1),
\]
using arguments in $(0,\pi)$. These satisfy $A\le u_j<M<\pi$. Suppose the product relation and the real argument identity both hold:
\begin{equation}\label{eq:analytic-product-phase}
 \begin{gathered}
 \omega^e\prod_{j=1}^m(\omega^q-\beta_j)
       =\prod_{j=1}^m(1-\beta_j),\\
 e\vartheta+\sum_{j=1}^m u_j=2\pi(r-mp).
 \end{gathered}
\end{equation}
Then
\begin{equation}\label{eq:sharp-scalar-inequality}
 \rho^{s/m}\sin A+\rho^q\sin B\le\sin(A+B),
\end{equation}
with equality if and only if $\beta_1=\cdots=\beta_m$. Consequently $\rho\le\rho_*$, where $\rho_*$ is the unique solution in $(0,1)$ of \eqref{eq:radial} with the present values of $q,s,m,A,B$.
\end{theorem}
\begin{proof}
Write $w=\omega^q=\rho^qe^{iA}=a+ib$, so $b>0$ and $a<1$. Divide each factor by its positive coefficient and set
\[
 g_j=\frac{w-\beta_j}{1-\beta_j},\qquad c=\frac{1-a}{b}>0.
\]
The positive denominator leaves the argument unchanged, so $\operatorname{Arg}g_j=u_j$. Moreover,
\[
 \Re g_j+c\Im g_j
 =\frac{a-\beta_j+cb}{1-\beta_j}=1.
\]
Thus all $g_j$ lie on the same straight line. Define $\mathcal D(u)=\cos u+c\sin u$. Writing $g_j=|g_j|e^{iu_j}$ in the line equation gives
\[
 |g_j|\mathcal D(u_j)=1.
\]
From $w-1=(a-1)+ib$, we have $c=-\cot M$, so
\[
 \mathcal D(u)=\frac{\sin(M-u)}{\sin M}>0\qquad(0<u<M).
\]
Define $F(u)=-\log\mathcal D(u)$ on this interval. Then $F(u_j)=\log|g_j|$. All $u_j$, and hence their average, lie in this interval.

Since $\mathcal D''=-\mathcal D$, differentiating $F$ and $u(\beta)=\operatorname{Arg}(w-\beta)$ gives
\begin{equation}\label{eq:one-factor-convexity}
 \begin{gathered}
 F''(u)=1+\left(\frac{\mathcal D'(u)}{\mathcal D(u)}\right)^2\ge1,\\
 \frac{du}{d\beta}=\frac{\rho^q\sin A}{|\omega^q-\beta|^2}>0.
 \end{gathered}
\end{equation}
Thus $F$ is strictly convex, and $u(\beta)$ is strictly increasing.

The real argument identity gives
\[
 \bar u:=\frac1m\sum_{j=1}^m u_j
 =\frac{2\pi(r-mp)-e\vartheta}{m}=A+B.
\]
Dividing the product relation by $\prod_j(1-\beta_j)$ and taking logarithms of moduli gives
\[
 \sum_{j=1}^m F(u_j)=-e\log\rho.
\]
Jensen's inequality therefore yields
\[
 F(A+B)=F(\bar u)\le\frac1m\sum_{j=1}^mF(u_j)
 =-\frac em\log\rho.
\]
Since $F=-\log\mathcal D$, this is $\mathcal D(A+B)\ge\rho^{e/m}$. Now
\[
 \mathcal D(A+B)
 =\frac{\sin(A+B)-\rho^q\sin B}{\rho^q\sin A}.
\]
Multiplying by $\rho^q\sin A>0$ and using $q+e/m=s/m$ gives \eqref{eq:sharp-scalar-inequality}.

Equality in \eqref{eq:sharp-scalar-inequality} holds exactly when Jensen's inequality is an equality, which by strict convexity means that all $u_j$ agree. The strict increase of $u(\beta)$ makes this equivalent to $\beta_1=\cdots=\beta_m$. Finally, the left-hand side of \eqref{eq:sharp-scalar-inequality} increases strictly with $\rho$ and equals the right-hand side at $\rho_*$, so $\rho\le\rho_*$.
\end{proof}


\subsection{Stochastic realization of the bound}
To attain the bound, we construct a matrix using the graph pattern of Johnson and Paparella \cite[Theorem~3.2]{JP}. We give the construction explicitly below, verify that the matrix is row-stochastic, and derive its characteristic polynomial directly. The graph contains $m$ cycles of length $q$, joined by connections forming a further cycle of length $s$. In cycle $j$, the edge returning to its starting vertex has weight $\beta_j$, while the connection to the next cycle has weight $1-\beta_j$. We allow the $\beta_j$ to vary independently. For the polynomial calculation, we group the determinant terms according to directed cycles with no vertices in common. Taking all $\beta_j$ equal will produce a stochastic eigenvalue attaining the bound. Appendix~\ref{app:matrix-family} describes the structure of this matrix family and its limits at the Farey endpoints.

Let $\rho=K_n(\theta)$ for an angle inside a Farey interval, and put $\omega=\rho e^{2\pi i y}$, using the endpoint choice and parameters above. Choose the common weights
\begin{equation}\label{eq:equal-parameters}
 \alpha=\frac{\rho^{s/m}\sin A}{\sin(A+B)},\qquad
 \beta=\frac{\rho^q\sin B}{\sin(A+B)}.
\end{equation}
Both quotients are positive, and \eqref{eq:radial} gives $\alpha+\beta=1$, so $\alpha,\beta\in(0,1)$. Comparing real and imaginary parts gives
\begin{equation}\label{eq:factor-resolution}
 \omega^q-\beta
   =\frac{\rho^q\sin A}{\sin(A+B)}e^{i(A+B)}
   =\alpha\rho^{q-s/m}e^{i(A+B)}.
\end{equation}
The positive coefficient in \eqref{eq:factor-resolution} shows that every factor $\omega^q-\beta$ has argument $A+B$. Moreover,
\[
 (s-mq)2\pi y+m(A+B)=2\pi(r-mp).
\]
Raising \eqref{eq:factor-resolution} to the $m$-th power and using this equality gives
\begin{equation}\label{eq:ito-attaining}
 \omega^s(\omega^q-\beta)^m=\alpha^m\omega^{mq}.
\end{equation}
Equation~\eqref{eq:ito-attaining} is the Ito equation, and the displayed argument equality gives the real identity in \eqref{eq:analytic-product-phase} with all coefficients equal. If $y=1-x$, conjugating $\omega$ restores the original argument $\theta$.

For $n\ge4$, fix a Farey interval with the parameters in \eqref{eq:scalar-data}. We show that, as $\vartheta=2\pi y$ increases from $2\pi p/q$ to $2\pi r/s$, $\beta(\vartheta)$ decreases strictly from $1$ to $0$. Thus each $\beta\in(0,1)$ corresponds to exactly one point of this candidate arc.

Write $\rho=\rho(\vartheta)$ for the solution of \eqref{eq:radial}, and define
\[
 \mathcal R(\rho,\vartheta)
 =\rho^{s/m}\sin A+\rho^q\sin B-\sin(A+B),
\]
where $A=q\vartheta-2\pi p$ and $B=(2\pi r-s\vartheta)/m$. Its derivative with respect to $\rho$ is
\[
 \mathcal R_\rho
 =\frac{s}{m}\rho^{s/m-1}\sin A+q\rho^{q-1}\sin B>0.
\]
The mean-value theorem, applied in $\rho$ and in $\vartheta$, together with continuity of $\rho(\vartheta)$, gives
\[
 \rho'=-\frac{\mathcal R_\vartheta}{\mathcal R_\rho}
\]
along the curve. This derivative is continuous, and repeated differentiation shows that $\rho$ is smooth. Formula~\eqref{eq:equal-parameters} then makes $\beta$ smooth as well. Primes below denote derivatives with respect to $\vartheta$.

Set $e=s-mq$ and
\[
 H(\omega,\beta)=\omega^e(\omega^q-\beta)^m-(1-\beta)^m.
\]
For $\omega\ne0$, this is complex differentiable even if $e<0$, and
\[
 H_\omega
 =\omega^{e-1}(\omega^q-\beta)^{m-1}
       (s\omega^q-e\beta).
\]
Because $\operatorname{Im}(\omega^q)=\rho^q\sin A>0$, both $\omega^q-\beta$ and $s\omega^q-e\beta$ are nonzero. Together with $\omega\ne0$, this gives $H_\omega\ne0$. At a root, the derivative of the Ito polynomial $\omega^{mq}H$ is $\omega^{mq}H_\omega\ne0$, so the root is simple: its multiplicity is one.

With $\omega(\vartheta)=\rho(\vartheta)e^{i\vartheta}$, equation~\eqref{eq:ito-attaining} gives $H(\omega(\vartheta),\beta(\vartheta))=0$. Differentiating yields
\[
 H_\omega\omega'+H_\beta\beta'=0,\qquad
 \omega'=e^{i\vartheta}(\rho'+i\rho)\ne0.
\]
The last inequality follows from $\rho>0$. If $\beta'=0$, the first identity would force $H_\omega\omega'=0$, a contradiction. Thus $\beta'$ never vanishes, and its continuity makes its sign constant on the interval.

Since $s,mq\ge3$, the endpoint limits $\rho\to1$ and \eqref{eq:equal-parameters} give $\beta\to1$ at $y=p/q$ and $\beta\to0$ at $y=r/s$. These limits force $\beta'<0$. Hence $\beta$ decreases strictly from $1$ to $0$, taking every value in $(0,1)$ exactly once. If $y=x$, then $\vartheta=\theta$. If $y=1-x$, then $\vartheta=2\pi-\theta$, so increasing $\vartheta$ means decreasing the original angle $\theta$.

\Needspace{14\baselineskip}
\begin{theorem}[Stochastic realization]
\label{thm:sparse-realization}\label{thm:independent-realization}
Let $m,q,s$ be positive integers with $s>(m-1)q$, and set $n_0=\max\{mq,s\}$. For each $\beta\in(0,1)$, there is a real row-stochastic matrix of order $n_0$ having every nonzero solution $\lambda$ of
\begin{equation}\label{eq:realizing-polynomial}
 \lambda^s(\lambda^q-\beta)^m=(1-\beta)^m\lambda^{mq}
\end{equation}
as an eigenvalue.

More generally, for any
\[
 \boldsymbol\beta=(\beta_0,\ldots,\beta_{m-1})\in[0,1]^m,
 \qquad \alpha_j=1-\beta_j,
\]
there is a row-stochastic matrix $A(\boldsymbol\beta)$ of order $n_0$ with
\begin{equation}\label{eq:independent-characteristic}
 \det(tI-A(\boldsymbol\beta))
 =t^{n_0-mq}\prod_{j=0}^{m-1}(t^q-\beta_j)
  -t^{n_0-s}\prod_{j=0}^{m-1}\alpha_j.
\end{equation}
Every entry is $0$, $1$, $\beta_j$, or $1-\beta_j$ for some $j$, so the matrix depends continuously on $\boldsymbol\beta$. Taking all $\beta_j=\beta$ gives the matrix in the first assertion.
\end{theorem}
\begin{proof}
We first construct the family $A(\boldsymbol\beta)$ and compute its characteristic polynomial. Taking all parameters equal will then give the eigenvalues in the first assertion.

A directed graph specifies the matrix entries. Its vertices index both the rows and the columns, and an arrow from $u$ to $v$ carries a nonnegative number $w_e$, called its weight, which contributes $w_e$ to row $u$, column $v$:
\[
 A_{uv}=\sum_{\text{arrows }e:u\to v}w_e.
\]
An empty sum is zero. If two arrows have the same starting and ending vertices, their weights are added in that one matrix entry. We will choose the weights leaving each vertex to sum to one. Thus each row distributes a total of one among its destination columns, and $A_{uv}$ can be read as the probability of moving from $u$ to $v$ in one step.

For each $j=0,\ldots,m-1$, create a block of vertices
\[
 v_{j,0},\ldots,v_{j,q-1}.
\]
The block after $m-1$ is block zero. Within block $j$, draw the arrows $v_{j,k}\to v_{j,k+1}$ with weight one for $0\le k<q-1$. At its last vertex $v_{j,q-1}$, draw one arrow back to the first vertex $v_{j,0}$ with weight $\beta_j$, and another to the first vertex of the following block with weight $\alpha_j=1-\beta_j$. These two numbers divide the row's unit total between the two arrows. The first arrow closes a cycle of $q$ steps within the block, whose weights multiply to $\beta_j$. Following the arrows labelled $\alpha_j$ instead joins the blocks in the order $0,1,\ldots,m-1,0$. We adjust only the arrow labelled $\alpha_{m-1}$, from the last block back to block zero, so that this latter cycle has length $s$.

If $s\le mq$, redirect that arrow to $v_{0,d}$, where $d=mq-s$, keeping its weight $\alpha_{m-1}$. The hypothesis $s>(m-1)q$ gives $0\le d<q$, so this vertex exists. The graph still has $mq$ vertices. The cycle through all blocks now omits the first $d$ vertices of block zero and therefore has length $mq-d=s$.

If $s>mq$, replace that arrow by a path containing $s-mq$ new vertices between $v_{m-1,q-1}$ and $v_{0,0}$. Its first arrow carries weight $\alpha_{m-1}$, and all subsequent arrows carry weight one. This adds $s-mq$ steps to the cycle through all blocks, giving length $s$ and a graph with $s$ vertices. Each new vertex has just one outgoing arrow of weight one.

In both cases, every row has either one outgoing weight equal to one or the two outgoing weights $\beta_j$ and $\alpha_j$, whose sum is one. Adding weights with the same destination preserves that sum. All entries are nonnegative, so $A$ is row-stochastic of order $n_0=\max\{mq,s\}$. Its entries are $0$, $1$, $\beta_j$, or $\alpha_j$, as asserted.

If $m=1$ and $s=q$, the two arrows from the last vertex of the only block both end at its first vertex. They therefore give the single matrix entry
\[
 A_{v_{0,q-1},v_{0,0}}=\beta_0+\alpha_0=1.
\]
The resulting matrix moves each vertex to the next and the last back to the first: it is the cyclic permutation matrix of order $q$. Its characteristic polynomial is
\[
 t^q-1=(t^q-\beta_0)-\alpha_0,
\]
which is \eqref{eq:independent-characteristic} in this case. In every other case the two arrows from the last vertex of each block have distinct destinations.

Figure~\ref{fig:realizing-graphs} illustrates the two constructions for $m=2$ and $q=3$, with common weights $\beta_j=\beta$ and $\alpha_j=\alpha$. Its label $\delta_0=mq-s$ describes the route from the last block into block zero: if $\delta_0\ge0$, the arrow ends at $v_{0,\delta_0}$; if $\delta_0<0$, the route contains $-\delta_0$ added vertices before reaching $v_{0,0}$.

\begin{figure}[htbp]
\centering
\begin{tikzpicture}[x=1.65cm,y=1.12cm,
  every node/.style={font=\small},
  vertex/.style={circle,draw,fill=white,inner sep=1.2pt,minimum size=7mm,
                 font=\scriptsize},
  edge/.style={->,>=stealth,thick},
  weight/.style={font=\small,fill=white,inner sep=1pt}]
\begin{scope}
  \node at (1,1.55) {$(a)\ s=5,\quad\delta_0=1$};
  \node[vertex] (a00) at (0,0) {$v_{0,0}$};
  \node[vertex] (a01) at (1,0) {$v_{0,1}$};
  \node[vertex] (a02) at (2,0) {$v_{0,2}$};
  \node[vertex] (a10) at (2,-2.1) {$v_{1,0}$};
  \node[vertex] (a11) at (1,-2.1) {$v_{1,1}$};
  \node[vertex] (a12) at (0,-2.1) {$v_{1,2}$};
  \draw[edge] (a00) -- (a01);
  \draw[edge] (a01) -- (a02);
  \draw[edge] (a10) -- (a11);
  \draw[edge] (a11) -- (a12);
  \draw[edge] (a02) to[out=130,in=50]
      node[weight,above] {$\beta$} (a00);
  \draw[edge] (a12) to[out=-50,in=-130]
      node[weight,below] {$\beta$} (a10);
  \draw[edge] (a02) -- node[weight,right] {$\alpha$} (a10);
  \draw[edge] (a12) -- node[weight,left] {$\alpha$} (a01);
\end{scope}
\begin{scope}[xshift=7.0cm]
  \node at (1,1.55) {$(b)\ s=7,\quad\delta_0=-1$};
  \node[vertex] (b00) at (0,0) {$v_{0,0}$};
  \node[vertex] (b01) at (1,0) {$v_{0,1}$};
  \node[vertex] (b02) at (2,0) {$v_{0,2}$};
  \node[vertex] (b10) at (2,-2.1) {$v_{1,0}$};
  \node[vertex] (b11) at (1,-2.1) {$v_{1,1}$};
  \node[vertex] (b12) at (0,-2.1) {$v_{1,2}$};
  \node[vertex] (bw) at (0,-1.05) {$w$};
  \draw[edge] (b00) -- (b01);
  \draw[edge] (b01) -- (b02);
  \draw[edge] (b10) -- (b11);
  \draw[edge] (b11) -- (b12);
  \draw[edge] (b02) to[out=130,in=50]
      node[weight,above] {$\beta$} (b00);
  \draw[edge] (b12) to[out=-50,in=-130]
      node[weight,below] {$\beta$} (b10);
  \draw[edge] (b02) -- node[weight,right] {$\alpha$} (b10);
  \draw[edge] (b12) -- node[weight,left] {$\alpha$} (bw);
  \draw[edge] (bw) -- (b00);
\end{scope}
\end{tikzpicture}
\caption{Graphs of the realizing matrices for $m=2$ and $q=3$, with common weights $\beta_j=\beta$ and $\alpha_j=\alpha=1-\beta$. Each arrow's label contributes that amount to the matrix entry with its starting vertex as row and its ending vertex as column; unlabelled arrows contribute one. In $(a)$, the arrow from $v_{1,2}$ labelled $\alpha$ ends at $v_{0,1}$, giving a cycle of five steps through both blocks in a graph with six vertices. In $(b)$, this route passes through the added vertex $w$, giving a cycle of seven steps in a graph with seven vertices. The arrows labelled $\beta$ close the two three-step cycles within the blocks.}
\label{fig:realizing-graphs}
\end{figure}

For the remaining cases, we compute the determinant by listing the possible simple directed cycles: closed paths following the arrows, with no repeated vertex except the start at the end. The product of the arrow weights along a cycle is the product of the corresponding matrix entries in the determinant expansion.

Every such cycle contains the last vertex of a block. If it uses the arrow from that vertex back to the first vertex of the same block, the arrows within the block bring it back to the same last vertex. It is therefore exactly that block's $q$-cycle, whose weights multiply to $\beta_j$. If it uses none of these arrows back to first vertices, it follows the paths joining successive blocks and is the single $s$-cycle, whose weights multiply to $\prod_j\alpha_j$. The $m$ cycles within blocks have no vertices in common, while the $s$-cycle contains the last vertex of every block and hence meets each of them. This also covers $q=1$, when the cycles within blocks are loops, and $m=1$, $s<q$, when the arrow labelled $\alpha_0$ closes a shorter cycle inside the only block.

In the Leibniz expansion of $\det(tI-A)$, select a collection of directed cycles with no vertices in common. A cycle of length $\ell$ contributes $\ell$ negative matrix entries and has permutation sign $(-1)^{\ell-1}$, giving one minus sign per cycle. A loop is selected by taking the term $-A_{uu}$ from $t-A_{uu}$, so it has the same sign. Every vertex outside the selected cycles contributes $t$. Hence a collection of $k$ cycles covering $v$ vertices contributes $(-1)^k t^{n_0-v}$ times the product of its arrow weights.

Any subset $J$ of the cycles within blocks can be selected, giving the total contribution
\[
 \sum_{J\subseteq\{0,\ldots,m-1\}}
 (-1)^{|J|}t^{n_0-q|J|}\prod_{j\in J}\beta_j
 =t^{n_0-mq}\prod_{j=0}^{m-1}(t^q-\beta_j).
\]
The $s$-cycle meets every cycle within a block, so it can only be selected alone. Its contribution is $-t^{n_0-s}\prod_j\alpha_j$. Adding these contributions gives \eqref{eq:independent-characteristic}. Zero arrow weights make the corresponding contributions vanish, so the formula holds for every parameter tuple in $[0,1]^m$. It is the full characteristic polynomial, including zero roots and their algebraic multiplicities.

Finally, take all $\beta_j=\beta\in(0,1)$ and denote the resulting matrix by $A_\beta$. For a nonzero solution $\lambda$ of \eqref{eq:realizing-polynomial}, multiplying the characteristic polynomial at $\lambda$ by $\lambda^{mq+s-n_0}$ gives
\[
 \lambda^{mq+s-n_0}\det(\lambda I-A_\beta)
 =\lambda^s(\lambda^q-\beta)^m-(1-\beta)^m\lambda^{mq}=0.
\]
Since $\lambda\ne0$, this implies $\det(\lambda I-A_\beta)=0$. The matrix $\lambda I-A_\beta$ is therefore singular, so there is a vector $x\in\mathbb C^{n_0}$, $x\ne0$, with $(\lambda I-A_\beta)x=0$, or equivalently $A_\beta x=\lambda x$. Thus $\lambda$ is an eigenvalue of $A_\beta$, proving the first assertion.
\end{proof}

For $n\ge4$, consider the Farey interval used above, with $m=\lfloor n/q\rfloor$. Farey adjacency gives $q+s>n$, while $mq\le n$ and $q,s\le n$. Therefore
\[
 s>n-q\ge(m-1)q,\qquad n_0=\max\{mq,s\}\le n.
\]
The realization theorem thus applies. With the weights in \eqref{eq:equal-parameters}, equation~\eqref{eq:ito-attaining} makes $\omega=K_n(\theta)e^{2\pi i y}\ne0$ an eigenvalue of a real row-stochastic matrix of order $n_0$. If $y=x$, then $\omega=K_n(\theta)e^{i\theta}$. If $y=1-x$, its conjugate has that value and is also an eigenvalue because the matrix is real. Adjoining an identity block if necessary gives order $n$ and retains this eigenvalue. The endpoint roots of unity are realized by cyclic permutation matrices of orders $q$ and $s$, extended to order $n$ in the same way. Consequently $K_n(\theta)e^{i\theta}\in\Theta_n$ on every ray for $n\ge4$.

Consider a radial maximum $z=\rho e^{i\theta}\in\Theta_n$ with $0<\rho<1$ and $0<\theta<\pi$. Let $N\ge4$ be the least order of a stochastic matrix having $z$ as an eigenvalue. Since $\Theta_N\subseteq\Theta_n$, $z$ is also maximal on its ray in $\Theta_N$. Theorems~\ref{thm:product} and \ref{thm:convex-product} give $\rho\le K_N(\theta)$. Theorem~\ref{thm:sparse-realization} realizes $K_N(\theta)e^{i\theta}$ in $\Theta_N$, giving the reverse inequality. Hence $\rho=K_N(\theta)$, and the equality condition forces all coefficients $\beta_j$ in the geometric product to agree. To identify the boundary at the prescribed order $n$, it remains to prove $K_N(\theta)\le K_n(\theta)$ for $N\le n$ directly from the scalar equations.

\section{Farey refinement and completion}
We compare candidate moduli at successive orders through two changes in the Farey data. On an interval whose endpoint fractions stay the same, the number $m$ of product factors may increase. An interval with endpoints $a/b<c/d$ may instead be split by inserting $(a+c)/(b+d)$, obtained by adding the endpoint numerators and denominators separately. We then compare the old and new scalar roots at the same angle on each of the two new open intervals. The first comparison follows from the strict Jensen inequality in Theorem~\ref{thm:convex-product}.

\begin{lemma}\label{lem:add-factor}
Let $p,r$ be integers and $q,s$ positive integers with $q<s$ and $rq-ps=1$. Fix $p/q<y<r/s$ and an integer $m\ge1$, and put
\[
 A=2\pi(qy-p),\qquad B_k=\frac{2\pi(r-sy)}{k}
 \quad(k=m,m+1).
\]
Assume $A+B_m<\pi$. For $k=m,m+1$, let $\rho_k\in(0,1)$ be the unique solution of
\[
 \rho_k^{s/k}\sin A+\rho_k^q\sin B_k=\sin(A+B_k).
\]
Then $\rho_{m+1}>\rho_m$.
\end{lemma}
\begin{proof}
Put $\rho=\rho_m$ and $\omega=\rho e^{2\pi i y}\ne0$. Formula~\eqref{eq:equal-parameters} gives
\[
 \beta=\frac{\rho^q\sin B_m}{\sin(A+B_m)}\in(0,1).
\]
We use this same value in all $m$ factors: $\beta_1=\cdots=\beta_m=\beta$. Equation~\eqref{eq:factor-resolution} then gives
\[
 \omega^{s-mq}(\omega^q-\beta)^m=(1-\beta)^m,
\]
and each factor $\omega^q-\beta$ has argument $A+B_m$.

Add the factor $\omega^q-0=\omega^q$, corresponding to a further parameter $\beta_{m+1}=0$. Its argument is $A$. Reducing the power of $\omega$ before the product by $q$ compensates for this addition:
\[
 \omega^{s-(m+1)q}(\omega^q-\beta)^m\omega^q
 =\omega^{s-mq}(\omega^q-\beta)^m=(1-\beta)^m.
\]
The added parameter contributes $1-0=1$ on the right. The arguments also satisfy the required real identity:
\[
 [s-(m+1)q]2\pi y+m(A+B_m)+A=2\pi[r-(m+1)p].
\]

Since $B_{m+1}<B_m$, we have $A+B_{m+1}<\pi$. Theorem~\ref{thm:convex-product} therefore applies with $m+1$ factors. The parameters $(\beta,\ldots,\beta,0)$ are not all equal because $\beta>0$, so its equality condition gives
\[
 \rho^{s/(m+1)}\sin A+\rho^q\sin B_{m+1}
 <\sin(A+B_{m+1}).
\]
The left-hand side increases strictly with $\rho$. It therefore reaches $\sin(A+B_{m+1})$ at a radius larger than $\rho_m$, giving $\rho_{m+1}>\rho_m$.
\end{proof}

To compare the roots after inserting $(p+r)/(q+s)$ between $p/q<r/s$, we evaluate each new scalar equation at the old root. We need its right-hand side to exceed its left-hand side; since the latter increases with the radius, this puts the new root above the old one. Below, for suitable distinct $a,b\in\mathbb C$ and a real angle $\chi$, these differences become positive multiples of
\[
 \operatorname{Im}\!\left(e^{i\chi}(\bar a-\bar b)(a^m-b^m)\right).
\]
The following lemma determines the sign from the derivative values along the segment joining $a$ and $b$. Geometrically, $\operatorname{Im}(\bar v w)=\operatorname{Re}(v)\operatorname{Im}(w)-\operatorname{Im}(v)\operatorname{Re}(w)$ is the determinant of the planar vectors $v,w$. For $v\ne0$, a positive value places $w$ to the left of the line through the origin directed by $v$, and a negative value places it to the right.

\begin{lemma}\label{thm:sector-difference}
Let $a,b\in\mathbb C$, $a\ne b$, and let $f$ be complex differentiable in an open neighborhood of the straight line segment $[a,b]$. Fix a real angle $\chi$. Suppose every derivative value, after rotation through $\chi$, has positive imaginary part:
\[
 \operatorname{Im}\bigl(e^{i\chi}f'(z)\bigr)>0
 \qquad\text{for every }z\in[a,b].
\]
Then
\[
 \operatorname{Im}\!\left(e^{i\chi}
          (\bar a-\bar b)(f(a)-f(b))\right)>0.
\]
\end{lemma}
\begin{proof}
Parametrize the segment from $b$ to $a$ by $z(t)=b+t(a-b)$, $0\le t\le1$. The chain rule gives
\[
 \frac{d}{dt}f(z(t))=(a-b)f'(z(t)).
\]
Apply the ordinary fundamental theorem of calculus separately to the real and imaginary parts. Multiplying the resulting identity by $\bar a-\bar b$ gives
\begin{equation}\label{eq:sector-difference}
 (\bar a-\bar b)(f(a)-f(b))
 =|a-b|^2\int_0^1 f'\bigl(b+t(a-b)\bigr)\,dt.
\end{equation}
After multiplication by $e^{i\chi}$, the imaginary part of the right-hand side is
\[
 |a-b|^2\int_0^1
 \operatorname{Im}\!\left(e^{i\chi}f'\bigl(b+t(a-b)\bigr)\right)\,dt>0.
\]
Indeed, $|a-b|^2>0$ and the integrand is positive at every point by hypothesis. This proves the claim.
\end{proof}
Below we use the lemma only for $f(z)=z^m$. The integrand is then $m(b+t(a-b))^{m-1}$, a polynomial in the real variable $t$. This application therefore uses only the ordinary fundamental theorem of calculus.

\begin{corollary}\label{lem:positive-mediant}
Suppose $a/b<c/d$ are consecutive in $F_{n-1}\cap[0,1/2]$ and $b+d=n\ge4$. Put
\[
 \xi=\frac{a+c}{b+d}=\frac{a+c}{n},
\]
the fraction formed by adding the endpoint numerators and denominators separately. Then
\[
 K_n(2\pi x)>K_{n-1}(2\pi x)
 \qquad\text{for every }x\in(a/b,\xi)\cup(\xi,c/d).
\]
\end{corollary}
\begin{proof}
Fix $x$ in either new interval and let $\rho=K_{n-1}(2\pi x)$. If $b>d$, use the reflected data $1-x$ and $(d-c)/d<(b-a)/b$. This is the conjugate choice in the definition of the scalar root and leaves the radius unchanged. We may therefore assume $b<d$, using $a,b,c,d,x$ for the reflected values if necessary. Put
\[
 \xi=\frac{a+c}{n},\qquad
 m=\lfloor(n-1)/b\rfloor,\qquad
 A=2\pi(bx-a),\qquad B=2\pi(c-dx)/m.
\]
The old root $\rho\in(0,1)$ satisfies
\begin{equation}\label{eq:old-mediant-scalar}
 \rho^{d/m}\sin A+\rho^b\sin B=\sin(A+B),
 \qquad A,B>0,\quad A+B<\pi.
\end{equation}
Since $n=b+d$,
\[
 mB-A=2\pi\bigl((a+c)-nx\bigr).
\]
Thus $x<\xi$ is equivalent to $A<mB$, and $x>\xi$ to $A>mB$.

In the definition of $K_n$, the number of product factors is $\lfloor n/q\rfloor$, where $q$ is the smaller endpoint denominator. This is the largest integer $k$ with $kq\le n$.
For $a/b<x<\xi$, the endpoint denominators are $b,n$, so the definition requires $\lfloor n/b\rfloor$ factors. For the comparison, initially keep the old number $m$; we will adjust it to $\lfloor n/b\rfloor$ at the end. The two angles in its scalar equation are $A$ and $B-A/m$.
For $\xi<x<c/d$, conjugation places the denominator $d$ at the lower endpoint. The endpoint denominators are then $d,n$, and the factor count is $\lfloor n/d\rfloor=1$ because $n=b+d<2d$. Its angles are $mB$ and $A-mB$.
In their respective intervals, both angle pairs are positive and their sums satisfy
\[
 A+B-A/m<\pi,\qquad mB+(A-mB)=A<\pi.
\]
Both equations therefore have unique roots in $(0,1)$.

Let $D_L$ and $D_R$ be the right-hand side minus the left-hand side of these equations, evaluated at the old radius $\rho$:
\begin{align*}
 D_L&=\sin(A+B-A/m)-\rho^{n/m}\sin A
                            -\rho^b\sin(B-A/m),\\
 D_R&=\sin A-\rho^n\sin(mB)-\rho^d\sin(A-mB).
\end{align*}
Each left-hand side increases strictly with the radius. Hence $D_L>0$ on the first interval and $D_R>0$ on the second will place the corresponding new roots above $\rho$.

To determine these signs, put
\[
 \eta=A/m,\qquad u=\rho^{b/m},\qquad v=\rho^{d/m},
\]
and, for $z,w\in\mathbb C$, define
\[
 G_m(z,w)=(\bar z-\bar w)(z^m-w^m).
\]
The old equation becomes
\[
 v\sin A+u^m\sin B=\sin(A+B).
\]
In $D_L$, the mixed power is $\rho^{n/m}=uv$. Multiplying the old equation by $u$ gives
\[
 uv\sin A=u\sin(A+B)-u^{m+1}\sin B.
\]
Substitution eliminates $v$ and yields
\[
 D_L=\sin\bigl(B+(m-1)\eta\bigr)-u\sin(B+m\eta)
       -u^m\sin(B-\eta)+u^{m+1}\sin B.
\]
This is the imaginary part of
$e^{iB}(e^{-i\eta}-u)(e^{im\eta}-u^m)$.

For $D_R$, use $\rho^n=u^mv^m$ and $\rho^d=v^m$. Eliminate $u^m$ through
\[
 u^m\sin B=\sin(A+B)-v\sin A.
\]
Multiplying $D_R$ by the positive number $\sin B/\sin A$ and using
\[
 \sin(A+B)\sin(mB)+\sin B\sin(A-mB)
 =\sin A\sin\bigl((m+1)B\bigr)
\]
gives
\[
 \frac{\sin B}{\sin A}D_R
 =\sin B-v^m\sin\bigl((m+1)B\bigr)+v^{m+1}\sin(mB).
\]
This is the imaginary part of
$e^{iB}(1-ve^{-iB})(1-v^me^{imB})$; the additional term $-v$ in this product is real. We have therefore obtained
\begin{equation}\label{eq:mediant-differences}
 \begin{split}
 D_L&=\operatorname{Im}\!\left(e^{iB}G_m(e^{i\eta},u)\right),\\
 D_R&=\frac{\sin A}{\sin B}\,
       \operatorname{Im}\!\left(e^{iB}G_m(1,ve^{iB})\right).
 \end{split}
\end{equation}

These are the expressions in Lemma~\ref{thm:sector-difference} for $f(z)=z^m$ and rotation angle $\chi=B$. Their segments join $u$ to $e^{i\eta}$ and $ve^{iB}$ to $1$, respectively. Points on the first segment have arguments in $[0,\eta]$, and points on the second in $[0,B]$. Neither segment passes through zero, since each pair of endpoint directions differs by less than $\pi$.

For $x<\xi$, the rotated derivative $e^{iB}mz^{m-1}$ on the first segment has arguments in
\[
 [B,B+(m-1)\eta]\subset(0,\pi),
 \qquad B+(m-1)\eta=A+B-A/m<A+B<\pi.
\]
For $x>\xi$, its arguments on the second segment lie in
\[
 [B,mB]\subset(0,\pi),
 \qquad mB<A<\pi.
\]
Thus the derivative has positive imaginary part after rotation on each segment in the case where it is used. The lemma gives $D_L>0$ on the first interval and $D_R>0$ on the second; the factor $\sin A/\sin B$ is positive. The argument also covers $m=1$, when the derivative is constant. The corresponding new roots are therefore larger than $\rho$.

On the second interval, we have used exactly the $\lfloor n/d\rfloor=1$ factor required by the definition of $K_n$. On the first interval, that definition requires $\lfloor n/b\rfloor$ factors, while our comparison used $m=\lfloor(n-1)/b\rfloor$. These differ precisely when $b\mid n$, equivalently $b\mid d$. The endpoint denominators are coprime, so this occurs only when $b=1$. If $b>1$, the counts agree and the comparison is proved. If $b=1$, then $m=d$ and the required count is $d+1$. Apply Lemma~\ref{lem:add-factor} to the first interval, whose angle sum $A+B-A/m$ is below $\pi$. Increasing its count from $m$ to $m+1$ increases its root further, so the comparison holds in this case as well.
\end{proof}

\begin{theorem}[Monotonicity in the order]\label{thm:order-monotonicity}
Let $n\ge5$, $0\le\theta\le\pi$, and $x=\theta/(2\pi)$. Then
\[
 K_{n-1}(\theta)\le K_n(\theta).
\]
If $x\in F_{n-1}$, both values equal $1$. Otherwise, let $a/b<c/d$ be the consecutive fractions of $F_{n-1}$ with $a/b<x<c/d$, and put $q=\min\{b,d\}$. Equality holds precisely when these endpoints remain consecutive in $F_n$ and their numbers of product factors agree:
\[
 \lfloor(n-1)/q\rfloor=\lfloor n/q\rfloor.
\]
\end{theorem}
\begin{proof}
Put $x=\theta/(2\pi)$. Every fraction newly added to $F_n$ has reduced denominator $n$. Let $a/b<c/d$ be consecutive in $F_{n-1}$. The Farey criterion proved earlier gives $bc-ad=1$ and $b+d>n-1$, hence $b+d\ge n$.
If a new fraction $h/n$ lies between them, the inequalities $a/b<h/n<c/d$ make $bh-an$ and $cn-dh$ positive integers. Since $bc-ad=1$, we have
\[
 (n,h)=(cn-dh)(b,a)+(bh-an)(d,c).
\]
Comparing first coordinates gives
\[
 n=(cn-dh)b+(bh-an)d\ge b+d\ge n.
\]
Thus $b+d=n$, both coefficients equal $1$, and the second coordinate gives $h=a+c$. Each old interval therefore either remains unchanged or is split by the sole inserted fraction $(a+c)/(b+d)$, formed by adding its endpoint numerators and denominators separately.

Suppose $x$ lies in an interval whose endpoints remain consecutive in $F_n$, and put $q=\min\{b,d\}$. The numbers of product factors used to define $K_{n-1}$ and $K_n$ are $\lfloor(n-1)/q\rfloor$ and $\lfloor n/q\rfloor$. They are equal or the second exceeds the first by one. If they agree, the scalar equations are identical and their roots agree. If the number increases, Lemma~\ref{lem:add-factor} gives a strictly larger root.

If the interval containing $x$ is split and $x$ is not the inserted fraction, Corollary~\ref{lem:positive-mediant} gives $K_{n-1}(\theta)<K_n(\theta)$ on either of the two new open intervals.

Finally, if $x\in F_{n-1}$, both values are $1$ by definition. If $x\in F_n\setminus F_{n-1}$, then $K_n(\theta)=1$, whereas $K_{n-1}(\theta)<1$ because $x$ lies inside an old interval.
\end{proof}

\subsection{Small orders and the boundary}
\begin{proposition}\label{prop:low-orders}
Let $\xi=e^{2\pi i/3}$. Then $\Theta_1=\{1\}$, $\Theta_2=[-1,1]$, and
\begin{equation}\label{eq:order-three}
 \Theta_3=\operatorname{conv}\{1,\xi,\bar\xi\}\ \cup\ [-1,-1/2].
\end{equation}
For $0\le\theta\le\pi$, define its radial maximum by
\[
 R_3(\theta)=\max\{r\ge0:re^{i\theta}\in\Theta_3\}.
\]
For $0<\theta<\pi$, this is the scalar root in \eqref{eq:radial} for $n=3$, taking the value $1$ at $\theta=2\pi/3$. Moreover,
\[
 R_3(\theta)\le K_4(\theta)\qquad(0\le\theta\le\pi).
\]
\end{proposition}
\begin{proof}
The only stochastic matrix of order one is $[1]$. A stochastic matrix of order two has the form
\[
 \begin{pmatrix}a&1-a\\ b&1-b\end{pmatrix},\qquad a,b\in[0,1],
\]
and its eigenvalues are $1$ and $a-b$. This gives $\Theta_1=\{1\}$ and $\Theta_2=[-1,1]$.

Let $\lambda=u+iv$, $v\ne0$, be an eigenvalue of a stochastic matrix $A=(a_{ij})$ of order three. Its spectrum is $1,\lambda,\bar\lambda$, so
\[
 1+2u=\operatorname{tr}A\ge0.
\]
Also,
\[
 \operatorname{tr}(A^2)=\sum_i a_{ii}^2+\sum_{i\ne j}a_{ij}a_{ji}
 \ge\sum_i a_{ii}^2.
\]
Cauchy--Schwarz therefore gives
\[
 (1+2u)^2\le3\sum_i a_{ii}^2\le3\operatorname{tr}(A^2)
 =3(1+2u^2-2v^2).
\]
Thus $u\ge-1/2$ and $3v^2\le(1-u)^2$. Since $|\lambda|\le1$, we also have $u\le1$. These are exactly the inequalities defining the triangle with vertices $1,\xi,\bar\xi$.

Conversely, let $a,b,c\ge0$ with $a+b+c=1$. If $C_3$ is a three-cycle permutation matrix, then $aI+bC_3+cC_3^2$ is row-stochastic. Since $C_3$ has eigenvalue $\xi$, this matrix has eigenvalue $a+b\xi+c\bar\xi$, realizing every point of the triangle. Real stochastic eigenvalues lie in $[-1,1]$, and every point of this interval is realized at order three by adjoining a $1\times1$ identity block to an order-two matrix. Since the triangle already contains $[-1/2,1]$, this proves the stated description of $\Theta_3$.

For angles inside the two Farey intervals, put $x=\theta/(2\pi)$ and let $\rho$ be the scalar root for $n=3$. Use $\omega=\rho e^{2\pi ix}$ for $0<x<1/3$, and $\omega=\rho e^{2\pi i(1-x)}$ for $1/3<x<1/2$. Choose $\alpha,\beta\in(0,1)$ by \eqref{eq:equal-parameters}, so $\alpha+\beta=1$. Since $\omega\ne0$, we can cancel powers of $\omega$ in \eqref{eq:ito-attaining}.
For $0<x<1/3$, the parameters are $q=1$ and $s=m=3$, and the equation reduces to
\[
 (\omega-\beta)^3=\alpha^3.
\]
Its root in the upper half-plane is $\omega=\beta+\alpha\xi$, on the side from $1$ to $\xi$.
For $1/3<x<1/2$, conjugation replaces $x$ by $1-x\in(1/2,2/3)$, giving $q=2$, $s=3$, and $m=1$. The reduced equation is
\[
 \omega^3-\beta\omega-\alpha
 =(\omega-1)(\omega^2+\omega+\alpha)=0.
\]
Its nonreal roots have real part $-1/2$. Restoring the upper-half-plane point gives
\[
 z=-\frac12+\frac{i}{2}\sqrt{4\alpha-1},
\]
which traces the vertical side from $-1/2$ to $\xi$ as $\alpha$ increases from $1/4$ to $1$.
Equation~\eqref{eq:factor-resolution} and uniqueness of the scalar root identify its modulus with $R_3(\theta)$: in either angle range, the indicated side is precisely where the ray leaves the triangle. At $\theta=2\pi/3$, the maximum is $1$. In particular,
\[
 \lim_{\theta\uparrow\pi}R_3(\theta)=\frac12,
 \qquad R_3(\pi)=1,
\]
because the extra real interval contains $-1$.

To compare with $K_4$, insert $1/4$ between $0$ and $1/3$. Corollary~\ref{lem:positive-mediant} gives the comparison on both resulting open intervals. On $(1/3,1/2)$, the endpoints stay fixed and the smaller denominator is $2$, so the number of factors increases from $\lfloor3/2\rfloor=1$ to $\lfloor4/2\rfloor=2$. Lemma~\ref{lem:add-factor} gives the comparison there; its angle condition follows from \eqref{eq:angle-budget}. At $x=1/4$, $K_4=1$ and $R_3<1$ by the triangle description. At $x=0,1/3,1/2$, both maxima equal $1$.
\end{proof}

The points of $\Theta_n$ on the unit circle are exactly the roots of unity of orders at most $n$. A permutation matrix that cyclically permutes $\ell$ coordinates has all $\ell$th roots of unity as eigenvalues. For $\ell\le n$, adjoining an identity block if needed realizes these eigenvalues at order $n$.

Conversely, let $A=(a_{ij})$ be a real row-stochastic matrix of order $n$ and let $Av=\lambda v$, with $v\ne0$ and $|\lambda|=1$. Put
\[
 M=\max_j|v_j|>0,\qquad S=\{j:|v_j|=M\}.
\]
For every $i\in S$,
\[
 M=|\lambda v_i|=\left|\sum_j a_{ij}v_j\right|
 \le\sum_j a_{ij}|v_j|\le M\sum_j a_{ij}=M.
\]
Both inequalities are therefore equalities. If $a_{ij}>0$, the second forces $|v_j|=M$, and equality in the triangle inequality then gives $v_j=\lambda v_i$.
Draw an arrow $i\to j$ whenever $a_{ij}>0$. Each index in $S$ has an outgoing arrow, and every such arrow stays in $S$. Following arrows in this finite set eventually repeats an index, producing a directed cycle of length $\ell\le|S|\le n$. Along that cycle, $v_i=\lambda^\ell v_i$. Since $v_i\ne0$, we obtain $\lambda^\ell=1$.

\begin{proof}[Completion of Theorem~\ref{thm:karpelevic}]
Fix $n\ge4$ and $0\le\theta\le\pi$, and put $x=\theta/(2\pi)$. Write
\[
 R_n(\theta)=\max\{r\ge0:re^{i\theta}\in\Theta_n\}.
\]
If $x\in F_n$, the unit-circle classification gives $R_n(\theta)=1=K_n(\theta)$. Now suppose $x\notin F_n$. The stochastic realization in Theorem~\ref{thm:sparse-realization} gives $K_n(\theta)e^{i\theta}\in\Theta_n$, so
\[
 K_n(\theta)\le R_n(\theta).
\]

Put $z=R_n(\theta)e^{i\theta}$, and let $k\le n$ be the smallest order of a stochastic matrix having $z$ as an eigenvalue. Since $K_n(\theta)>0$, we have $z\ne0$. Also, $0<\theta<\pi$ and the unit-circle classification gives $|z|<1$. Thus $z$ is nonreal and $k\ge3$. Because $\Theta_k\subseteq\Theta_n$, $z$ is also maximal at this angle in $\Theta_k$.

If $k=3$, Proposition~\ref{prop:low-orders} and Theorem~\ref{thm:order-monotonicity} give
\[
 |z|\le R_3(\theta)\le K_4(\theta)\le K_n(\theta)\le|z|.
\]
Hence $R_n(\theta)=K_n(\theta)$ in this case.

Suppose $k\ge4$. By the invariant-polygon criterion \eqref{eq:polygon-criterion}, $k$ is the smallest number of vertices of a polygon $P$ with $zP\subseteq P$. Moreover, for any real $t>1$, a polygon with at most $k$ vertices invariant under multiplication by $tz$ would give $tz\in\Theta_k\subseteq\Theta_n$, contradicting the definition of $R_n(\theta)$. These are precisely the two extremality conditions in \eqref{eq:extremal}, with $N=k$ and $\zeta=z$.
The product reduction, Theorem~\ref{thm:product}, therefore applies. The bound on the modulus in Theorem~\ref{thm:convex-product} gives $|z|\le K_k(\theta)$. Conversely, stochastic realization at order $k$ gives $K_k(\theta)\le|z|$, since $z$ is maximal in $\Theta_k$. Thus $|z|=K_k(\theta)$, and monotonicity gives
\[
 |z|=K_k(\theta)\le K_n(\theta)\le|z|.
\]
We have proved $R_n(\theta)=K_n(\theta)$ for every $0\le\theta\le\pi$.

Extend $K_n$ to $[-\pi,\pi]$ by $K_n(-\theta)=K_n(\theta)$. Conjugation symmetry gives the radial maxima in the lower half-plane as well. Since $\Theta_n$ contains the whole segment from zero to each of its points,
\[
 \Theta_n=\{re^{i\theta}:-\pi\le\theta\le\pi,\ 0\le r\le K_n(\theta)\}.
\]
The positivity and continuity of $K_n$ show that its boundary is exactly
\[
 \partial\Theta_n=\{K_n(\theta)e^{i\theta}:-\pi\le\theta\le\pi\}.
\]
Equations~\eqref{eq:ito-attaining} and \eqref{eq:factor-resolution} give the stated Ito equation on each arc. The coefficient takes every value in $[0,1]$ continuously along each arc, as proved above, and the endpoint limits join the corresponding roots of unity. Proposition~\ref{prop:low-orders} gives the remaining orders, including the additional real segment $[-1,-1/2]$ for $n=3$.
\end{proof}


\section{Measuring contraction with polygons}
\label{sec:gauges}

Invariant polyhedra and functions that measure vector size by scaling them are standard tools in stability and control \cite{Bitsoris,FarinaBenvenuti,Blanchini}. Hennet and Lasserre describe when a rotation followed by a dilation leaves a centered regular polygon invariant \cite[Lemma~2]{HennetLasserre}. Regular polygons also appear in constructions of polyhedral Lyapunov functions for planar systems with nonreal eigenvalues \cite{RossiColaneriShorten}. Kousoulidis and Forni study such constructions with a fixed number of vertices or defining linear inequalities \cite{KousoulidisForni}. Here the boundary theorem gives the best possible contraction factor when vector size is measured by scaling a polygon with a prescribed maximum number of vertices. The polygon need not be regular or symmetric.

Let $R_\theta$ denote planar rotation through the real angle $\theta$, and put $T_{\rho,\theta}=\rho R_\theta$, where $\rho>0$. Fix $N\ge4$. We allow any convex polygon $P$ with at most $N$ vertices and $0\in\operatorname{int}P$.

Measure the size of $x\in\mathbb R^2$ by
\[
 V_P(x)=\inf\{a>0:x\in aP\}.
\]
This measures how much we must scale $P$ to contain $x$. We impose no symmetry on $P$; in particular, $V_P(-x)$ may differ from $V_P(x)$, so $V_P$ need not be a norm.
For $\gamma>0$, the inclusion $T_{\rho,\theta}P\subseteq\gamma P$ is equivalent to
\[
 V_P(T_{\rho,\theta}x)\le\gamma V_P(x)
 \qquad(x\in\mathbb R^2).
\]
A factor $\gamma<1$ therefore means contraction in this measurement. Define the best possible factor over all the allowed polygons by
\[
 \Gamma_N(T_{\rho,\theta})
 =\inf\{\gamma>0:T_{\rho,\theta}P\subseteq\gamma P
                 \text{ for some such }P\}.
\]
For $0\le\theta\le\pi$, Theorem~\ref{thm:karpelevic} gives the exact boundary radius
\[
 K_N(\theta)=\max\{r\ge0:re^{i\theta}\in\Theta_N\}.
\]
The next theorem expresses the optimal factor in terms of this radius.

\Needspace{20\baselineskip}
\begin{theorem}[Optimal contraction factor]\label{thm:gauges}
For $N\ge4$, $\rho>0$, and $0\le\theta\le\pi$,
\begin{equation}\label{eq:gauge-optimum}
 \Gamma_N(T_{\rho,\theta})=\frac{\rho}{K_N(\theta)}.
\end{equation}
The infimum is attained by a convex polygon $P$ with at most $N$ vertices and $0\in\operatorname{int}P$. For this polygon,
\[
 V_P(T_{\rho,\theta}x)\le
 \Gamma_N(T_{\rho,\theta})V_P(x)\qquad(x\in\R^2).
\]
As the angle varies, the smallest boundary radius and the largest optimal factor relative to $\rho$ are
\begin{equation}\label{eq:uniform-gauge}
 \min_{0\le\theta\le\pi}K_N(\theta)=\cos(\pi/N),\qquad
 \max_{0\le\theta\le\pi}
       \frac{\Gamma_N(T_{\rho,\theta})}{\rho}=\frac{1}{\cos(\pi/N)}.
\end{equation}
For angles between $0$ and $2\pi/N$, the boundary radius is
\begin{equation}\label{eq:first-arc}
 K_N(\theta)=\frac{\cos(\pi/N)}{\cos(\theta-\pi/N)}
 \qquad(0\le\theta\le2\pi/N).
\end{equation}
Both extrema in \eqref{eq:uniform-gauge} are attained at $\theta=\pi/N$.
\end{theorem}
\begin{proof}
Suppose an allowed polygon $P$ satisfies $\rho R_\theta P\subseteq\gamma P$. Dividing this inclusion by $\gamma$ gives
\[
 (\rho/\gamma)R_\theta P\subseteq P.
\]
Identifying $\mathbb R^2$ with $\mathbb C$, the map on the left is multiplication by $(\rho/\gamma)e^{i\theta}$. The invariant-polygon criterion \eqref{eq:polygon-criterion} therefore gives $(\rho/\gamma)e^{i\theta}\in\Theta_N$, and hence
\[
 \rho/\gamma\le K_N(\theta),\qquad
 \gamma\ge\frac{\rho}{K_N(\theta)}.
\]
Taking the infimum over all allowed polygons and factors proves the lower bound in \eqref{eq:gauge-optimum}.

For the reverse bound, we construct an allowed polygon invariant under multiplication by the boundary point $\lambda=K_N(\theta)e^{i\theta}$. If $\lambda$ is nonreal and $|\lambda|<1$, the criterion supplies a polygon with at most $N$ vertices. The observation following \eqref{eq:polygon-criterion} shows that it has zero in its interior. If $\lambda$ is nonreal and $|\lambda|=1$, the unit-circle classification makes it a root of unity of order $d\le N$. Take
\[
 P=\operatorname{conv}\{1,\lambda,\ldots,\lambda^{d-1}\}.
\]
This is a regular $d$-gon with zero in its interior, and $\lambda P=P$. Finally, at $\theta=0$ or $\theta=\pi$, we have $K_N(\theta)=1$; a square centered at zero is invariant under multiplication by $1$ or $-1$ and has four vertices, within the allowed bound.
In every case,
\[
 K_N(\theta)R_\theta P\subseteq P,
 \qquad
 T_{\rho,\theta}P\subseteq\frac{\rho}{K_N(\theta)}P.
\]
Thus the lower bound is attained. By the equivalence established above between this inclusion and the inequality for $V_P$, the same polygon satisfies the stated size bound.

To bound the radius for every angle, let $Q_N$ be the regular polygon whose vertices are the $N$th roots of unity. If $z\in Q_N$, multiplying a convex combination representing $z$ by any vertex of $Q_N$ merely permutes these roots. Thus $z$ times each vertex lies in $Q_N$, and convexity gives $zQ_N\subseteq Q_N$. The polygon criterion yields $Q_N\subseteq\Theta_N$. Since $Q_N$ contains the disk centered at zero with radius $\cos(\pi/N)$,
\[
 K_N(\theta)\ge\cos(\pi/N)\qquad(0\le\theta\le\pi).
\]

For $0<\theta<2\pi/N$, the corresponding Farey interval is from $0/1$ to $1/N$. Equation~\eqref{eq:radial} has $q=1$, $s=m=N$, $A=\theta$, and $B=2\pi/N-\theta$, so it becomes
\[
 K_N(\theta)\bigl(\sin\theta+\sin(2\pi/N-\theta)\bigr)
 =\sin(2\pi/N).
\]
The identity for a sum of sines gives
\[
 K_N(\theta)=
 \frac{\sin(2\pi/N)}{\sin\theta+\sin(2\pi/N-\theta)}
 =\frac{\cos(\pi/N)}{\cos(\theta-\pi/N)}.
\]
Continuity extends this formula to both endpoints. At $\theta=\pi/N$, the radius equals $\cos(\pi/N)$, so the bound is attained. Since $\Gamma_N(T_{\rho,\theta})/\rho=1/K_N(\theta)$, the largest optimal factor relative to $\rho$ is $1/\cos(\pi/N)$, attained at the same angle.
\end{proof}

A planar polygon has as many sides as vertices. If it has $L\le N$ sides and zero in its interior, write its defining inequalities as $\ell_j(x)\le1$, where each $\ell_j$ is linear. Then
\[
 V_P(x)=\max_{1\le j\le L}\ell_j(x).
\]
Thus the vertex bound also limits the number of linear functions needed to represent $V_P$. By the theorem, such a polygon can be chosen to give a contraction factor smaller than one exactly when
\[
 \rho<K_N(\theta).
\]
If the polygon may depend on the angle, strict contraction with at most $N$ vertices is possible for every angle exactly when $\rho<\cos(\pi/N)$.
The assumption $N\ge4$ is needed in the negative real direction. Although $-1\in\Theta_3$, the inclusion $-P\subseteq P$ also gives $P\subseteq-P$, so $P=-P$. Its vertices then occur in opposite pairs, which is impossible for a triangle with zero in its interior.

\section{Farey asymptotics and vertex complexity}
\label{sec:asymptotic}

The formulas above give the boundary radius and optimal contraction factor at each order. We now estimate their behavior as $N\to\infty$ and determine how many polygon vertices are needed for a given accuracy. These estimates follow from the completed boundary theorem.

\DJ okovi\'c used polygons generated by powers of a complex number to simplify the description of the stochastic eigenvalue region \cite{Djokovic}. Dubuc and Malik connected these convex hulls with Farey sequences and trinomial equations, that is, polynomial equations with three nonzero terms \cite{DubucMalik}. Dubuc and Zaoui studied the fractal dimension of unions of the resulting arcs \cite{DubucZaoui}. For points $z$ on a trinomial Farey arc associated with $p/q<r/s$ and $s>q>2$, Dubuc and Malik prove the lower bound
\[
 |z|^s\ge1-\frac{\pi^2}{qs}
\]
\cite[Theorem~13.1]{DubucMalik}. Kirkland, Laffey, and \v{S}migoc later characterized the boundary at a fixed argument \cite{KirklandLaffeySmigoc}. Here the matrix order $N$ varies. The resulting equation gives a leading term for the gap $1-K_N(\theta)$ with relative error $O(N^{-2})$: the error divided by the leading term is $O(N^{-2})$. This term depends on both endpoint denominators and the argument's position between them, and the error bound remains valid arbitrarily close to either endpoint. The estimate refines the lower bound in the trinomial case and also applies to products with several factors.

Let $N\ge4$ and suppose $x=\theta/(2\pi)$ lies strictly between consecutive fractions $f<g$ in $F_N\cap[0,1/2]$. If $f$ has the smaller denominator, take $y=x$ and keep the interval $(f,g)$. Otherwise reflect it to $(1-g,1-f)$ and take $y=1-x$. Write the chosen interval as $(p/q,r/s)$, so its lower endpoint $p/q$ has the smaller denominator:
\[
 \frac pq<y<\frac rs,\qquad q<s,\qquad rq-ps=1.
\]
Set $m=\lfloor N/q\rfloor$, the number of factors in the product. Since $r/s-p/q=1/(qs)$, define
\[
 t=\frac{y-p/q}{r/s-p/q}=s(qy-p)\in(0,1).
\]
Thus $t$ measures the distance from $p/q$ to $y$ as a proportion of the interval length, and
\[
 y-\frac pq=\frac{t}{qs},\qquad
 \frac rs-y=\frac{1-t}{qs}.
\]
We allow $t$ to approach $0$ or $1$ with no restriction on its rate.

\begin{theorem}[Uniform asymptotic estimate]\label{thm:farey-asymptotic}
For the interval data defined above,
\begin{equation}\label{eq:uniform-farey-asymptotic}
 1-K_N(\theta)=
 \frac{2\pi^2t(1-t)}{qs}
 \left(\frac{t}{s}+\frac{1-t}{mq}\right)
 \bigl(1+O(N^{-2})\bigr)
 \qquad(N\to\infty).
\end{equation}
The constant in $O(N^{-2})$ is absolute: the same relative error bound holds for every interval and every $0<t<1$, even arbitrarily close to either endpoint.

In particular, fix an irrational $x\in(0,1/2)$ and apply the preceding setup at $\theta=2\pi x$ for each $N$. Write the chosen interval as $p_N/q_N<y_N<r_N/s_N$, where $y_N=x$ or $1-x$, and put
\[
 m_N=\lfloor N/q_N\rfloor,\qquad
 t_N=s_N(q_Ny_N-p_N).
\]
Then, as $N\to\infty$,
\[
 1-K_N(2\pi x)\sim
 \frac{2\pi^2t_N(1-t_N)}{q_Ns_N}
 \left(\frac{t_N}{s_N}+\frac{1-t_N}{m_Nq_N}\right).
\]
The ratio of the two sides tends to one without any lower bound on $t_N$ or $1-t_N$.
\end{theorem}
\begin{proof}
Since the endpoint fractions are consecutive in $F_N$, we have $q+s>N$. As $q<s$, this gives $s>N/2$. Also, $m=\lfloor N/q\rfloor\ge1$ gives
\[
 N<(m+1)q\le2mq,
\]
so $mq>N/2$. Set
\[
 A=\frac{2\pi t}{s},\qquad
 B=\frac{2\pi(1-t)}{mq},\qquad
 H(u)=u^{s/m}\sin A+u^q\sin B\quad(u\ge0).
\]
Thus $A$, $B$, and $A+B$ are $O(N^{-1})$, uniformly over the interval and $t$. Write $K=K_N(\theta)$. Equation~\eqref{eq:radial} is $H(K)=\sin(A+B)$, so we can estimate $1-K$ by comparing $H(K)$ with $H(1)$.
The excess of $H(1)$ over the target value factors as
\begin{equation}\label{eq:sine-defect}
 \begin{split}
 D:=H(1)-\sin(A+B)
 &=4\sin(A/2)\sin(B/2)\sin((A+B)/2)\\
 &=\frac{AB(A+B)}2\bigl(1+O(N^{-2})\bigr).
 \end{split}
\end{equation}
Indeed, each ratio of a sine to its positive argument is $1+O(N^{-2})$. These estimates remain uniform as the arguments approach zero, so the error is relative even when $t$ approaches $0$ or $1$.
Similarly,
\begin{equation}\label{eq:farey-derivative}
 \begin{split}
 H'(1)&=\frac{s}{m}\sin A+q\sin B\\
 &=\frac{2\pi}{m}
   \left(t\frac{\sin A}{A}+(1-t)\frac{\sin B}{B}\right)
 =\frac{2\pi}{m}\bigl(1+O(N^{-2})\bigr).
 \end{split}
\end{equation}
The coefficients $t$ and $1-t$ are positive and sum to one, so this estimate also stays uniform when either tends to zero.

By the mean-value theorem,
\[
 D=H(1)-H(K)=(1-K)H'(\xi),\qquad K<\xi<1.
\]
To obtain the required relative error for $1-K$, we must replace $H'(\xi)$ by $H'(1)$ with relative error $O(N^{-2})$. The bound from Theorem~\ref{thm:gauges},
\[
 K\ge\cos(\pi/N),
\]
gives $1-K=O(N^{-2})$ and hence $|\log\xi|=O(N^{-2})$. Since $q,s/m\le N$, the powers $\xi^{q-1}$ and $\xi^{s/m-1}$ in $H'(\xi)$ differ from one by $O(N^{-1})$. Both derivative terms have positive coefficients, so
\[
 \frac{H'(\xi)}{H'(1)}=1+O(N^{-1}).
\]
This first estimate is enough to improve the bound on $1-K$. Using $D=O(N^{-3})$, \eqref{eq:farey-derivative}, and $m\le N/q$ in the mean-value identity gives
\[
 1-K=O(mN^{-3})=O\bigl((qN^2)^{-1}\bigr).
\]
Consequently $|\log\xi|=O((qN^2)^{-1})$. Moreover, $N<(m+1)q$ gives the stronger exponent bound
\[
 s/m\le N/m<2q.
\]
Both powers in the derivative now differ from one by $O(N^{-2})$. Their coefficients are positive, so
\[
 \frac{H'(\xi)}{H'(1)}=1+O(N^{-2}).
\]
Combining this estimate with the mean-value identity and \eqref{eq:sine-defect}--\eqref{eq:farey-derivative} yields
\[
 1-K=\frac{mAB(A+B)}{4\pi}\bigl(1+O(N^{-2})\bigr).
\]
Substituting the expressions for $A$ and $B$ gives \eqref{eq:uniform-farey-asymptotic}. Every constant is independent of the interval and $t$, so the estimate also applies to the successive intervals of any fixed irrational angle.
\end{proof}

The relative error statement concerns only open Farey intervals: at an endpoint, both the gap $1-K_N(\theta)$ and the leading term are zero. For an irrational $x=\theta/(2\pi)$, both endpoint denominators and the position within the interval vary with $N$. For badly approximable $x$, the following corollary gives a simpler rate. Recall that $x$ is badly approximable if some $c_x>0$ satisfies
\[
 \left|x-\frac ab\right|\ge\frac{c_x}{b^2}
 \qquad\text{for every rational }a/b\text{ with }b\ge1.
\]
For positive sequences, we write $a_N\asymp_x b_N$ when $a_N/b_N$ is bounded above and below by positive constants depending only on $x$.

\begin{corollary}[Contraction loss and required vertex counts]
\label{cor:complexity-asymptotic}
Let $\rho>0$, $x\in[0,1/2]$, and $N\ge4$. If $x=a/b$ is rational in lowest terms, then $K_N(2\pi x)=1$ exactly when $b\le N$. If $x$ is irrational, then $K_N(2\pi x)<1$ at every finite order.
For badly approximable $x$, both the radius gap and the relative excess of the optimal factor over $\rho$ satisfy
\begin{equation}\label{eq:badly-approximable-loss}
 1-K_N(2\pi x)\asymp_x N^{-3},\qquad
 \frac{\Gamma_N(T_{\rho,2\pi x})}{\rho}-1\asymp_x N^{-3}
 \qquad(N\to\infty).
\end{equation}
For every $N\ge4$, the largest relative excess over all angles is
\begin{equation}\label{eq:uniform-relative-loss}
 \max_{0\le\theta\le\pi}
 \left(\frac{\Gamma_N(T_{\rho,\theta})}{\rho}-1\right)
 =\frac{1}{\cos(\pi/N)}-1\sim\frac{\pi^2}{2N^2}
 \qquad(N\to\infty).
\end{equation}
For $\varepsilon>0$, define $B_x(\varepsilon)$ to be the smallest integer $N\ge4$ for which the relative excess at angle $2\pi x$ is at most $\varepsilon$:
\[
 B_x(\varepsilon)=\min\left\{N\ge4:
   \frac{\Gamma_N(T_{\rho,2\pi x})}{\rho}-1\le\varepsilon\right\}.
\]
For badly approximable $x$,
\[
 B_x(\varepsilon)\asymp_x\varepsilon^{-1/3}
 \qquad(\varepsilon\downarrow0).
\]
If the polygon may depend on the angle but the same vertex bound must give this accuracy for every angle, the smallest such bound is
\begin{equation}\label{eq:uniform-vertex-budget}
 B_{\mathrm{unif}}(\varepsilon)
 =\max\left\{4,
   \left\lceil\frac{\pi}{\arccos((1+\varepsilon)^{-1})}\right\rceil\right\}
 \sim\frac{\pi}{\sqrt{2\varepsilon}}
 \qquad(\varepsilon\downarrow0).
\end{equation}
Both vertex bounds are independent of $\rho$.
\end{corollary}
\begin{proof}
Theorem~\ref{thm:karpelevic} says that the points of $\Theta_N$ on the unit circle are precisely the roots of unity of orders at most $N$. If $x=a/b$ is rational in lowest terms, the least positive integer $k$ for which $kx$ is an integer is $b$. Thus $e^{2\pi ix}$ has order $b$, and belongs to $\Theta_N$ exactly when $b\le N$. Since $K_N(2\pi x)$ is an attained maximum in the closed unit disk, it equals one exactly when $e^{2\pi ix}\in\Theta_N$. This proves the rational assertion. If $x$ is irrational, no positive integer $k$ makes $kx$ an integer, so $e^{2\pi ix}$ is not a root of unity and does not belong to $\Theta_N$. The attained maximum is therefore strictly less than one.

Suppose now that $x$ is badly approximable. Replacing $x$ by $1-x$ preserves the defining inequality with the same constant $c_x$. Applying it to the two endpoints of the chosen Farey interval gives
\[
 \frac{t}{qs}=y-\frac pq\ge\frac{c_x}{q^2},\qquad
 \frac{1-t}{qs}=\frac rs-y\ge\frac{c_x}{s^2}.
\]
Hence
\[
 t\ge c_xs/q,\qquad 1-t\ge c_xq/s.
\]
Since $q<s$ and $t<1$, we obtain $c_xs<q<s$. Thus $q$ and $s$ are comparable, with constants depending only on $x$. The same inequalities give $t\ge c_x$ and $1-t>c_x^2$, so neither endpoint distance, measured by $t$ or $1-t$, can tend to zero.
Together with $q+s>N$, $s\le N$, and $N/2<mq\le N$, this yields
\[
 q\asymp_x s\asymp_x N,\qquad mq\asymp N.
\]
In the leading term of \eqref{eq:uniform-farey-asymptotic}, $t(1-t)$ is bounded above and below by positive constants depending only on $x$, $qs\asymp_x N^2$, and
\[
 \frac{t}{s}+\frac{1-t}{mq}\asymp_x N^{-1}.
\]
This proves the first estimate in \eqref{eq:badly-approximable-loss}. The map $T_{\rho,2\pi x}$ multiplies Euclidean lengths by $\rho$. The difference between the optimal polygon factor and $\rho$, divided by $\rho$, is
\[
 \frac{\Gamma_N(T_{\rho,2\pi x})}{\rho}-1
 =\frac{1-K_N(2\pi x)}{K_N(2\pi x)}.
\]
Since $K_N(2\pi x)\to1$, this quantity has the same $N^{-3}$ bounds as $1-K_N(2\pi x)$, proving the second estimate.

Equation~\eqref{eq:uniform-gauge} gives the largest value of this quantity over all angles, as stated in \eqref{eq:uniform-relative-loss}. The expansion $1/\cos h-1\sim h^2/2$ as $h\to0$ gives its asymptotic.
To achieve accuracy $\varepsilon$ at angle $2\pi x$, we seek a convex polygon $P$ with at most $N$ vertices and zero in its interior such that
\[
 V_P(T_{\rho,2\pi x}v)\le(1+\varepsilon)\rho V_P(v)
 \qquad(v\in\mathbb R^2).
\]
By the definition and attainment of $\Gamma_N$, such a polygon exists exactly when $\Gamma_N(T_{\rho,2\pi x})/\rho-1\le\varepsilon$. The number $B_x(\varepsilon)$ is the smallest integer $N\ge4$ meeting this condition. For irrational $x$, the quantity on the left is positive at every finite order, so $B_x(\varepsilon)\to\infty$ as $\varepsilon\downarrow0$.
For badly approximable $x$, the upper $N^{-3}$ bound ensures the condition once $N$ is at least a sufficiently large constant multiple of $\varepsilon^{-1/3}$. The lower bound excludes values below a positive constant multiple of $\varepsilon^{-1/3}$. This proves the stated estimate for $B_x$.
For the same accuracy at every angle, we use the largest value over angles and solve
\[
 \frac{1}{\cos(\pi/N)}-1\le\varepsilon
 \quad\Longleftrightarrow\quad
 N\ge\frac{\pi}{\arccos((1+\varepsilon)^{-1})}.
\]
Taking the smallest integer $N\ge4$ gives \eqref{eq:uniform-vertex-budget}. The expansion $\arccos((1+\varepsilon)^{-1})\sim\sqrt{2\varepsilon}$ gives its asymptotic. Finally, $\Gamma_N(T_{\rho,\theta})/\rho=1/K_N(\theta)$ shows that neither vertex bound depends on $\rho$.
\end{proof}


\clearpage
\appendix
\section{Near equality and products with different exponents}
\label{app:comparisons}

Theorem~\ref{thm:convex-product} uses Jensen's inequality to bound the modulus, with equality when all coefficients $\beta_j$ agree. We now quantify this equality condition: a small nonnegative difference between the two sides of Jensen's inequality bounds how far the factor arguments, and hence the coefficients, can differ. We then allow factors with different exponents. Their arguments have different allowed ranges, so we optimize the product by choosing these arguments within their individual ranges while keeping their sum fixed.

\subsection{Bounds on differences between the coefficients}

\begin{proposition}\label{prop:quantitative-rigidity}
Assume the hypotheses of Theorem~\ref{thm:convex-product}, including both the product relation and the real argument identity, with $\omega=\rho e^{i\vartheta}$. Recall the factor arguments $u_j$ and the argument $M$ approached when a coefficient tends to one:
\[
 u_j=\operatorname{Arg}(\omega^q-\beta_j),\qquad
 M=\operatorname{Arg}(\omega^q-1).
\]
The real argument identity gives their average as
\[
 \bar u=\frac1m\sum_{j=1}^m u_j=A+B.
\]
Define
\[
 F(u)=\log\sin M-\log\sin(M-u)\qquad(0<u<M),
 \qquad J=\sum_{j=1}^mF(u_j)-mF(\bar u).
\]
Thus $J$ is the nonnegative difference between the two sides of Jensen's inequality. Define also the nonnegative difference in the theorem's bound on the modulus:
\[
 D_\rho=\sin(A+B)-\rho^{s/m}\sin A-\rho^q\sin B.
\]
Then
\begin{equation}\label{eq:jensen-gap}
 J\ge\frac12\sum_{j=1}^m(u_j-\bar u)^2,
 \qquad
 D_\rho=\rho^{s/m}\sin A\bigl(e^{J/m}-1\bigr).
\end{equation}
Consequently, the sum of squared deviations of the arguments from their average satisfies
\begin{equation}\label{eq:argument-variance-bound}
 \sum_{j=1}^m(u_j-\bar u)^2
 \le2m\log\left(1+\frac{D_\rho}{\rho^{s/m}\sin A}\right)
 \le\frac{2mD_\rho}{\rho^{s/m}\sin A}.
\end{equation}
The coefficients satisfy a corresponding bound. Put
\[
 \bar\beta=\frac1m\sum_{j=1}^m\beta_j,
 \qquad c_\rho=\frac{\rho^q\sin A}{(1+\rho^q)^2}>0.
\]
Then
\begin{equation}\label{eq:coefficient-variance-bound}
 \sum_{j=1}^m(\beta_j-\bar\beta)^2
 \le c_\rho^{-2}\sum_{j=1}^m(u_j-\bar u)^2.
\end{equation}
Let $\rho_*$ be the unique solution of \eqref{eq:radial} for these interval data. With the Farey endpoints, factor count and interior angle fixed, both displayed sums of squared deviations are $O(\rho_*-\rho)$ as $\rho\to\rho_*$ from below, provided the coefficients continue to satisfy the product relation and real argument identity. The constants are uniform as the angle varies over a compact subset of the open Farey interval and $\rho$ stays bounded away from zero.
\end{proposition}
\begin{proof}
The real argument identity \eqref{eq:analytic-product-phase} in Theorem~\ref{thm:convex-product} is
\[
 (s-mq)\vartheta+\sum_{j=1}^m u_j=2\pi(r-mp).
\]
With $A=q\vartheta-2\pi p$ and $B=(2\pi r-s\vartheta)/m$, it gives
\[
 \sum_{j=1}^m u_j
 =2\pi(r-mp)-(s-mq)\vartheta=m(A+B)=m\bar u.
\]
Hence $\sum_j(u_j-\bar u)=0$.
All $u_j$ and their average lie in $(0,M)$, where \eqref{eq:one-factor-convexity} gives $F''\ge1$. Taylor's formula with integral remainder yields
\[
 \begin{aligned}
 F(u_j)-F(\bar u)-F'(\bar u)(u_j-\bar u)
 &=(u_j-\bar u)^2\int_0^1(1-t)
 F''\bigl(\bar u+t(u_j-\bar u)\bigr)\,dt\\
 &\ge\tfrac12(u_j-\bar u)^2.
 \end{aligned}
\]
The last inequality follows from $F''\ge1$ and $\int_0^1(1-t)\,dt=1/2$. Summing cancels the linear terms and proves $J\ge\tfrac12\sum_j(u_j-\bar u)^2$.

To connect $J$ with $D_\rho$, put
\[
 g_j=\frac{\omega^q-\beta_j}{1-\beta_j},\qquad
 c=\frac{1-\rho^q\cos A}{\rho^q\sin A}=-\cot M.
\]
Since $\omega^q=\rho^qe^{iA}$, direct substitution gives $\Re g_j+c\Im g_j=1$. Writing $g_j=|g_j|e^{iu_j}$ therefore gives $|g_j|\mathcal D(u_j)=1$, where
\[
 \mathcal D(u)=\cos u+c\sin u
 =\frac{\sin(M-u)}{\sin M}=e^{-F(u)}.
\]
Thus $F(u_j)=\log|g_j|$. Dividing the product relation in \eqref{eq:analytic-product-phase} by $\prod_j(1-\beta_j)$ gives $\omega^{s-mq}\prod_jg_j=1$. Taking logarithms of moduli yields
\[
 \sum_jF(u_j)=(mq-s)\log\rho.
\]
Using $\bar u=A+B$ in the definition of $J$ now gives $e^{J/m}=\rho^{q-s/m}\mathcal D(A+B)$. To evaluate this, substitute the displayed value of $c$:
\[
 \begin{aligned}
 \mathcal D(A+B)
 &=\cos(A+B)+\frac{1-\rho^q\cos A}{\rho^q\sin A}\sin(A+B)\\
 &=\frac{\sin(A+B)-\rho^q\sin B}{\rho^q\sin A}.
 \end{aligned}
\]
Here $\sin A\cos(A+B)-\cos A\sin(A+B)=-\sin B$. Consequently,
\[
 e^{J/m}=\frac{\sin(A+B)-\rho^q\sin B}{\rho^{s/m}\sin A}.
\]
Rearranging proves the identity for $D_\rho$ in \eqref{eq:jensen-gap} and gives
\[
 J=m\log\left(1+\frac{D_\rho}{\rho^{s/m}\sin A}\right).
\]
Together with the lower bound on $J$ and $\log(1+x)\le x$ for $x\ge0$, this proves \eqref{eq:argument-variance-bound}.

We have bounded the differences between the factor arguments. To bound the differences between the coefficients, we use how each argument changes with its coefficient. Define
\[
 u(\beta)=\operatorname{Arg}(\rho^qe^{iA}-\beta),\qquad0\le\beta<1,
\]
so that $u_j=u(\beta_j)$. Equation~\eqref{eq:one-factor-convexity} gives
\[
 \frac{du}{d\beta}
 =\frac{\rho^q\sin A}{|\rho^qe^{iA}-\beta|^2}
 \ge\frac{\rho^q\sin A}{(1+\rho^q)^2}=c_\rho>0,
\]
because $|\rho^qe^{iA}-\beta|\le1+\rho^q$. Thus changing a coefficient by $|\beta_i-\beta_j|$ changes its argument by at least $c_\rho|\beta_i-\beta_j|$. More precisely, the mean-value theorem gives
\[
 |u_i-u_j|\ge c_\rho|\beta_i-\beta_j|.
\]
For real numbers $x_1,\ldots,x_m$ with average $\bar x=m^{-1}\sum_jx_j$,
\[
 \sum_{j=1}^m(x_j-\bar x)^2
 =\frac1m\sum_{i<j}(x_i-x_j)^2.
\]
Using this identity for the coefficients and arguments gives
\[
 \begin{aligned}
 \sum_j(\beta_j-\bar\beta)^2
 &=\frac1m\sum_{i<j}(\beta_i-\beta_j)^2\\
 &\le\frac{c_\rho^{-2}}m\sum_{i<j}(u_i-u_j)^2
 =c_\rho^{-2}\sum_j(u_j-\bar u)^2,
 \end{aligned}
\]
which is \eqref{eq:coefficient-variance-bound}.

Finally, put $H(t)=t^{s/m}\sin A+t^q\sin B$. Since $H(\rho_*)=\sin(A+B)$,
\[
 D_\rho=H(\rho_*)-H(\rho)
 =H'(\rho_*)(\rho_*-\rho)+o(\rho_*-\rho).
\]
Thus $D_\rho=O(\rho_*-\rho)$. Substituting this into the two bounds on squared deviations proves the stated rate. On compact subsets of the open angle interval, with $\rho$ bounded away from zero, $\rho^{s/m}\sin A$ and $c_\rho$ have positive lower bounds and $H'$ is uniformly bounded. The same estimates therefore hold with uniform constants.
\end{proof}
These estimates require both the product relation and the argument identity as an equality of real numbers, not merely modulo $2\pi$. They control the factor arguments and coefficients in these relations. They give no bound on the distance of an arbitrary stochastic matrix from matrices realizing the boundary, even if one of its eigenvalues is close to the boundary.

\subsection{Products with different exponents}

The exponent may now differ from one factor to another. For factor $j$, the argument $u_j$ increases from $A_j$ towards $M_j$ as its coefficient increases from zero towards one. Write $v_j=M_j-u_j$ for the remaining difference, whose largest allowed value is $U_j=M_j-A_j$. For $z=\rho e^{i\theta}$ and the integer exponent $e$ in the product below, the real argument identity
\[
 e\theta+\sum_{j=1}^m u_j=2\pi L,\qquad L\in\mathbb Z,
\]
fixes the sum of these differences as $\sum_jv_j=\sum_jM_j+e\theta-2\pi L$. To maximize $\prod_j\sin v_j$ with this fixed sum, the optimal choice makes the $v_j$ equal except where that value would exceed an individual upper bound $U_j$; there we take $v_j=U_j$.

\begin{theorem}[Product bound with different exponents]\label{thm:unequal-product-bound}
Let $m\ge1$, let $q_1,\ldots,q_m$ be positive integers, and let $e\in\mathbb Z$. Write $z=\rho e^{i\theta}$ with $\rho>0$ and $\theta\in\mathbb R$. For $1\le j\le m$, assume
\[
 A_j=\operatorname{Arg}(z^{q_j})\in(0,\pi),\qquad0\le\beta_j<1,
\]
and define
\[
 M_j=\operatorname{Arg}(z^{q_j}-1),\qquad
 u_j=\operatorname{Arg}(z^{q_j}-\beta_j),\qquad
 v_j=M_j-u_j,\qquad U_j=M_j-A_j.
\]
Assume both the product relation and the following real argument identity, for an integer $L$:
\begin{equation}\label{eq:unequal-product-phase}
 \begin{gathered}
 z^e\prod_{j=1}^m(z^{q_j}-\beta_j)=\prod_{j=1}^m(1-\beta_j),\\
 e\theta+\sum_{j=1}^m u_j=2\pi L.
 \end{gathered}
\end{equation}
The differences $v_j$ satisfy $0<v_j\le U_j<\pi$, and their sum is fixed by the argument identity:
\[
 S:=\sum_{j=1}^mM_j+e\theta-2\pi L=\sum_{j=1}^m v_j,
 \qquad0<S\le\sum_{j=1}^mU_j.
\]
If $S<\sum_jU_j$, let $c\in(0,\max_jU_j)$ be the unique solution of
\[
 \sum_{j=1}^m\min(c,U_j)=S.
\]
If $S=\sum_jU_j$, set $c=\max_jU_j$. Put
\[
 v_j^*=\min(c,U_j).
\]
Thus each $v_j^*$ equals the common value $c$ unless its upper bound $U_j$ is smaller, in which case it equals $U_j$.
These values uniquely maximize $\prod_j\sin v_j$ among all choices satisfying $0<v_j\le U_j$ and $\sum_jv_j=S$. Consequently,
\begin{equation}\label{eq:unequal-sharp-product}
 \rho^e\prod_{j=1}^m\sin M_j
   =\prod_{j=1}^m\sin v_j
   \le\prod_{j=1}^m\sin v_j^*,
\end{equation}
with equality exactly when $v_j=v_j^*$ for every $j$. The logarithmic difference from this maximum also bounds the squared deviations from the maximizing values:
\begin{equation}\label{eq:unequal-quadratic-deficit}
 \log\frac{\prod_j\sin v_j^*}{\rho^e\prod_j\sin M_j}
       \ge\frac12\sum_{j=1}^m(v_j-v_j^*)^2.
\end{equation}
If $S/m\le\min_jU_j$, the common value $S/m$ lies within every allowed range, and the bound reduces to
\[
 \rho^e\prod_{j=1}^m\sin M_j\le\sin^m(S/m),
\]
with equality exactly when $v_1=\cdots=v_m=S/m$.
\end{theorem}
\begin{proof}
Put $w_j=z^{q_j}$. Since $\Im w_j>0$,
\[
 \frac{d}{d\beta}\operatorname{Arg}(w_j-\beta)
 =\frac{\Im w_j}{|w_j-\beta|^2}>0.
\]
Thus the argument increases from $A_j$ towards $M_j$ as $\beta$ increases from zero towards one. For $0\le\beta_j<1$, this gives $A_j\le u_j<M_j$ and hence $0<v_j\le U_j<\pi$.

Define
\[
 g_j=\frac{w_j-\beta_j}{1-\beta_j},\qquad
 c_j=\frac{1-\Re w_j}{\Im w_j}=-\cot M_j.
\]
The coefficient $c_j$ may have either sign. Direct substitution gives
\[
 \Re g_j+c_j\Im g_j
 =\frac{\Re w_j-\beta_j+c_j\Im w_j}{1-\beta_j}=1.
\]
The denominator $1-\beta_j$ is positive, so $\operatorname{Arg}g_j=u_j$. Writing $g_j=|g_j|e^{iu_j}$ in this line equation gives
\[
 |g_j|=\frac1{\cos u_j+c_j\sin u_j}
 =\frac{\sin M_j}{\sin(M_j-u_j)}
 =\frac{\sin M_j}{\sin v_j}.
\]
Taking moduli in the product relation \eqref{eq:unequal-product-phase} gives $\rho^e\prod_j|g_j|=1$, and therefore
\[
 \rho^e\prod_j\sin M_j=\prod_j\sin v_j,
\]
which is the equality in \eqref{eq:unequal-sharp-product}. The real argument identity $e\theta+\sum_j u_j=2\pi L$ gives
\[
 \sum_jv_j=\sum_jM_j-\sum_ju_j
 =\sum_jM_j+e\theta-2\pi L=S.
\]
Together with the bounds on $v_j$, this proves $0<S\le\sum_jU_j$.

We now maximize the sine product over all $v=(v_1,\ldots,v_m)$ satisfying $0<v_j\le U_j$ and $\sum_jv_j=S$. If $S=\sum_jU_j$, these conditions force $v_j=U_j=v_j^*$ for every $j$, so both bounds and their equality conditions follow directly. Suppose $S<\sum_jU_j$. The function $c\mapsto\sum_j\min(c,U_j)$ increases continuously from zero to $\sum_jU_j$ and is strictly increasing for $0\le c<\max_jU_j$. There is therefore a unique $c$ as stated, and $v_j^*=\min(c,U_j)$ satisfies all the constraints.

Set $\Phi(v)=\sum_j\log\sin v_j$ and $\Delta_j=v_j-v_j^*$. Since $(\log\sin t)'=\cot t$ and $(\log\sin t)''=-1/\sin^2t$, Taylor's formula with integral remainder gives
\[
 \begin{aligned}
 \Phi(v)-\Phi(v^*)
 &=\sum_j\cot(v_j^*)\Delta_j
 -\sum_j\Delta_j^2\int_0^1
 \frac{1-t}{\sin^2(v_j^*+t\Delta_j)}\,dt\\
 &\le\sum_j\cot(v_j^*)\Delta_j
 -\frac12\sum_j\Delta_j^2.
 \end{aligned}
\]
Each segment $v_j^*+t\Delta_j$ lies in $(0,\pi)$. The inequality follows from $1/\sin^2t\ge1$ and $\int_0^1(1-t)\,dt=1/2$.
Because $v$ and $v^*$ both have sum $S$, $\sum_j\Delta_j=0$. Subtracting $\cot c\sum_j\Delta_j$ therefore leaves the linear term unchanged and gives
\[
 \begin{aligned}
 \sum_j\cot(v_j^*)\Delta_j
 &=\sum_j\bigl(\cot(v_j^*)-\cot c\bigr)\Delta_j\\
 &=\sum_{U_j<c}(\cot U_j-\cot c)(v_j-U_j)\le0.
 \end{aligned}
\]
Only indices with $U_j<c$ contribute, since otherwise $v_j^*=c$. For each contributing index, $\cot U_j-\cot c>0$ because cotangent decreases on $(0,\pi)$, whereas $v_j-U_j\le0$ by the upper bound. Consequently,
\[
 \Phi(v^*)-\Phi(v)\ge\frac12\sum_j(v_j-v_j^*)^2.
\]
This proves the unique maximum: every $v\ne v^*$ gives a strictly smaller value of $\Phi$, and hence of the sine product. For the original factor arguments, the product identity proved above turns the left-hand side into
\[
 \log\frac{\prod_j\sin v_j^*}{\rho^e\prod_j\sin M_j},
\]
proving \eqref{eq:unequal-quadratic-deficit}. Exponentiation gives the product bound and its equality condition. Finally, if $S/m\le\min_jU_j$, the choice $c=S/m$ gives $v_j^*=S/m$ for every $j$, yielding the last assertion.
\end{proof}

To recover the equal-exponent case of Theorem~\ref{thm:convex-product}, take $q_j=q$, $z=\omega$, $\theta=\vartheta$, $e=s-mq$ and $L=r-mp$. Then $M_j=M$, $U_j=M-A$, and the real argument identity gives
\[
 S=m(M-A-B).
\]
Since $B>0$, the common value $S/m=M-A-B$ is smaller than every upper bound $U_j=M-A$. Thus equality requires $v_j=M-A-B$, or equivalently $u_j=A+B$, for every $j$. The strict increase of $\operatorname{Arg}(\omega^q-\beta)$ with $\beta$ makes this equivalent to equal coefficients. With different exponents, it is the differences $M_j-u_j$ that share the common maximizing value $c$, except where an upper bound $U_j$ forces a smaller value. The theorem bounds the product under the stated product and real argument identities. To determine the largest eigenvalue modulus at a fixed argument and matrix order, one would also need a stochastic matrix of that order attaining the proposed bound and a proof that no matrix of that order has an eigenvalue at the same argument with a larger modulus.


\section{Matrices with independently varying parameters}\label{app:matrix-family}
Theorem~\ref{thm:sparse-realization} constructs a row-stochastic family $A(\boldsymbol\beta)$, with $\boldsymbol\beta=(\beta_0,\ldots,\beta_{m-1})\in[0,1]^m$, and gives its full characteristic polynomial. Its directed graph is illustrated in Figure~\ref{fig:realizing-graphs}. Here we count the positive entries, show when every vertex can be reached from every other, and determine when a matrix power has all entries positive. We then examine the choices $\beta_j=0$ and $\beta_j=1$ for all $j$. Finally, for the order-$n$ Farey data and $0\le\beta_j<1$, consider a root $\omega=\rho e^{i\vartheta}$ with $0<\rho<1$ and argument inside the chosen interval. If it satisfies the real argument identity \eqref{eq:analytic-product-phase}, we show that $\rho\le K_n(\theta)$, where $\theta=\min\{\vartheta,2\pi-\vartheta\}$. Its modulus reaches $K_n(\theta)$ exactly when all parameters $\beta_j$ are equal.

\begin{proposition}\label{prop:family-structure}
Let $A(\boldsymbol\beta)$ be the family constructed in Theorem~\ref{thm:sparse-realization}, with positive integers $m,q,s$ satisfying $s>(m-1)q$ and $q<s$. Put $n_0=\max\{mq,s\}$.
If $0<\beta_j<1$ for every $j=0,\ldots,m-1$, then $A(\boldsymbol\beta)$ is irreducible: every vertex can reach every other along a directed path whose arrow weights are positive. It has exactly $n_0+m$ positive entries: $m$ rows have two each, and the remaining $n_0-m$ rows have one each.
If also $\gcd(q,s)=1$, the matrix is primitive: some positive integer $k$ makes every entry of $A(\boldsymbol\beta)^k$ strictly positive.

For any $\boldsymbol\beta\in[0,1]^m$ and integer $n\ge n_0$, the block diagonal matrix
\[
 A(\boldsymbol\beta)\oplus I_{n-n_0}
 =\begin{pmatrix}
 A(\boldsymbol\beta)&0\\
 0&I_{n-n_0}
 \end{pmatrix}
\]
is row-stochastic of order $n$ and retains every eigenvalue of $A(\boldsymbol\beta)$. Here $I_{n-n_0}$ is the identity matrix of order $n-n_0$; when $n=n_0$, no block is added.
\end{proposition}
\begin{proof}
Suppose $0<\beta_j<1$ for every $j$, so both $\beta_j$ and $1-\beta_j$ are positive. From any vertex in a block, the arrows of weight one lead to its last vertex. The arrow of weight $\beta_j$ leads back to the first vertex, allowing every vertex of that block to be reached. The arrows of weight $1-\beta_j$ lead through all blocks in cyclic order, including any vertices inserted between the last block and block zero. Thus every vertex can reach every other by arrows of positive weight, proving irreducibility.

At the last vertex of each block, the two outgoing arrows have distinct destinations. For $m\ge2$, the arrow of weight $\beta_j$ stays in its block, while the other leads to a different block or an inserted vertex. For $m=1$, the inequality $q<s$ requires inserted vertices, so the destinations are again distinct. These $m$ rows therefore have two positive entries each. All other $n_0-m$ rows have one, giving $2m+(n_0-m)=n_0+m$ positive entries.

Now assume $\gcd(q,s)=1$. The vertex $v_{0,q-1}$ lies on both the cycle of length $q$ within block zero and the cycle of length $s$ through all blocks. Every integer $L\ge(q-1)s$ can be written as $L=aq+bs$ with nonnegative integers $a,b$. Indeed, coprimality allows a choice $b\in\{0,\ldots,q-1\}$ with $bs\equiv L\pmod q$; then $a=(L-bs)/q$ is an integer and is nonnegative because $bs\le(q-1)s\le L$.

For any vertices $u,v$, choose paths from $u$ to $v_{0,q-1}$ and from there to $v$, and let their total length be $\ell_{uv}$. Repeating the two cycles $a$ and $b$ times between these paths gives a walk from $u$ to $v$ of length $\ell_{uv}+aq+bs$; a walk may revisit vertices. Thus there is such a walk of every length $k\ge\ell_{uv}+(q-1)s$. Its arrow weights are positive, so their product contributes a positive term to $(A^k)_{uv}$. There are finitely many vertex pairs, so one common threshold works for all of them. Every sufficiently high power of $A$ therefore has all entries positive, proving primitivity.

For every parameter tuple in $[0,1]^m$, the construction gives a row-stochastic matrix $A$. Adding an identity block preserves nonnegative entries and row sums of one. Moreover,
\[
 \det\bigl(tI_n-(A\oplus I_{n-n_0})\bigr)
 =(t-1)^{n-n_0}\det(tI_{n_0}-A),
\]
so every eigenvalue of $A$ is retained, including when $n=n_0$ and no block is added.
\end{proof}

Let $p/q<r/s$ be consecutive fractions in $F_n$, with $q<s$, and put $m=\lfloor n/q\rfloor$. Farey adjacency gives $q+s>n$ and $rq-ps=1$. Since $mq\le n$ and $q,s\le n$,
\[
 s>n-q\ge(m-1)q,\qquad n_0=\max\{mq,s\}\le n.
\]
Thus the realization theorem applies and its matrices can be extended to order $n$. The identity $rq-ps=1$ implies $\gcd(q,s)=1$, so the preceding proposition also gives primitivity whenever every $\beta_j$ lies in $(0,1)$.

If all $\beta_j=\beta$ and $\alpha=1-\beta$, substituting into \eqref{eq:independent-characteristic} gives
\[
 \det(tI-A)=
 \begin{cases}
 (t^q-\beta)^m-\alpha^m t^{mq-s},&s\le mq,\\
 t^{s-mq}(t^q-\beta)^m-\alpha^m,&s>mq.
 \end{cases}
\]
For $\boldsymbol\beta=\mathbf1=(1,\ldots,1)$ and $\boldsymbol\beta=\mathbf0=(0,\ldots,0)$, respectively, this becomes
\[
 \begin{aligned}
 \det(tI-A(\mathbf1))&=t^{n_0-mq}(t^q-1)^m,\\
 \det(tI-A(\mathbf0))&=t^{n_0-s}(t^s-1).
 \end{aligned}
\]
The nonzero eigenvalues are therefore precisely the $q$th roots of unity for $A(\mathbf1)$ and the $s$th roots of unity for $A(\mathbf0)$. The powers of $t$ give the multiplicities of zero. Since the entries vary continuously with $\boldsymbol\beta$, these endpoint matrices belong to the same continuous family. The characteristic formula remains valid when some parameters are zero or one and others lie in $(0,1)$. The irreducibility assertion above requires every parameter to lie in $(0,1)$.

The following corollary identifies the parameter choices whose eigenvalue modulus reaches $K_n(\theta)$. For $0\le\beta_j<1$, the characteristic equation \eqref{eq:independent-characteristic} determines the sum of arguments only modulo $2\pi$. We therefore also require the real argument identity \eqref{eq:analytic-product-phase}, with the specific value $2\pi(r-mp)$ on its right-hand side.

\begin{corollary}\label{cor:family-equality}
Let $n\ge4$, and let $f<g$ be consecutive fractions in $F_n^+$. If the smaller denominator belongs to $f$, put $p/q=f$ and $r/s=g$. Otherwise put $p/q=1-g$ and $r/s=1-f$, corresponding to reflection of the arc in the real axis. In either case, $p/q<r/s$ and $q<s$.
Put $m=\lfloor n/q\rfloor$ and let
\[
 \boldsymbol\beta=(\beta_0,\ldots,\beta_{m-1})\in[0,1)^m.
\]
Let $\omega=\rho e^{i\vartheta}$, with $0<\rho<1$, be a root of the characteristic polynomial \eqref{eq:independent-characteristic}, with
\[
 \frac pq<\frac{\vartheta}{2\pi}<\frac rs.
\]
Define
\[
 A=q\vartheta-2\pi p,\qquad
 B=\frac{2\pi r-s\vartheta}{m},\qquad
 \theta=\min\{\vartheta,2\pi-\vartheta\},
\]
and use arguments in $(0,\pi)$ to set
\[
 M=\operatorname{Arg}(\omega^q-1),\qquad
 u_j=\operatorname{Arg}(\omega^q-\beta_j)\quad(0\le j<m).
\]
These satisfy $A\le u_j<M$. Assume the following identity holds as an equality of real numbers:
\[
 (s-mq)\vartheta+\sum_{j=0}^{m-1}u_j=2\pi(r-mp).
\]
Then
\[
 \rho\le K_n(\theta),
\]
and $\rho=K_n(\theta)$ if and only if $\beta_0=\cdots=\beta_{m-1}$. In that case, their common value is uniquely determined by
\[
 \beta_* =\frac{K_n(\theta)^q\sin B}{\sin(A+B)}.
\]
\end{corollary}
\begin{proof}
Since $\omega\ne0$, dividing the characteristic equation \eqref{eq:independent-characteristic} at $t=\omega$ by $\omega^{n_0-s}$ gives
\[
 \omega^{s-mq}\prod_{j=0}^{m-1}(\omega^q-\beta_j)
 =\prod_{j=0}^{m-1}(1-\beta_j).
\]
The angle identities \eqref{eq:angle-budget} give $A,B>0$ and $A+B<\pi$. Together with the assumed real argument identity, these verify the hypotheses of Theorem~\ref{thm:convex-product}. It yields
\[
 \rho^{s/m}\sin A+\rho^q\sin B\le\sin(A+B),
\]
with equality exactly when all $\beta_j$ agree. The left-hand side is strictly increasing in $\rho$ and equals the right-hand side at $\rho=K_n(\theta)$ by \eqref{eq:radial}. Hence $\rho\le K_n(\theta)$, with $\rho=K_n(\theta)$ exactly when all parameters agree.

In that case, write $\beta$ for their common value and $u$ for their common factor argument. The assumed real identity gives
\[
 u=\frac{2\pi(r-mp)-(s-mq)\vartheta}{m}=A+B.
\]
Put $R=|\omega^q-\beta|>0$, so $\omega^q-\beta=Re^{i(A+B)}$. Taking real and imaginary parts, as in \eqref{eq:factor-resolution}, gives
\[
 \begin{aligned}
 R&=\frac{\rho^q\sin A}{\sin(A+B)},\\
 \beta&=\rho^q\cos A-R\cos(A+B)
 =\frac{\rho^q\sin B}{\sin(A+B)}.
 \end{aligned}
\]
The last equality uses $\cos A\sin(A+B)-\sin A\cos(A+B)=\sin B$. Substituting $\rho=K_n(\theta)$ gives the unique value $\beta_*$. Conversely, \eqref{eq:equal-parameters}, \eqref{eq:factor-resolution} and \eqref{eq:ito-attaining}, evaluated at $\rho=K_n(\theta)$, give this tuple with coefficients in $(0,1)$ satisfying both the product relation and the real argument identity. Thus it attains the stated modulus.
\end{proof}

Let $\vartheta$ vary inside $(2\pi p/q,2\pi r/s)$, taking every parameter equal to $\beta_*(\vartheta)$. At the left endpoint, $K_n(\theta)\to1$, $A\to0$ and $B\to2\pi/(mq)$. At the right endpoint, $K_n(\theta)\to1$, $B\to0$ and $A\to2\pi/s$, as given by \eqref{eq:angle-budget}. Since $mq,s\ge3$, these nonzero limiting angles lie in $(0,\pi)$. Substitution into the formula for $\beta_*$ therefore gives
\[
 \beta_*(\vartheta)\longrightarrow
 \begin{cases}
 1,&\vartheta\downarrow2\pi p/q,\\
 0,&\vartheta\uparrow2\pi r/s.
 \end{cases}
\]
Every matrix entry depends continuously on the parameters, so the realizing matrices satisfy
\[
 A\bigl(\beta_*(\vartheta)\mathbf1\bigr)\longrightarrow
 \begin{cases}
 A(\mathbf1),&\vartheta\downarrow2\pi p/q,\\
 A(\mathbf0),&\vartheta\uparrow2\pi r/s.
 \end{cases}
\]


\section*{Funding}
Vincent Ginis acknowledges support from the Research Foundation--Flanders
(FWO) under grants G032822N and G0K9322N.

\section*{Use of generative artificial intelligence}
The manuscript was developed with OpenAI's ChatGPT 5.5 and Anthropic's Claude Fable 5 as research companions. These systems were used to explore proof strategies, alternative approaches, and connections between ideas. OpenAI's Codex 6.1 was used to rewrite, clarify, reorganize, and refine the exposition. Codex also assisted with diagrams and computational checks.



\end{document}